\documentclass[12pt,a4paper,reqno,twoside]{amsart}

\usepackage[T1]{fontenc}
\usepackage[utf8]{inputenc}
\usepackage[scaled=.95,helvratio=.96,tighter,theoremfont]{newtxtext}
\usepackage{makecell}
\usepackage[12pt]{moresize}
\usepackage[margin=1.5in]{geometry}

\usepackage{amsmath,amsthm}
\usepackage{amssymb}
\usepackage{mathrsfs}
\usepackage{graphicx}
\usepackage{color,tikz,siunitx,xfrac,dsfont,bropd,etoolbox,xparse,pgffor,longtable,booktabs,makecell,array,placeins,xcolor,array,microtype,multirow,tabularx}
\usepackage[shortlabels]{enumitem}
\setlist{
    font=\upshape,
    itemsep=0.25\baselineskip,
    leftmargin=\parindent,
    parsep=0pt,
    topsep=0.25\baselineskip
}
\usetikzlibrary{arrows.meta,calc,decorations.pathreplacing,positioning}
\definecolor{freeviolet}{RGB}{0,0,255}
\definecolor{contactred}{RGB}{0,0,0}
\definecolor{phasegreen}{RGB}{255,255,255}

\newtheoremstyle{plain}{6.5pt}{6.5pt}{\itshape}{}{\bfseries}{.}{.5em}{}
\newtheoremstyle{definition}{6.5pt}{6.5pt}{\normalfont}{}{\bfseries}{.}{.5em}{}

\theoremstyle{plain}
\newtheorem{theorem}{Theorem}[section]
\newtheorem{lemma}[theorem]{Lemma}
\newtheorem{corollary}[theorem]{Corollary}
\newtheorem{proposition}[theorem]{Proposition}

\theoremstyle{definition}
\newtheorem{definition}[theorem]{Definition}
\newtheorem{example}[theorem]{Example}
\newtheorem{remark}[theorem]{Remark}

\theoremstyle{plain}
\newtheorem{innercustomgeneric}{\customgenericname}
\providecommand{\customgenericname}{}
\newcommand{\newcustomtheorem}[2]{%
    \newenvironment{#1}[1]
    {%
        \renewcommand\customgenericname{#2}%
        \renewcommand\theinnercustomgeneric{##1}%
        \innercustomgeneric
    }
    {\endinnercustomgeneric}
}
\newcustomtheorem{maintheorem}{Theorem}

\numberwithin{equation}{section}

\usepackage[pdfa,colorlinks,allcolors=blue]{hyperref}

\usepackage[pdfa, colorlinks, allcolors=blue]{hyperref}

\makeatletter

\let\keywords\relax
\let\msc\relax
\let\abstract\relax

\gdef\@keywords{}
\gdef\@msc{}

\newcommand{\keywords}[1]{%
    \gdef\@keywords{\textit{Keywords:} #1.}%
}

\newcommand{\msc}[2][]{%
    \gdef\@msc{%
        \textit{Mathematics Subject Classification 2020:} #2%
        \ifx&#1&\else\ (#1)\fi.%
    }%
}

\newcommand{\ems@titlefootnote}[1]{%
    \insert\footins{%
        \reset@font\footnotesize
        \interlinepenalty\interfootnotelinepenalty
        \splittopskip\footnotesep
        \splitmaxdepth\dp\strutbox
        \floatingpenalty\@MM
        \hsize\columnwidth
        \@parboxrestore
        \color@begingroup
            \noindent\strut #1\par
        \color@endgroup
    }%
}

\newbox\emsabstractbox
\newenvironment{abstract}{%
    \global\setbox\emsabstractbox\vbox\bgroup
        \normalcolor\small
        \noindent\strut\textbf{Abstract.}\enskip\ignorespaces
}{%
        \unskip\strut\par
    \egroup
}

\renewcommand{\maketitle}{%
    \thispagestyle{empty}
    \begingroup
        \vspace*{26pt}
        \centering
        {\Large\bfseries\boldmath \@title \par}
        \vspace{8mm}
        {\large \@author \par}
        \vspace{5mm}
    \endgroup
    \ifx\@msc\@empty
    \else
        \ems@titlefootnote{\@msc}%
    \fi
    \ifx\@keywords\@empty
    \else
        \ems@titlefootnote{\@keywords}%
    \fi
    \ifvoid\emsabstractbox
    \else
        \noindent\unvbox\emsabstractbox\par
    \fi
    \vspace{26pt}
}

\newcounter{new@auth}
\gdef\@author{}

\newcommand{\Author}[3]{%
    \stepcounter{new@auth}%
    \begingroup
        \def\givenname##1{\expandafter\gdef\csname auth@#1@first\endcsname{##1}}%
        \def\surname##1{\expandafter\gdef\csname auth@#1@last\endcsname{##1}}%
        \def\mrid##1{\expandafter\gdef\csname auth@#1@mrid\endcsname{##1}}%
        \def\zblid##1{\expandafter\gdef\csname auth@#1@zblid\endcsname{##1}}%
        \def\orcid##1{\expandafter\gdef\csname auth@#1@orcid\endcsname{##1}}%
        #2%
    \endgroup
    \ifnum\value{new@auth}=1
        \expandafter\gdef\expandafter\@author\expandafter{%
            \csname auth@#1@first\endcsname\ %
            \csname auth@#1@last\endcsname
        }%
    \else
        \expandafter\g@addto@macro\expandafter\@author\expandafter{%
            , \csname auth@#1@first\endcsname\ %
            \csname auth@#1@last\endcsname
        }%
    \fi
}

\def\new@setmeta#1#2#3{%
    \expandafter\gdef\csname affil@#1@#2\endcsname{#3}%
}

\def\new@makeaffilmacro#1{%
    \expandafter\def\csname #1\endcsname##1{%
        \@ifnextchar\bgroup
            {\csname ems@#1@two\endcsname{##1}}%
            {\csname ems@#1@one\endcsname{##1}}%
    }%
    \expandafter\def\csname ems@#1@two\endcsname##1##2{%
        \new@setmeta{\current@affil@id @##1}{#1}{##2}%
    }%
    \expandafter\def\csname ems@#1@one\endcsname##1{%
        \new@setmeta{\current@affil@id @1}{#1}{##1}%
    }%
}

\newcommand{\Allil}[2]{%
    \def\current@affil@id{#1}%
    \begingroup
        \new@makeaffilmacro{department}%
        \new@makeaffilmacro{organisation}%
        \new@makeaffilmacro{address}%
        \new@makeaffilmacro{post}%
        \new@makeaffilmacro{city}%
        \new@makeaffilmacro{country}%
        \new@makeaffilmacro{affemail}%
        #2%
    \endgroup
}

\def\ems@printid#1#2#3{%
    \ifcsdef{auth@\aid @#1}{%
        \edef\tempval{\csname auth@\aid @#1\endcsname}%
        \expandafter\ifblank\expandafter{\tempval}{}{%
            #2 \href{#3\tempval}{\tempval}\space
            \gdef\hasid{1}%
        }%
    }{}%
}

\AtEndDocument{%
    \par \addvspace{19.5pt}
    \begingroup
    \small \raggedright
    \ifnum\value{new@auth}>0
        \foreach \i in {1,...,\value{new@auth}} {%
            \def\aid{\i}%
            {\bfseries
                \csname auth@\aid @first\endcsname\ %
                \csname auth@\aid @last\endcsname
            }\par\nobreak
            \foreach \j in {1,2,3,4} {
                \def\affkey{\aid @\j}%
                \@ifundefined{affil@\affkey @organisation}{}{%
                    \@ifundefined{affil@\affkey @department}{}{%
                        \csname affil@\affkey @department\endcsname, %
                    }%
                    \csname affil@\affkey @organisation\endcsname\par
                    \@ifundefined{affil@\affkey @address}{}{%
                        \csname affil@\affkey @address\endcsname, %
                    }%
                    \@ifundefined{affil@\affkey @post}{}{%
                        \csname affil@\affkey @post\endcsname\ %
                    }%
                    \@ifundefined{affil@\affkey @city}{}{%
                        \csname affil@\affkey @city\endcsname, %
                    }%
                    \@ifundefined{affil@\affkey @country}{}{%
                        \csname affil@\affkey @country\endcsname
                    }%
                    \par
                    \ifcsdef{affil@\affkey @affemail}{%
                        \edef\tempemail{\csname affil@\affkey @affemail\endcsname}%
                        \expandafter\ifblank\expandafter{\tempemail}{}{%
                            \href{mailto:\tempemail}{\tempemail}\par
                        }%
                    }{}%
                }%
            }%
            \gdef\hasid{0}%
            \setbox0=\hbox{%
                \ems@printid{zblid}{zbMATH}{https://zbmath.org/authors/}%
                \ems@printid{mrid}{MR}{https://mathscinet.ams.org/mathscinet/MRAuthorID/}%
                \ems@printid{orcid}{ORCID}{https://orcid.org/}%
            }%
            \ifnum\hasid=1
                \noindent Author IDs:\space \unhbox0 \par
            \fi
            \vspace{13pt}%
        }%
    \fi
    \endgroup
}

\makeatother

\makeatletter
\renewcommand*\thebibliography[1]{%
    \section*{References}%
    \list{\@biblabel{\@arabic\c@enumiv}}{%
        \small
        \settowidth\labelwidth{\@biblabel{#1}}%
        \leftmargin \dimexpr\labelwidth+\labelsep\relax
        \itemsep 0pt
        \parsep 0pt
        \usecounter{enumiv}%
        \let\p@enumiv\@empty
        \renewcommand\theenumiv{\@arabic\c@enumiv}%
    }%
    \interlinepenalty \@M
    \emergencystretch 1em
    \sfcode`\.\@m
}
\makeatother

\newcommand*\arxiv[1]{%
    , arXiv:\href{https://arxiv.org/abs/#1}{#1}%
}

\newcommand{\dx}[1]{\mathop{}\!\mathrm{d}#1}
\newcommand{\dH}[1]{\mathop{}\!\mathrm{d}\mathcal{H}^{#1}}

\DeclareMathOperator{\dist}{dist}

\def\Xint#1{%
    \mathchoice
        {\XXint\displaystyle\textstyle{#1}}%
        {\XXint\textstyle\scriptstyle{#1}}%
        {\XXint\scriptstyle\scriptscriptstyle{#1}}%
        {\XXint\scriptscriptstyle\scriptscriptstyle{#1}}%
    \!\int
}

\def\XXint#1#2#3{%
    {%
        \setbox0=\hbox{$#1{#2#3}{\int}$}%
        \vcenter{\hbox{$#2#3$}}%
        \kern-.53\wd0
    }%
}

\def\barint{\Xint-}

\newcommand{\N}{\mathbb{N}}

\newcommand{\R}{\mathbb{R}}

\newcommand{\cC}{\ensuremath{\mathcal C}}

\newcommand{\cH}{\ensuremath{\mathcal H}}

\newcommand{\cJ}{\ensuremath{\mathcal J}}

\newcommand{\cL}{\ensuremath{\mathcal L}}

\newcommand{\cN}{\ensuremath{\mathcal N}}
\newcommand{\cO}{\ensuremath{\mathcal O}}

\newcommand{\cS}{\ensuremath{\mathcal S}}

\let\ge\geqslant
\let\le\leqslant

\newcommand{\mres}{\mathbin{\vrule height 1.6ex depth 0pt width 0.13ex\vrule height 0.13ex depth 0pt width 1.3ex}}
\newcommand{\msf}{\mathsf}

\graphicspath{{figures/}}

\begin{document}

\title{A Minimization Problem for a Cooperative System with a Degenerate Bernoulli Weight}

\Author{1}{
    \givenname{Lili}
    \surname{Du}
    \mrid{}
    \zblid{}
    \orcid{}
}{L.~Du}

\Author{2}{
    \givenname{Chunlei}
    \surname{Yang}
    \mrid{}
    \zblid{}
}{C.~Yang}

\Author{3}{
    \givenname{Jing}
    \surname{Yang}
    \mrid{}
    \zblid{}
    \orcid{}
}{J.~Yang}

\Allil{1}{
    \department{Department of Mathematics}
    \organisation{Sichuan University}
    \address{No. 24, South Section 1, Yihuan Road}
    \post{610000,}
    \city{Chengdu, Sichuan Province}
    \country{China}
    \affemail{dulili@scu.edu.cn}
}

\Allil{2}{
    \department{School of Mathematical Sciences}
    \organisation{Shenzhen University}
    \address{3688 Nanhai Avenue, Nanshan District}
    \post{518000,}
    \city{Shenzhen, Guangdong Province}
    \country{China}
    \affemail{yangchunlei@szu.edu.cn}
}

\Allil{3}{
    \department{Department of Mathematics}
    \organisation{Sichuan University}
    \address{No. 24, Wuhou District}
    \post{610000,}
    \city{Chengdu, Sichuan Province}
    \country{China}
    \affemail{yangjing6@stu.scu.edu.cn}
}

\msc[49J10]{35R35}

\begin{abstract}
  In this paper, we study local minimizers of the energy functional
  \[
  J(\mathbf{u})=\int_D\left(|\nabla \mathbf{u}|^2+Q^2(x)\chi_{\Omega_{\mathbf{u}}}\right)\,\dx{x},
  \]
  which give rise to a singular cooperative system. Here \(\mathbf{u}=(u_1,\dots,u_m): D\to\R_+^m\) is a vector-valued unknown function, \(\Omega_{\mathbf{u}}:=\{|\mathbf{u}|>0\}\) is the positive set, \(\chi_{\Omega_{\mathbf{u}}}\) is the characteristic function of the positive set \(\Omega_{\mathbf{u}}\), and \(Q(x)\) is the Bernoulli weight function. This problem was first introduced by Caffarelli, Shahgholian, and Yeressian in the poineer work ({\it Duke Math. J.} \textbf{167}(10), 2018), where the regularity theory for minimizers and for the free boundary \(\partial\Omega_{\mathbf{u}}\) was established for nondegenerate Bernoulli weights, i.e., \(Q(x)\ge Q_{\rm min}>0\). The present paper studies the same free boundary problem but for a degenerate Bernoulli weight. The main difficulty lies in the coupling between the degeneracy of \(Q(x)\) and the vector-valued nature of the problem.
  
  Our main results show that the free boundary set \(\partial\Omega_{\mathbf{u}}\cap\{Q(x)=0\}\) decomposes into a nondegenerate part and a degenerate part. Every nondegenerate point admits a nontrivial blow-up limit, while degenerate points have zero weighted density. Moreover, we further classify nondegenerate points into single-phase, multi-phase non-branching, and multi-phase branching points, and study the geometric structure of these types of free boundary points. In particular, the branching-point singular structure for multi-phase points is a purely vectorial phenomenon and appears to be new in this setting.
  
  Finally, as a byproduct, in the spirit of Naber--Valtorta ({\it Ann. Math.} \textbf{185}(1), 2017) and Edelen--Engelstein ({\it Trans. Amer. Math. Soc.} \textbf{371}(3), 2019), we prove that the nondegenerate set has locally finite \(\cH^{n-2}\)-measure and is countably \((n-2)\)-rectifiable.
\end{abstract}

\maketitle

\tableofcontents

\section*{Notation}

We first list some notational conventions that are frequently used in the paper. Throughout the paper, \(n\ge 2\), \(m\ge 1\) are fixed integers. 

\begin{itemize}
    \item \(\R^n\): Euclidean space; 
    \item \(\R_+^n\): The set \(\{x=(x_1,x_2,\dots,x_n)\in\R^n: x_i\ge 0,\ i=1,\dots,n\}\); 
    \item \(\mathbb S^{n-1}\): The unit \((n-1)\)-dimensional sphere; 
    \item \(x\cdot y\): \(\sum_{i=1}^n x_i y_i\),\quad \(x,y\in \R^n\); 
	\item \(\dist(x,E)\): \(\inf_{y\in E} |x-y|\),\quad \(x\in\R^n\),\quad \(E\subset \R^n\); 
	\item \(\Sigma\): A given \(C^{1,\mathrm{Dini}}\) closed embedded hypersurface; 
	\item \(T_x\Sigma\): The tangent space of \(\Sigma\) at \(x\in \Sigma\); 
	\item \(\msf d_\Sigma(x)\): \(\dist(x,\Sigma)\), the distance from \(x\) to \(\Sigma\); 
	\item \(\Pi_\Sigma(x)\): The set of nearest points to \(x\) on \(\Sigma\); 
    \item \(\pi_\Sigma(x)\): The set of  nearest points when \(\Pi_\Sigma(x)\) is a singleton; 
	\item \(B_r(x)\): \(\{y\in \R^n: |y-x|<r\}\); 
	\item \(\partial B_r(x)\): \(\{y\in \R^n: |y-x|=r\}\); 
	\item \(N_\epsilon(E)\): The \(\epsilon\)-neighborhood of \(E\subset\R^n\), \(\{x\in \R^n: \dist(x,E)<\epsilon\}\); 
	\item \(\cH^k\): The \(k\)-dimensional Hausdorff measure (on \(\R^n\)); 
	\item \(\chi_A\): The characteristic function of a set \(A\) (\(A\subset\R^n\)); 
    \item \(\# A\): The cardinality of a set \(A\); 
	\item \(f_+,f_-\): \(\max\{f,0\},\max\{-f,0\}\); 
	\item \(A\triangle B\): \((A\setminus B)\cup(B\setminus A)\). 
	\item \(\barint_A f\,\dx{\mu}\): Integral mean \(\frac{1}{\mu(A)}\int_A f\,\mathrm d\mu\); 
    \item \(\Omega_{\mathbf{u}}\), \(\Omega_0\), \(\Omega_{u_i}\): \(\{|\mathbf{u}|>0\}\), \(\{|\mathbf{u}_0|>0\}\), \(\{u_i>0\}\); 
    \item \(\partial\Omega_{\mathbf{u}}\), \(\partial\Omega_{u_i}\), \(\partial^*\Omega_{\mathbf{u}}\): \(\partial\{|\mathbf{u}|>0\}\), \(\partial\{u_i>0\}\), the reduced free boundary of \(\Omega_{\mathbf{u}}\); 
    \item \(\Gamma_\Sigma(\mathbf{u})\): \(\partial\{|\mathbf{u}|>0\}\cap\Sigma\cap D\); 
    \item \(\subset\subset\): Compactly contained. 
\end{itemize}

\section{Introduction and Main Results}

\subsection{Introduction to the variational problem}

Free boundary problems are a special type of boundary value problem, in which the domain where the PDE is to be satisfied is not known a priori and depends on the solution. The aim of this paper is to study local minimizers \(\mathbf{u}=(u_1,\dots,u_m)\) of the functional
\begin{equation}\label{eq.eng}
	J(\mathbf{u};D)=\int_D\bigl(|\nabla \mathbf{u}|^2+Q^2(x)\chi_{\Omega_{\mathbf{u}}}\bigr)\,\dx{x},
\end{equation}
over the admissible class
\[
    \begin{aligned}
        \mathbb{A}:={}\bigl\{\mathbf{u}\in W^{1,2}(D;\R^m):&\ \mathbf{u}-\mathbf{g}\in W_0^{1,2}(D;\R^m),\\ 
        &u_i\ge 0 \text{ a.e.\ in }D,\ i=1,\dots,m\bigr\},
    \end{aligned}
\]
where \(D\subset\R^n\) is a bounded open set and \(\mathbf{g}=(g_1,\dots,g_m)\in W^{1,2}(D;\R^m)\) is the prescribed boundary datum, with \(g_i\ge0\) a.e. in \(D\) for every \(i=1,\dots,m\). The real-valued function \(Q(x):D\to\R_+\) in \eqref{eq.eng} is usually referred to as the \textit{Bernoulli weight function}. The main theoretical challenges associated with the above minimization problem concern the regularity theory of minimizers and of free boundaries. Suppose that the Bernoulli weight function \(Q(x)\) is H\"{o}lder continuous and satisfies the following non-degeneracy condition
\begin{equation}\label{eq.ndc}
	Q(x) \ge Q_{\mathrm{min}}>0\qquad\text{ in }D,
\end{equation}
for some positive constant \(Q_{\mathrm{min}}>0\). For each \(x_0\in\partial\Omega_{\mathbf{u}}\cap D\), the local NTA construction in \cite[Theorem 5]{CSY18} provides an NTA domain \footnote{See Definitions 6 and 7 on p. 1859 of \cite{CSY18} for the precise definition for NTA domain. Classical examples are Lipschitz domains.} (nontangentially accessible domain) \(G\) and radii \(0<\rho<R\), with \(B_R(x_0)\subset\subset D\), such that
\[
    \Omega_{\mathbf{u}}\cap B_\rho(x_0)
    \subset G
    \subset\Omega_{\mathbf{u}}\cap B_R(x_0).
\]
A suitable component \(u_{i_0}\) is positive in \(G\). The boundary Harnack principle gives local H\"older extensions of the ratios \(u_j/u_{i_0}\) to the free boundary near \(x_0\). Consequently, the effective scalar Bernoulli weight
\[
   Q(x)\frac{u_{i_0}}{|\mathbf{u}|}
    =\frac{Q(x)}{
        \sqrt{1+\sum_{j\ne i_0}(u_j/u_{i_0})^2}}
    \qquad\text{in }G,
\]
extends as a positive H\"older continuous function to the nearby free boundary. The scalar reduction developed in \cite{CSY18}, together with the scalar regularity theory, then gives $C^{1,\alpha}$ regularity of the regular part of the free boundary. Higher regularity requires corresponding additional regularity of the Bernoulli coefficient.

The regularity results for vector-valued cooperative free boundary problems established in \cite{CSY18} represent a foundational contribution to the field. Following this work, a systematic regularity theory for cooperative systems has been further developed in \cite{DST20,BFS24,SV26a}, among many others. In particular, Mazzoleni, Terracini and Velichkov \cite{MTV20} removed the sign constraint \(u_i\ge 0\) on components. Under the choice of \(Q(x) \equiv \Lambda > 0\) for some positive constant \(\Lambda\) and the non-degeneracy condition \eqref{eq.ndc}, they identified a class of singular points with weighted density one, and decomposed the free boundary in arbitrary dimensions into a smooth regular part, a scalar-type one-phase singular set, and a density-one singular set. This development is analogous to the extension from one-phase to two-phase problems in scalar Bernoulli-type free boundary problems. In the two-dimensional case, Spolaor and Velichkov \cite{SV19} had earlier obtained a fine description of the free boundary by means of an epiperimetric inequality. The epiperimetric approach and the fine two-dimensional analysis developed in \cite{SV19} provide an alternative idea and technical inspiration for further investigation of the regularity and structure of singular points in vector-valued free boundary problems.

In parallel to the minimization problem formulation, cooperative systems also arise naturally in spectral shape optimization \cite{BM93}, where they correspond to Bernoulli quasi-minimizers. The regularity of solutions and their free boundaries can be derived entirely from this perspective; we refer the reader to \cite{MTV17,KL18,KL19} for core developments along this path. Additionally, there is a line of research that operates under the non-degeneracy condition \eqref{eq.ndc} but imposes alternative constraints in place of the componentwise nonnegativity constraint; representative works in this direction include \cite{DS24,FV25}.

However, we emphasize that the results discussed above for vector-valued free boundary problems rely crucially on the non-degeneracy condition \eqref{eq.ndc}. The nondegenerate theory applies locally wherever the Bernoulli weight function \(Q(x)\) is bounded below by a positive constant. Our focus is the behavior at contact points where this uniform lower bound fails. To this end, we consider the following Bernoulli weight function 
\begin{equation}\label{eq.Q}
    Q(x)=[\dist(x,\Sigma)]^\gamma:=\msf d_\Sigma^\gamma(x),\quad x\in D,
\end{equation}
where \(\gamma>0\) is a given constant and \(\Sigma\subset\R^n\) is a given closed embedded hypersurface of class \(C^{1,\rm Dini}\). By construction, it follows immediately that \(Q(x_0)=0\) whenever \(x_0\in\partial\Omega_{\mathbf{u}}\cap\Sigma\), the intersection of the free boundary and the given hypersurface. We formally define

\begin{definition}[Contact point  and contact set, cf. Figure \ref{fig:contact_set}]\label{def.contact}
    Let \(x_0\in \partial \Omega_{\mathbf{u}}\cap D\). We say that \(x_0\) is a \emph{contact point} if \(x_0\in \partial\Omega_{\mathbf{u}}\cap\Sigma\); otherwise, we refer to \(x_0\) as a \emph{non-contact point}. The set of all contact points is called the \emph{contact set}, denoted by
    \[
        \Gamma_\Sigma(\mathbf{u}):=\partial\Omega_{\mathbf{u}}\cap\Sigma\cap D.
    \]
    We focus on the case \(\Gamma_\Sigma(\mathbf{u})\ne\varnothing\). If \(\Gamma_\Sigma(\mathbf{u})=\varnothing\), the nondegenerate theory applies locally along the interior free boundary; this does not exclude its possible singular points.
\end{definition}

\begin{figure}[!htbp]
    \centering
    \includegraphics[width=\textwidth]{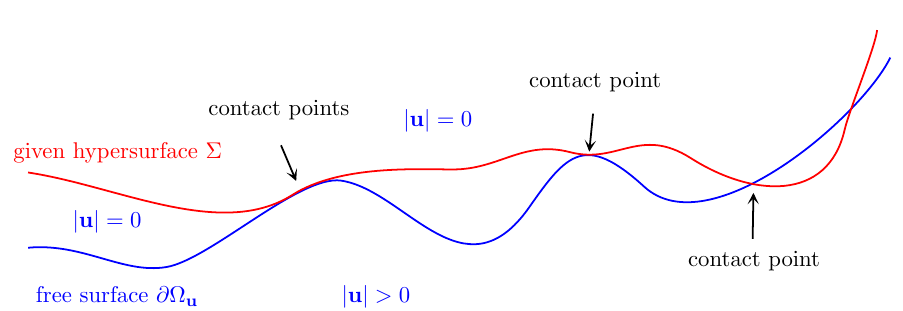}
    \caption{An illustration of the contact set \(\Gamma_\Sigma(\mathbf{u})\) where the free boundary \(\partial\Omega_{\mathbf{u}}\) intersects the hypersurface \(\Sigma\).}
    \label{fig:contact_set}
\end{figure}

The choice of the weight function in \eqref{eq.Q} is also motivated by some physical considerations. In the scalar two-dimensional case (\(m=1\) and \(n=2\)), we set \(\Sigma=\{x_2=0\}\), a flat smooth hypersurface and choose \(\gamma=1/2\). This choice is closely related to the famous Stokes conjecture in fluid mechanics. Note that under this choice, we obtain \(\msf d_\Sigma(x)=|x_2|\), and the energy functional \eqref{eq.eng} reduces to
\begin{equation}\label{eq.vw}
    J(u;D)=\int_D\left( |\nabla u|^2+|x_2|\chi_{\Omega_u} \right)\,\dx{x}.
\end{equation}
The contact set and contact points for \eqref{eq.vw} are thoroughly studied in \cite{VW11} under the physical consideration \(u\equiv 0\) in \(\{x_2\le 0\}\). We note that critical points (in particular, local minimizers) of \eqref{eq.vw} describe the motion of a two-dimensional traveling water wave under gravity. Correspondingly, contact points correspond to stagnation points of the wave. A well-known conjecture concerning stagnation points was formulated by G. Stokes in 1880 \cite{Sto80}, which states that the interface between water and air forms a \(120^\circ\) opening angle at the stagnation point. The Stokes conjecture and its variants have been extensively studied over the past 50 years \cite{AFT82,VW11,VW12,VW14,McC24,DPY25}, and serve as a canonical example of the interplay between mathematics and physics. In this paper, we investigate the contact set in the more general vectorial setting with \(m\ge 2\) and \(n\ge 2\).

Let us remark that the class of weights \(Q(x)=\msf d_\Sigma^\gamma(x)\) defined in \eqref{eq.Q} contains both flat and curved degeneracy geometries. We list some examples of \(\Sigma\), the corresponding Bernoulli weight function \(Q(x)\) and the corresponding contact sets below.

\begin{example}
    \leavevmode
    \begin{enumerate}
        \item 
        {\it Hyperplanes.}\quad Let \(\Sigma=\{x\in\R^n:x\cdot\nu=0\}\) be a hyperplane with unit normal \(\nu\in\mathbb S^{n-1}\). Then \(\Sigma\) is a \(C^\infty\), hence \(C^{1,\rm Dini}\), embedded hypersurface. The distance function reads \(\msf d_\Sigma(x)=|x\cdot\nu|\), so \(Q(x)=|x\cdot\nu|^\gamma\). In particular, taking \(\nu=e_n\) yields \(\Sigma=\{x_n=0\}\) and \(Q(x)=|x_n|^\gamma\). This geometry is closely tied to stagnation point analysis for gravity water waves with \(m=1,n=2\) and \(\gamma=1/2\). In this example, the contact set is \(\Gamma_\Sigma(\mathbf{u})=\partial\Omega_{\mathbf{u}}\cap D\cap\{x\cdot\nu=0\}\).
        \item {\it Ellipsoids.}\quad Let \(a_1,\dots,a_n>0\), and set  
        \[
            \Sigma:=\left\{
                x\in\R^n:
                \sum_{i=1}^n \frac{x_i^2}{a_i^2}=1
            \right\}.
        \]
        Then \(\Sigma\) is a compact \(C^\infty\) embedded hypersurface, hence in particular \(C^{1,\rm Dini}\). For any point \(x\in D\) with \(\sum_{i=1}^n(x_i^2/a_i^2)\ge 1\), a direct computation yields 
        \[
            \msf d_\Sigma(x)=\sqrt{\sum_{i=1}^n\frac{\lambda^2(x)x_i^2}{(a_i^2+\lambda(x))^2}},
        \]
        where \(\lambda(x)\ge 0\) is the unique solution to \(\sum_{i=1}^n(a_i^2x_i^2)/(a_i^2+\lambda)^2=1\). The Bernoulli weight function \(Q(x)=\msf d_\Sigma^\gamma(x)\) follows immediately, and the contact set is given by \(\Gamma_\Sigma(\mathbf{u})=\partial\Omega_{\mathbf{u}}\cap D\cap\bigl\{\sum_{i=1}^nx_i^2/a_i^2=1\bigr\}\).
        \item {\it Paraboloids.}\quad Write \(x=(x',x_n)\in\R^{n-1}\times\R\) and define
        \[
            \Sigma=\left\{
                (x',x_n)\in\R^n:
                x_n= \frac12 |x'|^2
            \right\}.
        \]
        As a closed \(C^\infty\) embedded hypersurface, \(\Sigma\) is automatically \(C^{1,\rm Dini}\). For any \(x=(x',x_n)\in\R^n\), direct computation yields
        \[
            \msf d_\Sigma(x)=\sqrt{\min_{y'\in\R^{n-1}}\left\lbrace
                    |x'-y'|^2+\left(x_n-\frac12|y'|^2\right)^2
                \right\rbrace},
        \]
        which automatically gives the weight function \(Q(x)=\msf d_\Sigma^\gamma(x)\). The contact set is \(\Gamma_\Sigma(\mathbf{u})=\partial\Omega_{\mathbf{u}}\cap D\cap\{x_n=\tfrac 12 |x'|^2\}\).
        \item {\it A standard torus in \(\R^3\).}\quad Let \(R>a>0\) and define
        \[
            \Sigma=\left\{
                x=(x_1,x_2,x_3)\in\R^3:
                \left(\sqrt{x_1^2+x_2^2}-R\right)^2+x_3^2=a^2
            \right\}.
        \]
        This is a compact, connected \(C^\infty\) embedded hypersurface in \(\R^3\), and hence \(C^{1,\rm Dini}\). The distance function is given by 
        \[
            \msf d_\Sigma(x):=\left|
                \sqrt{
                    \left( 
                        \sqrt{x_1^2+x_2^2}-R
                    \right)^2+x_3^2
                }
                -a
            \right|,
        \]
        which immediately gives the weight function \(Q(x)=\msf d_\Sigma^\gamma(x)\) and the contact set \(\Gamma_\Sigma(\mathbf{u})=\partial\Omega_{\mathbf{u}}\cap D\cap\{(\sqrt{x_1^2+x_2^2}-R)^2+x_3^2-a^2=0\}\).
    \end{enumerate}
\end{example}

\subsection{Main results}

Before presenting our main results, we first fix some basic definitions and notation used throughout the paper.

\begin{definition}[Absolute and local minimizers]
    A vector-valued function \(\mathbf{u}\in\mathbb A\) is called an \textit{absolute minimizer} of the energy \(J\) in \eqref{eq.eng} if 
    \[
        J(\mathbf{u};D)\le J(\mathbf{v};D)\quad\text{ for every }\mathbf{v}\in\mathbb A.
    \]
    We say that \(\mathbf{u}\in\mathbb A\) is an \textit{\(\varepsilon_0\)-local minimizer} of \(J\) if there exists \(\varepsilon_0>0\) such that \(J(\mathbf{u};D)\le J(\mathbf{v};D)\) for every \(\mathbf{v}\in\mathbb A\) satisfying 
        \begin{equation}\label{eq.local-min}
        \|\nabla (\mathbf{u}-\mathbf{v})\|_{L^2(D)}^2+\|\chi_{\Omega_{\mathbf{u}}}-\chi_{\Omega_{\mathbf{v}}}\|_{L^1(D)}<\varepsilon_0.
        \end{equation} 
\end{definition}

\begin{remark}
    In \cite{CSY18}, \(\varepsilon_0\)-local minimizers are defined via the following metric \(d\) on the Sobolev space \(W^{1,2}(D;\R^m)\) 
    \[
        d(\mathbf{u},\mathbf{v}):=\|\mathbf{u}-\mathbf{v}\|_{W^{1,2}(D;\R^m)}+\|\chi_{\Omega_{\mathbf{u}}}-\chi_{\Omega_{\mathbf{v}}}\|_{L^1(D)}.
    \]
    Note that since \(\mathbf{u}-\mathbf{v}\in W_0^{1,2}(D;\R^m)\) for every \(\mathbf{u},\mathbf{v}\in \mathbb A\), these two definitions are equivalent after possibly changing the threshold \(\varepsilon_0\), by Poincaré's inequality. Our formulation \eqref{eq.local-min} is adopted only for technical convenience, to simplify subsequent computations.
\end{remark}

\begin{remark}
    The existence of absolute minimizers is standard, following from the direct method of the calculus of variations. We refer the reader to \cite[Theorem 1]{CSY18} for a full proof.
\end{remark}

Now let \(\mathbf{u}\) be an \(\varepsilon_0\)-local minimizer of \(J\). We define the \textit{weighted density} of the positive set \(\Omega_{\mathbf{u}}\) at a contact point \(x_0\in \Gamma_\Sigma(\mathbf{u})\) by 
\begin{equation}\label{eq.weighted-density}
    \begin{aligned}
        \Theta_\Sigma(\mathbf{u};x_0,r)
        :=&{}
        r^{-n-2\gamma}
        \int_{B_r(x_0)}
        Q^2(x)\chi_{\Omega_{\mathbf{u}}}\,\dx{x}\\
        =&{}
        r^{-n-2\gamma}
        \int_{B_r(x_0)}
        \msf d_\Sigma^{2\gamma}(x)\chi_{\Omega_{\mathbf{u}}}\,\dx{x}.
    \end{aligned}
\end{equation}

Our main results fall into four parts, centered on the geometry of contact points and multi-phase contact points on the free boundary. First, we establish the density dichotomy and global geometric characterization of the contact set (cf. Theorem \ref{thm:main}). Second, we provide a classification of branching structures and local topological properties for multi-phase contact points (cf. Theorem \ref{thm:structure-ND} and Theorem \ref{thm:structure-mp-br}). Third, we prove the existence, homogeneity and vectorial structure of blow-up limits at contact points (cf. Theorem \ref{thm:blow-up-ND}), and obtain the full classification of limit solutions in two dimensions (cf. Theorem \ref{thm:planar-blow-up-classification}). 
\begin{table}[htbp!]
    \centering
    \footnotesize
    \setlength{\tabcolsep}{4pt}
    \renewcommand{\arraystretch}{1.25}
    \caption{Classification of contact points.}
    \label{tab:contact-classification}
   \begin{tabular}{|p{0.25\textwidth}|p{0.33\textwidth}|p{0.36\textwidth}|}
        \hline
        \multicolumn{3}{|c|}{\textbf{Contact set \(\Gamma_\Sigma(\mathbf{u})\)}} \\
        \hline\hline
\multirow{2}{*}{\makecell[l]{\textbf{Degenerate} contact point \\ \(x_0\in\Gamma_{\Sigma,\rm D}(\mathbf{u})\), \\ i.e. \(\Theta_\Sigma=0\)}}
& \multicolumn{2}{l|}{\textbf{Non-degenerate} contact point \(x_0\in\Gamma_{\Sigma,\rm ND}(\mathbf{u})\), i.e. \(\Theta_\Sigma>0\)} \\
\cline{2-3}
      &\makecell[l]{\textbf{Single-phase }point \\ \(x_0\in\Gamma_{\rm sp}(\mathbf{u})\), i.e. \(\#\,\mathcal{I}_{\mathbf{u}}(x_0)=1\)} 
      &\makecell[l]{\textbf{Multi-phase }point\\ \(x_0\in\Gamma_{\rm mp}(\mathbf{u})\), i.e. \(\#\,\mathcal{I}_{\mathbf{u}}(x_0)\geq 2\)} \\
      \hline
 & \multirow{2}{*}{\makecell[l]{Singular profile is reduced to \\ \protect\cite{VW11,McC24}}}
        & \textit{Non-branching} \(\Gamma_{\rm mp}^{\rm nb}\): the geometric structure and the blow-up limits are studied. \\
        \cline{3-3}
        & & \textit{Branching} \(\Gamma_{\rm mp}^{\rm br}\): the geometric structure and the blow-up limits are studied. \\
        \hline
    \end{tabular}
\end{table}
As by-products of the above analysis, we also derive Hausdorff measure estimates for the set of contact points (cf. Theorem \ref{thm:measure-GammaND}), and rectifiability results for the free boundary near non-degenerate contact points (cf. Theorem \ref{thm:measure-GammaND}). The results can also be summarized into Table \ref{tab:contact-classification}.

\begin{maintheorem}{I}[Structure of the contact points]\label{thm:main}
    Let \(\mathbf{u}\) be an \(\varepsilon_0\)-local minimizer of \(J\) in \(\mathbb A\). For every \(x_0\in \Gamma_\Sigma(\mathbf{u})\), the limit \(\Theta_\Sigma(\mathbf{u};x_0,0^+):=\lim_{r\to 0^+}\Theta_\Sigma(\mathbf{u};x_0,r)\) exists, and there are two positive constants \(\theta_1(n,\gamma)\), \(\theta_2(n,\gamma)\) depending on \(n\) and \(\gamma\) such that 
    \begin{equation}\label{eq.density-gap}
        \Theta_\Sigma(\mathbf{u};x_0,0^+)\in\{0\}\cup[\theta_1(n,\gamma),\theta_2(n,\gamma)].
    \end{equation}
    In particular, the contact set \(\Gamma_\Sigma(\mathbf{u})\) can be decomposed into two disjoint sets
    \[
        \Gamma_\Sigma(\mathbf{u}) = \Gamma_{\Sigma,\rm ND}(\mathbf{u})\cup \Gamma_{\Sigma,\rm D}(\mathbf{u}),
    \]
    where \(\Gamma_{\Sigma,\rm ND}(\mathbf{u})\) and \(\Gamma_{\Sigma,\rm D}(\mathbf{u})\) are defined by 
    \begin{equation}\label{eq:GammaND-definition-density}
        \Gamma_{\Sigma,\rm ND}(\mathbf{u})
        :=\left\{ 
            x_0\in \Gamma_\Sigma(\mathbf{u}):\Theta_\Sigma(\mathbf{u};x_0,0^+)>0
        \right\},
    \end{equation}
    and 
    \begin{equation}\label{eq:GammaD-definition-density}
        \Gamma_{\Sigma,\rm D}(\mathbf{u})
        :=
        \left\{
            x_0\in\Gamma_\Sigma(\mathbf{u}):
            \Theta_\Sigma(\mathbf{u};x_0,0^+)=0
        \right\}.
    \end{equation}
\end{maintheorem}

Based on the conclusions of Theorem \ref{thm:main}, we formally define the notion of non-degenerate contact points and degenerate contact points.

\begin{definition}[Non-degenerate and degenerate contact points]
    We say that \(x_0\in \Gamma_\Sigma(\mathbf{u})\) is a non-degenerate contact point if \(x_0\in \Gamma_{\Sigma,\rm ND}(\mathbf{u})\), a degenerate contact point if \(x_0\in \Gamma_{\Sigma,\rm D}(\mathbf{u})\). 
\end{definition}

\begin{remark}
    If \(x_0\in \Gamma_{\Sigma,{\rm ND}}(\mathbf{u})\) is a non-degenerate contact point, then every blow-up limit \(\mathbf{u}_0\) is nontrivial (\(\mathbf{u}_0 \not\equiv 0\)); by contrast, if \(x_0\in \Gamma_{\Sigma,{\rm D}}(\mathbf{u})\) is a degenerate contact point, then all blow-up limits satisfy \(\mathbf{u}_0 \equiv 0\). The main challenge in studying degenerate contact points arises precisely from the vanishing of the blow-up limits.
\end{remark}

\begin{remark}
    \eqref{eq.density-gap} is often referred to as the density-gap property for free boundary points. The gap separates contact points with vanishing weighted density from those with uniformly positive weighted density.
\end{remark}

We first focus on the non-degenerate contact set \(\Gamma_{\Sigma,\rm ND}(\mathbf{u})\) defined in \eqref{eq:GammaND-definition-density}. Let \(x_0\in\Gamma_{\Sigma,\rm ND}(\mathbf{u})\) be a non-degenerate contact point. We denote by \(\mathscr C_{\mathbf{u}}(x_0;r)\) the collection of all connected components that have \(x_0\) as a boundary point, i.e.,
\begin{equation}\label{eq:touching-components}
	\mathscr C_{\mathbf{u}}(x_0;r)
    :=
    \left\{
        \mathcal O:
        \begin{array}{l}
        \mathcal O\text{ is a connected component of }
        \Omega_\mathbf{u}\cap B_r(x_0)\\
        \text{ with }x_0\in\partial\mathcal O
        \end{array}
    \right\}.
\end{equation}
We further assume that for every \(x_0\in \Gamma_{\Sigma,\rm ND}(\mathbf{u})\), there exists \(R_{x_0}>0\) such that
\begin{equation}\label{eq:touching-branch-assumption}
    \mathscr C_{\mathbf{u}}(x_0;r)\ne\varnothing\quad\text{ for every }0<r<R_{x_0}.
\end{equation}

For any fixed \(\mathcal O\in\mathscr C_{\mathbf{u}}(x_0;r)\), we define its {\it active index} set by
\begin{equation}\label{eq:active-set-touching-component}
    I_{\mathbf{u}}(\mathcal O)
    :=
    \{i\in\{1,\dots,m\}: u_i>0 \text{ in }\mathcal O\}.
\end{equation}
Moreover, we define the {\it local active set} by the union of the active index sets of all connected components in \(\mathscr C_{\mathbf{u}}(x_0;r)\), i.e.,
\begin{equation}\label{eq:local-active-index-set}
    \mathcal I_{\mathbf{u}}(x_0,r)
    :=
    \bigcup_{\mathcal O\in \mathscr C_{\mathbf{u}}(x_0;r)} I_{\mathbf{u}}(\mathcal O).
\end{equation}
Since each component \(u_i\) is harmonic and nonnegative in every connected component of \(\Omega_\mathbf{u}\), we prove (see Appendix \ref{appendix:local-active-stabilization}) that there exist a radius \(r_{x_0}>0\) and a nonempty set \(\mathcal I_{\mathbf{u}}(x_0)\subset\{1,\dots,m\}\) such that
\[
    \mathcal I_{\mathbf{u}}(x_0,r)=\mathcal I_{\mathbf{u}}(x_0)\quad\text{ for every }0<r<r_{x_0}.
\]
A visual illustration of the notations \(\mathscr{C}_{\mathbf{u}}\), \(I\) and \(\mathcal{I}_{\mathbf{u}}\) can be found in Figure \ref{fig:four-configurations} (A). We are now able to define the single- and multi-phase points in the non-degenerate contact set \(\Gamma_{\Sigma,\rm ND}(\mathbf{u})\).
\begin{definition}[Single- and multi-phase points]\label{def:single-phase-ND}
    Let \(x_0\in\Gamma_{\Sigma,\rm ND}(\mathbf{u})\). We say that \(x_0\) is a \emph{single-phase point}  (cf. Figure \ref{fig:four-configurations} (B))  if \(\#\,\mathcal I_{\mathbf{u}}(x_0)=1\), and we denote the set of single-phase points by \(\Gamma_{\rm sp}(\mathbf{u})\). Otherwise, we call \(x_0\) a \emph{multi-phase point} and denote the set of such points by \(\Gamma_{\rm mp}(\mathbf{u})\).  
\end{definition}

\begin{remark}
    Obviously, \(\Gamma_{\Sigma,\rm ND}(\mathbf{u})=\Gamma_{\rm sp}(\mathbf{u})\cup\Gamma_{\rm mp}(\mathbf{u})\) and \(\Gamma_{\rm sp}(\mathbf{u})\cap\Gamma_{\rm mp}(\mathbf{u})=\varnothing\). 
\end{remark}

If \(\mathcal I_{\mathbf{u}}(x_0)=\{i_0\}\), then, for every \(0<r<r_{x_0}\),
\[
    \mathbf{u}=e_{i_0}u_{i_0}
    \qquad\text{on }
    \bigcup_{\mathcal O\in\mathscr C_{\mathbf{u}}(x_0;r)}\mathcal O,
\]
where \(e_{i_0}\) denotes the \(i_0\)-th standard basis vector. 
Thus each touching branch carries only the component \(u_{i_0}\). Let us emphasize that this identity does not imply a scalar reduction throughout a neighborhood containing additional, non-touching components. Scalar degenerate Bernoulli problems are studied in \cite{VW11,McC24}; here we focus on contact points with at least two active indices. For distinct \(i,j\in\mathcal I_{\mathbf{u}}(x_0)\), define the \emph{separation set} by
\[
    \triangle_{i,j}(x_0,r)
    :=
    \left(
        \partial\Omega_{u_i}
        \mathbin{\triangle}
        \partial\Omega_{u_j}
    \right)
    \cap B_r(x_0),
\]
where \(\partial\Omega_{u_i}:=\partial\{u_i>0\}\) and \(A\triangle B:=(A\setminus B)\cup(B\setminus A)\) denotes the symmetric difference of two sets \(A\) and \(B\). Note that if \(\triangle_{i,j}(x_0,r)=\varnothing\), then \(\partial\Omega_{u_i}\cap B_r(x_0)=\partial\Omega_{u_j}\cap B_r(x_0)\). On the other hand, if \(\triangle_{i,j}(x_0,r)\neq\varnothing\) for every \(0<r<r_{x_0}\), then the two component free boundaries \(\partial\Omega_{u_i}\) and \(\partial\Omega_{u_j}\) differ in every neighborhood of \(x_0\). We then define the coincidence set as the complement of the separation set relative to \((\partial\Omega_{u_i}\cup\partial\Omega_{u_j})\cap B_r(x_0)\), i.e.,
\[
    \mathcal K_{i,j}(x_0,r) := \left(\partial\Omega_{u_i}\cap \partial\Omega_{u_j}\right)\cap B_r(x_0).
\]
Thus
\[
    \left(
        \partial\Omega_{u_i}
        \cup
        \partial\Omega_{u_j}
    \right)
    \cap B_r(x_0)
    =
    \triangle_{i,j}(x_0,r)
    \,\cup\,
    \mathcal K_{i,j}(x_0,r).
\]

\begin{figure}[htbp]
	\centering
	
	\begin{minipage}[t]{0.47\textwidth}
		\centering
		
		\textbf{(A) Definitions of $ \mathscr C$, $I$  and $\mathcal I$.}
		
		\vspace{2mm}
		
		\includegraphics[
		width=\linewidth,
		height=4.2cm,
		keepaspectratio
		]{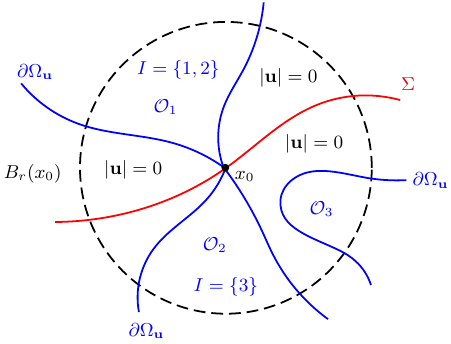}
		
		\vspace{0.5mm}
		
		\begin{flushleft}
			\small
			\[
			\begin{aligned}
				& \mathscr  C_{\mathbf u}(x_0;r)
				=\{\mathcal O_1,\mathcal O_2\},\\
			&\mathcal I_{\mathbf u}(x_0;r)
				=\{1,2\}\cup\{3\}=\{1,2,3\},\\
				&\mathcal O_3\notin \mathscr C_{\mathbf u}(x_0;r).
			\end{aligned}
			\]
		\end{flushleft}
	\end{minipage}
	\hfill
	\begin{minipage}[t]{0.47\textwidth}
		\centering
		
		\textbf{(B) Single-phase point.}
		
		\vspace{3mm}
		
		\includegraphics[
		width=\linewidth,
		height=4.0cm,
		keepaspectratio
		]{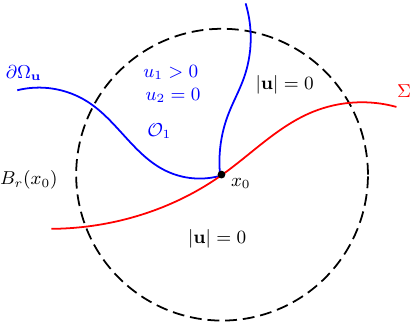}
		
		\vspace{3mm}
		
		\begin{flushleft}
			\small
			\[
			\mathcal I_{\mathbf u}(x_0;r)=\{1\}.
			\]
		\end{flushleft}
	\end{minipage}
	
	\vspace{5mm}
	
	\begin{minipage}[t]{0.47\textwidth}
		\centering
		
		\textbf{(C) Multi-phase non-branching point.}
		
		\vspace{2mm}
		
		\includegraphics[
		width=\linewidth,
		height=4cm,
		keepaspectratio
		]{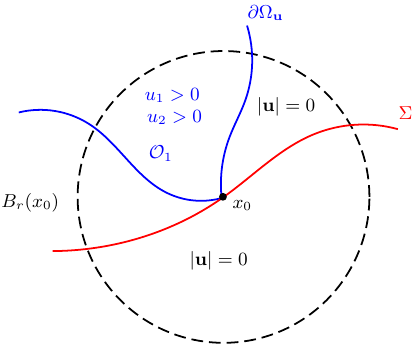}
		
		\vspace{2mm}
		
		\begin{flushleft}
			\small
			\[
			\begin{aligned}
				&\mathcal I_{\mathbf u}(x_0;r)=\{1,2\},\\
				&\Delta_{1,2}(x_0;r)=\varnothing.
			\end{aligned}
			\]
		\end{flushleft}
	\end{minipage}
	\hfill
	\begin{minipage}[t]{0.47\textwidth}
		\centering
		
		\textbf{(D) Multi-phase branching point.}
		
		\vspace{7mm}
		
		\includegraphics[
		width=\linewidth,
		height=4.0cm,
		keepaspectratio
		]{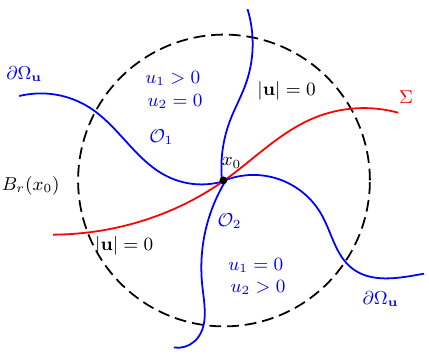}
		
		\vspace{2mm}
		
		\begin{flushleft}
			\small
			\[
			\begin{aligned}
			&	\mathcal I_{\mathbf u}(x_0;r)=\{1,2\},\\
			&	\Delta_{1,2}(x_0;r)\neq\varnothing,\\
			&\text{Phases separate topologically.}
			\end{aligned}
			\]
		\end{flushleft}
	\end{minipage}
	
	\caption{Local configurations near the point $x_0$.}
	\label{fig:four-configurations}
\end{figure}

We then classify the multi-phase points into two types: the \emph{multi-phase non-branching points} \(\Gamma_{\rm mp}^{\rm nb}(\mathbf{u})\) and the \emph{multi-phase branching points} \(\Gamma_{\rm mp}^{\rm br}(\mathbf{u})\) (cf. Figure \ref{fig:four-configurations} (C) and (D)).

\begin{maintheorem}{II\,(a)}[Structure of multi-phase nondegenerate contact points]
\label{thm:structure-ND}
    Let \(\mathbf{u}\) be an \(\varepsilon_0\)-local minimizer of \(J\) in \(\mathbb{A}\), and assume \eqref{eq:touching-branch-assumption}. Then the set of multi-phase contact points admits the disjoint decomposition
    \[
        \Gamma_{\rm mp}(\mathbf{u}) = 
        \Gamma_{\rm mp}^{\rm nb}(\mathbf{u}) 
        \cup 
        \Gamma_{\rm mp}^{\rm br}(\mathbf{u}).
    \]
    Here, the set of multi-phase non-branching points is defined by
    \begin{equation}\label{eq.multi-nb}
        \begin{aligned}
            \Gamma_{\rm mp}^{\rm nb}(\mathbf{u}) &:= \Bigl\{ x_0\in \Gamma_{\Sigma,{\rm ND}}(\mathbf{u}): \#\,(\mathcal I_{\mathbf{u}}(x_0))\ge 2 \text{ and there exists } 0<r<r_{x_0} \\
            &\qquad\qquad\text{such that } \triangle_{i,j}(x_0,r) = \varnothing \text{ for all distinct } i,j\in \mathcal I_{\mathbf{u}}(x_0) \Bigr\},
        \end{aligned}
    \end{equation}
    and the set of branching multi-phase points is defined by
    \begin{equation}\label{eq.multi-br}
        \begin{aligned}
            \Gamma_{\rm mp}^{\rm br}(\mathbf{u}) &:= \Bigl\{ x_0\in \Gamma_{\Sigma,{\rm ND}}(\mathbf{u}): \#\,(\mathcal I_{\mathbf{u}}(x_0))\ge 2 \text{ and there exist distinct } i,j\in \mathcal I_{\mathbf{u}}(x_0) \\
            &\qquad\qquad\text{such that } \triangle_{i,j}(x_0,r) \neq \varnothing \text{ for every } 0<r<r_{x_0} \Bigr\}.
        \end{aligned}
    \end{equation}
    Moreover, let \(x_0\in\Gamma_{\rm mp}^{\rm nb}(\mathbf{u})\), and suppose that there exists \(0<\bar r<r_{x_0}\) with \(B_{\bar r}(x_0)\subset\subset D\) for which the non-branching condition in \eqref{eq.multi-nb} holds. Fix any index \(i_0\in\mathcal I_{\mathbf{u}}(x_0)\), and take any point \(z\in \bigl(\partial\Omega_{u_{i_0}} \cap B_{\bar r}(x_0)\bigr) \setminus \Sigma\). Then there exists \(\rho_z>0\) such that \(B_{\rho_z}(z)\subset\subset B_{\bar r}(x_0) \cap (D\setminus\Sigma)\) and
    \[
        \partial\Omega_{\mathbf{u}} \cap B_{\rho_z}(z) = \partial\Omega_{u_i} \cap B_{\rho_z}(z) \quad\text{ for every } i\in \mathcal I_{\mathbf{u}}(x_0).
    \]
\end{maintheorem}

\begin{remark}
    Let \(x_0\in\Gamma_{\rm mp}^{\rm nb}(\mathbf{u})\). Then at every point of the common active free boundary lying away from \(\Sigma\), all active component free boundaries agree locally with \(\partial\Omega_{\mathbf{u}}\). However, this does not imply the global identity
    \[
        \partial\Omega_{\mathbf{u}}
        \cap
        \bigl(B_{\bar r}(x_0)\setminus\Sigma\bigr)
        =
        \partial\Omega_{u_i}
        \cap
        \bigl(B_{\bar r}(x_0)\setminus\Sigma\bigr)
        \qquad
        \text{for every }i\in\mathcal I_{\mathbf{u}}(x_0).
    \]
    Indeed, indices \(j\notin\mathcal I_{\mathbf{u}}(x_0)\) may have detached component free boundaries inside \(B_{\bar r}(x_0)\setminus\Sigma\), and these may contribute to \(\partial\Omega_{\mathbf{u}}\).
\end{remark}

The multi-phase branching set \(\Gamma_{\rm mp}^{\rm br}(\mathbf{u})\) is a genuinely vectorial phenomenon and has no scalar analogue. We isolate this case because the free boundaries of active components may fail to coincide in every neighborhood of the contact point. 

To better state the results, we first recall some notations. Let \(x_0\in\Gamma_{\rm mp}^{\rm br}(\mathbf{u})\), and fix \(0<\bar r<r_{x_0}\) such that \(B_{\bar r}(x_0)\subset\subset D\). We call \((i,j)\) a \emph{branching pair at \(x_0\)} for distinct \(i,j\in\mathcal I_{\mathbf{u}}(x_0)\) if \(\triangle_{i,j}(x_0,r)\ne\varnothing\) for every \(0<r<r_{x_0}\). We will prove in the forthcoming Theorem \ref{thm:structure-mp-br} that the family \(\mathscr C_{\mathbf{u}}(x_0;\bar r)\) is at most countable if \(x_0\in \Gamma_{\rm mp}^{\rm br}(\mathbf{u})\). We therefore choose an index set \(\Lambda_{\mathbf{u}}(x_0;\bar r)\subset\N\) and write 
\[
    \mathscr C_{\mathbf{u}}(x_0;\bar r)
    =
    \{\Omega_{\mathbf{u}}^{(\ell)}:\ell\in\Lambda_{\mathbf{u}}(x_0;\bar r)\},\quad x_0\in\Gamma_{\rm mp}^{\rm br}(\mathbf{u}).
\]
Here, for every \(\ell\in\Lambda_{\mathbf{u}}(x_0;\bar r)\), the notation \(\Omega_{\mathbf{u}}^{(\ell)}\) represents a connected component of the localized positive set \(\Omega_{\mathbf{u}}\cap B_{\bar r}(x_0)\) which satisfies \(x_0\in\partial\Omega_{\mathbf{u}}^{(\ell)}\). The components \(\Omega_{\mathbf{u}}^{(\ell)}\) are pairwise distinct and pairwise disjoint, and the above family consists of all connected components of \(\Omega_{\mathbf{u}}\cap B_{\bar r}(x_0)\) having \(x_0\) as a boundary point. We thus call them the \emph{local branches touching \(x_0\)}. 

In general, these local branches need not cover the entire positive set, since \(\Omega_{\mathbf{u}}\cap B_{\bar r}(x_0)\) may also contain connected components whose boundaries do not contain \(x_0\). Thus one only has the inclusion
\[
    \bigcup_{\ell\in\Lambda_{\mathbf{u}}(x_0;\bar r)}\Omega_{\mathbf{u}}^{(\ell)}\subset\Omega_{\mathbf{u}}\cap B_{\bar r}(x_0),
\] 
where the union on the left is disjoint. For every \(\ell\in\Lambda_{\mathbf{u}}(x_0;\bar r)\), define
\[
    \mathcal I_{\mathbf{u}}^{(\ell)}(x_0;\bar r)
    :=
    I_{\mathbf{u}}(\Omega_{\mathbf{u}}^{(\ell)})
    =
    \{i\in\{1,\dots,m\}:u_i>0
    \text{ in }\Omega_{\mathbf{u}}^{(\ell)}\}.
\]
It follows that
\[
    \mathcal I_{\mathbf{u}}(x_0,\bar r)
    =
    \bigcup_{\ell\in\Lambda_{\mathbf{u}}(x_0;\bar r)}
    \mathcal I_{\mathbf{u}}^{(\ell)}(x_0;\bar r).
\]
Since \(\bar r\in(0,r_{x_0})\), the stabilization property gives \(\mathcal I_{\mathbf{u}}(x_0,\bar r)=\mathcal I_{\mathbf{u}}(x_0)\). This is the construction presented in \eqref{eq:active-set-touching-component}--\eqref{eq:local-active-index-set}, but restricted to a multi-phase branching point \(x_0\in\Gamma_{\rm mp}^{\rm br}(\mathbf{u})\). The precise geometric properties of the multi-phase branching points are summarized in the following Theorem.

\begin{maintheorem}{II\,(b)}[Structure of multi-phase branching contact points]\label{thm:structure-mp-br}
    Let \(\mathbf{u}\) be an \(\varepsilon_0\)-local minimizer of \(J\) in \(\mathbb{A}\), assume \eqref{eq:touching-branch-assumption}, and let \(x_0\in \Gamma_{\rm mp}^{\rm br}(\mathbf{u})\). Fix \(0<\bar r<r_{x_0}\) such that \(B_{\bar r}(x_0)\subset\subset D\).
    \begin{enumerate}[label=(\roman*)]
        \item The collection of branching components \(\mathscr C_{\mathbf{u}}(x_0;\bar r)\) is at most countable. Hence there exists an index set \(\Lambda_{\mathbf{u}}(x_0;\bar r)\subset\mathbb{N}\) such that
        \[
            \mathscr C_{\mathbf{u}}(x_0;\bar r)
            =
            \{\Omega_{\mathbf{u}}^{(\ell)}:
            \ell\in\Lambda_{\mathbf{u}}(x_0;\bar r)\},
        \]
        where \(x_0\in\partial\Omega_{\mathbf{u}}^{(\ell)}\) for every \(\ell\in\Lambda_{\mathbf{u}}(x_0;\bar r)\).
        \item Let \((i,j)\) be a branching pair at \(x_0\). Fix \(0<r<\bar r\), and take any point \(z\in\triangle_{i,j}(x_0,r)\setminus\Sigma\). After interchanging \(i\) and \(j\) if necessary, we may assume that \(z\in(\partial\Omega_{u_i}\setminus\partial\Omega_{u_j})\cap B_r(x_0)\). Then there exists \(\rho_z>0\) such that \(B_{\rho_z}(z)\subset\subset B_r(x_0)\cap(D\setminus\Sigma)\), and
        \[
            \partial\Omega_{\mathbf{u}}\cap B_{\rho_z}(z)
            =
            \partial\Omega_{u_i}\cap B_{\rho_z}(z),\qquad \partial\Omega_{u_j}\cap B_{\rho_z}(z)
            =
            \varnothing.
        \]
        Furthermore, if \(z\in\partial\Omega_{\mathbf{u}}^{(\ell)}\) for some \(\ell\in\Lambda_{\mathbf{u}}(x_0;\bar r)\), then \(i\in\mathcal I_{\mathbf{u}}^{(\ell)}(x_0;\bar r)\) while \(j\notin\mathcal I_{\mathbf{u}}^{(\ell)}(x_0;\bar r)\). This local configuration is illustrated in Figure~\ref{fig:branch-local}.    
        \begin{figure}[htbp]
        	\centering
        	\includegraphics[width=0.45\textwidth]{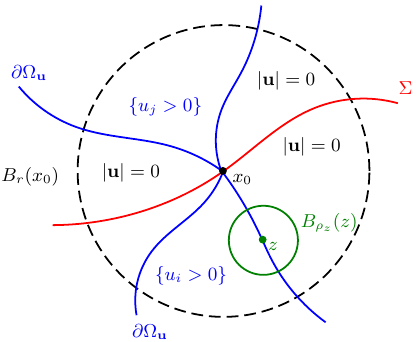}
        	\caption{Local configuration of the branching pair \((i,j)\) near \(x_0\).}
        	\label{fig:branch-local}
        \end{figure}  
    \end{enumerate}
\end{maintheorem}

\begin{remark}
    In what follows, if \(x_0=0\), we will simply write \(\mathscr C_{\mathbf{u}}(r)\), \(\triangle_{i,j}(r)\), \(\mathcal K_{i,j}(r)\), and \(\Lambda_{\mathbf{u}}(\bar r)\) in place of \(\mathscr C_{\mathbf{u}}(0;r)\), \(\triangle_{i,j}(0,r)\), \(\mathcal K_{i,j}(0,r)\), and \(\Lambda_{\mathbf{u}}(0;\bar r)\), respectively.
\end{remark}

We now turn to the third part of our main results, which concerns the blow-up analysis at non-degenerate contact points. This approach consists in zooming in near a contact point and adopting the natural rescaling dictated by the degeneracy of the Bernoulli weight. In this work, we use the natural scaling \(r^{1+\gamma}\) and consider the blow-up sequence defined by
\[
    \mathbf{u}_{x_0,r}(x) := \frac{\mathbf{u}(x_0 + rx)}{r^{1+\gamma}}, \qquad r>0, \qquad x_0\in \Gamma_{\Sigma,{\rm ND}}(\mathbf{u}).
\]
We prove that for every sequence \(r_j \to 0^+\), there exists a subsequence, still denoted by \(r_j\), such that \(\mathbf{u}_{x_0,r_j} \to \mathbf{u}_0\) strongly in \(W_{\rm loc}^{1,2}(\mathbb{R}^n;\mathbb{R}^m)\) and locally uniformly on \(\mathbb{R}^n\). The limit \(\mathbf{u}_0\in C_{\rm loc}^{0,1}(\mathbb{R}^n;\mathbb{R}^m)\) is nontrivial, and is called a nontrivial blow-up limit of \(\mathbf{u}\) at \(x_0\). The results concerning the above blow-up procedure are summarized as follows.

\begin{maintheorem}{III}[Blow-up limits at non-degenerate points]\label{thm:blow-up-ND}
    Let \(\mathbf{u}\) be an \(\varepsilon_0\)-local minimizer of \(J\) in \(\mathbb{A}\), let \(x_0\in \Gamma_{\Sigma,{\rm ND}}(\mathbf{u})\), and let \(\{r_j\}_{j\in\mathbb{N}}\) be a sequence of positive radii with \(r_j \to 0^+\). Then there exists a subsequence, still denoted \(\{r_j\}\), and a nontrivial blow-up limit \(\mathbf{u}_0\in C_{\rm loc}^{0,1}(\mathbb{R}^n;\mathbb{R}^m)\) such that
    \[
        \mathbf{u}_{x_0,r_j} \to \mathbf{u}_0
        \quad\text{strongly in } W_{\rm loc}^{1,2}(\mathbb{R}^n;\mathbb{R}^m)
        \quad\text{ and locally uniformly on } \mathbb{R}^n.
    \]
    The limit \(\mathbf{u}_0=(u_{0,1},u_{0,2},\dots,u_{0,m})\) is a \((1+\gamma)\)-homogeneous global minimizer of the tangent-plane functional
    \[
        J_T(\mathbf{v}; B_R) := \int_{B_R} \left( |\nabla \mathbf{v}|^2 + \mathsf{d}_0^{2\gamma} \chi_{\Omega_{\mathbf{v}}} \right) \dx{x}, \qquad R>0,
    \]
    where \(\mathsf{d}_0(x) := \operatorname{dist}(x, T_{x_0}\Sigma)\) denotes the Euclidean distance to the tangent hyperplane to \(\Sigma\) at \(x_0\).
    
    Moreover, if the positivity set \(\Omega_0 := \{|\mathbf{u}_0| > 0\}\) is connected, then setting \(v_0 := |\mathbf{u}_0|\), there exists a unique unit vector \(\mathbf{a} = (a_1,\dots,a_m)^{\mathsf{T}} \in \mathbb{R}^m\) with \(a_i \ge 0\) for all \(i\), such that
    \[
        \mathbf{u}_0(x) = \mathbf{a}  \, v_0(x) \qquad \text{for all } x\in\mathbb{R}^n.
    \]
    The scalar function \(v_0\) is a nontrivial \((1+\gamma)\)-homogeneous global minimizer of the scalar tangent functional
    \[
        \mathcal{J}_T(w; B_R) := \int_{B_R} \left( |\nabla w|^2 + \mathsf{d}_0^{2\gamma} \chi_{\Omega_w} \right) \dx{x}, \qquad R>0.
    \]
\end{maintheorem}

\begin{remark}
    \(\mathbf{u}_0\) is a global minimizer of \(J_T(\mathbf{v};B_R(0))\) in the sense that for every \(R>0\) and every \(\mathbf{v}=(v_1,\dots,v_m)\in W^{1,2}(B_R(0);\R^m)\) with \(\mathbf{v}-\mathbf{u}_0\in W_0^{1,2}(B_R;\R^m)\) and \(v_i\ge 0\), we have \(J_T(\mathbf{u}_0;B_R)\le J_T(\mathbf{v};B_R)\).
\end{remark}

\begin{remark}
    Let us remark that the blow-up limit \(\mathbf{u}_0\) is nontrivial due to the definition of non-degenerate contact points \(\Gamma_{\Sigma,\rm ND}(\mathbf{u})\). 
\end{remark}

For a nontrivial blow-up limit \(\mathbf{u}_0=(u_{0,1},\dots,u_{0,m})\), let \(\Omega_0:=\{|\mathbf{u}_0|>0\}\) be its positive set and write 
\[
    \mathcal I_{\mathbf{u}_0}(0) :=\{i\in\{1,\dots,m\}: u_{0,i}\not\equiv 0\}.
\]
Every connected component of \(\Omega_0\) is a cone touching the origin, so this agrees with the touching component definition of the local active set. If \(\Omega_0\) is connected, the strong maximum principle gives
\[
    \mathcal I_{\mathbf{u}_0}(0) = \{i\in\{1,\dots,m\}: u_{0,i}>0\text{ in }\Omega_0\}.
\]

The following corollaries state the componentwise structure of nontrivial blow-up limits. Their assertions concern the active indices of the limiting map.

\begin{corollary}[Blow-up limits for multi-phase non-branching points]\label{cor:blow-up-classification}
    Let \(x_0\in\Gamma_{\rm mp}^{\rm nb}(\mathbf{u})\), and let \(\mathbf{u}_0\) be a blow-up limit of \(\mathbf{u}\) at \(x_0\). Assume that \(\Omega_0=\{|\mathbf{u}_0|>0\}\) is connected, and write \(\mathbf{u}_0=\mathbf{a}  v_0\) for \(\mathbf{a} =(a_1,\dots,a_m)^{\mathsf T}\) as in Theorem \ref{thm:blow-up-ND}. Define 
    \[
        \operatorname{spt}\mathbf{a} :=\{i\in\{1,\dots,m\}: a_i>0\}.
    \]
    Then \(\mathcal I_{\mathbf{u}_0}(0)=\operatorname{spt}\mathbf{a} \) and 
    \[
        \mathbf{u}_0(x) = \mathbf{a}  v_0(x),
    \]
    Moreover, \(\partial\Omega_{u_{0,i}}=\partial\Omega_0\) for every \(i\in\operatorname{spt}\mathbf{a} \), whereas \(u_{0,i}\equiv 0\) for every \(i\notin\operatorname{spt}\mathbf{a} \).
\end{corollary}

\begin{corollary}[Branchwise structure of nontrivial blow-up limits]\label{cor:blow-up-branching}
    Let \(\mathbf{u}_0\) be any nontrivial blow-up limit of \(\mathbf{u}\) at \(x_0\in \Gamma_{\Sigma,\rm ND}(\mathbf{u})\). Write the disjoint decomposition of \(\Omega_0:=\{|\mathbf{u}_0|>0\}\) into its connected components as
    \[
        \Omega_0
        =
        \bigcup_{\ell\in\Lambda}\Omega_0^{(\ell)}.
    \]
   Then \(\Lambda\subset\mathbb N\) is finite, and every \(\Omega_0^{(\ell)}\) is an open cone satisfying \(0\in\partial\Omega_0^{(\ell)}\). Moreover, the following statements hold.
    \begin{enumerate}[label=(\roman*)]
        \item For every \(\ell\in\Lambda\), set \(v_\ell:=|\mathbf{u}_0|\) in \(\Omega_0^{(\ell)}\). Then there exists a unique vector \(\mathbf{a} _\ell=(a_{\ell,1},\dots,a_{\ell,m})^T\in\mathbb S^{m-1}\cap\mathbb R_+^m\) such that
        \[
            \mathbf{u}_0=\mathbf{a} _\ell v_\ell
            \qquad
            \text{ in }\Omega_0^{(\ell)}.
        \]
        The function \(v_\ell\) is positive, harmonic, and \((1+\gamma)\)-homogeneous in \(\Omega_0^{(\ell)}\).
        \item Define
        \[
            \operatorname{spt}\mathbf{a} _\ell
            :=
            \{i\in\{1,\dots,m\}:a_{\ell,i}>0\}.
        \]
        If \(0\in\Gamma_{\rm mp}^{\rm br}(\mathbf{u}_0)\), then there exist distinct \(\ell_1,\ell_2\in\Lambda\) such that \(\operatorname{spt}\mathbf{a} _{\ell_1}\ne\operatorname{spt}\mathbf{a} _{\ell_2}\).
    \end{enumerate}
\end{corollary}

In order to better state branching and non-branching, we restrict  attention to two dimensions. Note that in this case, we are able to classify all nontrivial homogeneous blow-up limits of \(\mathbf{u}\) at non-degenerate contact points. Let \(x_0\in\Gamma_{\Sigma,\rm ND}(\mathbf{u})\), and let \(\mathbf{u}_0\) be a nontrivial blow-up limit as in Theorem \ref{thm:blow-up-ND}. After a suitable rotation, assume that \(T_\Sigma:=T_{x_0}\Sigma=\{x_2=0\}\), then \(\msf d_0(x)=\dist(x,T_\Sigma)=|x_2|\). Define the four axial directions
\begin{equation}\label{eq:planar-axial-directions}
    \mathfrak A
    :=
    \left\{
        0,\frac\pi2,\pi,\frac{3\pi}{2}
    \right\}.
\end{equation}
For \(\theta,\mu\in\R\), let 
\[
    d_{\mathbb S^1}(\theta,\mu)
    :=
    \min_{k\in\mathbb Z}|\theta-\mu+2k\pi|.
\]
For every \(\mu\in\mathfrak A\), define the sector
\begin{equation}\label{eq:planar-canonical-sector}
    S_\mu
    :=
    \Bigl\{
        (r\cos\theta,r\sin\theta):
        r>0,\ d_{\mathbb S^1}(\theta,\mu)<\frac{\pi}{2(1+\gamma)}
    \Bigr\},
\end{equation}
and the function
\begin{equation}\label{eq:planar-canonical-profile}
    \Phi_\mu(r,\theta)
    :=
    \begin{cases}
        \kappa_\mu r^{1+\gamma}
        \cos\bigl((1+\gamma)\,d_{\mathbb S^1}(\theta,\mu)\bigr),
        &d_{\mathbb S^1}(\theta,\mu)<\frac{\pi}{2(1+\gamma)},\\
        0,
        &d_{\mathbb S^1}(\theta,\mu)\ge\frac{\pi}{2(1+\gamma)},
    \end{cases}
\end{equation}
where
\begin{equation}\label{eq:planar-kappa}
    \kappa_\mu
    :=
    \begin{cases}
        \frac{1}{1+\gamma}\sin^\gamma\!\frac{\pi}{2(1+\gamma)},
        &\mu\in\{0,\pi\},\\[8pt]
        \frac{1}{1+\gamma}\cos^\gamma\!\frac{\pi}{2(1+\gamma)},
        &\mu\in\{\frac\pi2,\frac{3\pi}{2}\}.
    \end{cases}
\end{equation}
Then we have

\begin{maintheorem}{IV}[Classification of planar blow-up limits]\label{thm:planar-blow-up-classification}
    \itshape 
    There exist a unique nonempty set \(K\subset\mathfrak A\) and a unique vector \(\mathbf{a} _\mu\in\mathbb S^{m-1}\cap\mathbb R_+^m\) for every \(\mu\in K\) such that
    \[
        \mathbf{u}_0
        =
        \sum_{\mu\in K}\mathbf{a} _\mu\Phi_\mu
        \qquad
        \text{in }\mathbb R^2.
    \]
    Moreover,
    \[
        \Omega_0
        =
        \bigcup_{\mu\in K}S_\mu,
    \]
    where the union is disjoint. Distinct sectors in this decomposition do not share a boundary ray. The set \(K\), and hence the possible form of \(\mathbf{u}_0\), is classified as follows.

    \begin{enumerate}[label=\textnormal{(\roman*)}]
        \item Suppose that \(0<\gamma\le1\). Then \(1\le\,\#\,K\le 2\) and the only possible
        nontrivial blow-up limits are of the following six forms: \(\mathbf{a} _0\Phi_0\), \(\mathbf{a} _\pi\Phi_\pi\), \(\mathbf{a} _0\Phi_0+\mathbf{a} _\pi\Phi_\pi\), \(\mathbf{a} _{\pi/2}\Phi_{\pi/2}\), \(\mathbf{a} _{3\pi/2}\Phi_{3\pi/2}\), and \(\mathbf{a} _{\pi/2}\Phi_{\pi/2}+\mathbf{a} _{3\pi/2}\Phi_{3\pi/2}\), where every coefficient vector occurring in a given expression belongs to \(\mathbb S^{m-1}\cap\mathbb R_+^m\) (see Figure \ref{fig:planar-cones-gamma-at-most-one}).

        \item Suppose that \(\gamma>1\). Then \(1\le \#\,K\le4\), and the possible sets \(K\subset\mathfrak A\), together with the forms of blow-up limit \(\mathbf{u}_0\), are listed in Table~\ref{tab:planar-blowups-gamma-large} (see Appendix \ref{sec:planar-blow-up-appendix}).
    \end{enumerate}
    In both cases, each connected component of \(\Omega_0\) is a sector of opening angle \(\pi/(1+\gamma)\).
\end{maintheorem}

\begin{figure}[ht]
    \centering
    \def\gammafig{0.5}
    \pgfmathsetmacro{\betaangle}{90/(1+\gammafig)}
    \def\farR{5.2}
    \def\halfW{1.62}
    \def\halfH{1.30}

    \begin{tikzpicture}[
        x=1cm,y=1cm,
        line cap=round,
        line join=round,
        every node/.style={font=\small},
        panel/.style={draw=black, dashed,line width=.45pt},
        tangent/.style={draw=black,densely dashed,line width=.45pt},
        normal/.style={draw=black,-{Latex[length=1.55mm,width=1.05mm]},line width=.45pt},
        freebdry/.style={draw=freeviolet,line width=1.05pt},
        phase/.style={fill=phasegreen,fill opacity=.115},
        title/.style={font=\small},
        profile/.style={font=\scriptsize,align=center}
    ]
        \def\DrawSector#1{%
            \path[phase]
                (0,0)
                -- ({#1-\betaangle}:\farR)
                arc[start angle={#1-\betaangle},end angle={#1+\betaangle},radius=\farR]
                -- cycle;
            \draw[freebdry] (0,0) -- ({#1-\betaangle}:\farR);
            \draw[freebdry] (0,0) -- ({#1+\betaangle}:\farR);
        }
        \def\PanelBase{%
            \draw[tangent] (-\halfW,0) -- (\halfW,0);
            \node[anchor=north east] at ({\halfW-.05},-.04) {\(T_\Sigma\)};
            \draw[normal] (0,-\halfH) -- (0,\halfH);
            \node[anchor=north west] at (.05,{\halfH-.03}) {\(e_2\)};
            \fill[contactred] (0,0) circle (1.25pt);
            \draw[panel] (-\halfW,-\halfH) rectangle (\halfW,\halfH);
        }

        \begin{scope}[shift={(-3.75,1.75)}]
            \node[title] at (0,1.61) {\textnormal{(a)} \(K=\{0\}\)};
            \begin{scope}\clip (-\halfW,-\halfH) rectangle (\halfW,\halfH);\DrawSector{0}\end{scope}
            \PanelBase
            \node[profile] at (.90,.34) {\(\mathbf{a} _0\Phi_0\)};
        \end{scope}

        \begin{scope}[shift={(0,1.75)}]
            \node[title] at (0,1.61) {\textnormal{(b)} \(K=\{\pi\}\)};
            \begin{scope}\clip (-\halfW,-\halfH) rectangle (\halfW,\halfH);\DrawSector{180}\end{scope}
            \PanelBase
            \node[profile] at (-.92,.34) {\(\mathbf{a} _\pi\Phi_\pi\)};
        \end{scope}

        \begin{scope}[shift={(3.75,1.75)}]
            \node[title] at (0,1.61) {\textnormal{(c)} \(K=\{0,\pi\}\)};
            \begin{scope}\clip (-\halfW,-\halfH) rectangle (\halfW,\halfH);\DrawSector{0}\DrawSector{180}\end{scope}
            \PanelBase
            \node[profile] at (-.92,.34) {\(\mathbf{a} _\pi\Phi_\pi\)};
            \node[profile] at (.92,.34) {\(\mathbf{a} _0\Phi_0\)};
        \end{scope}

        \begin{scope}[shift={(-3.75,-1.75)}]
            \node[title] at (0,1.61) {\textnormal{(d)} \(K=\{\frac{\pi}{2}\}\)};
            \begin{scope}\clip (-\halfW,-\halfH) rectangle (\halfW,\halfH);\DrawSector{90}\end{scope}
            \PanelBase
            \node[profile] at (.0,.72) {\(\mathbf{a} _{\pi/2}\Phi_{\pi/2}\)};
        \end{scope}

        \begin{scope}[shift={(0,-1.75)}]
            \node[title] at (0,1.61) {\textnormal{(e)} \(K=\{\frac{3\pi}{2}\}\)};
            \begin{scope}\clip (-\halfW,-\halfH) rectangle (\halfW,\halfH);\DrawSector{270}\end{scope}
            \PanelBase
            \node[profile] at (.0,-.72) {\(\mathbf{a} _{3\pi/2}\Phi_{3\pi/2}\)};
        \end{scope}

        \begin{scope}[shift={(3.75,-1.75)}]
            \node[title] at (0,1.61)
                {\textnormal{(f)} \(K=\{\frac{\pi}{2},\frac{3\pi}{2}\}\)};
            \begin{scope}\clip (-\halfW,-\halfH) rectangle (\halfW,\halfH);\DrawSector{90}\DrawSector{270}\end{scope}
            \PanelBase
            \node[profile] at (.0,.72) {\(\mathbf{a} _{\pi/2}\Phi_{\pi/2}\)};
            \node[profile] at (.0,-.72) {\(\mathbf{a} _{3\pi/2}\Phi_{3\pi/2}\)};
        \end{scope}
    \end{tikzpicture}
    \caption{The six possible sector configurations for \(0<\gamma\le1\), illustrated at \(\gamma=1/2\).}
    \label{fig:planar-cones-gamma-at-most-one}
\end{figure}
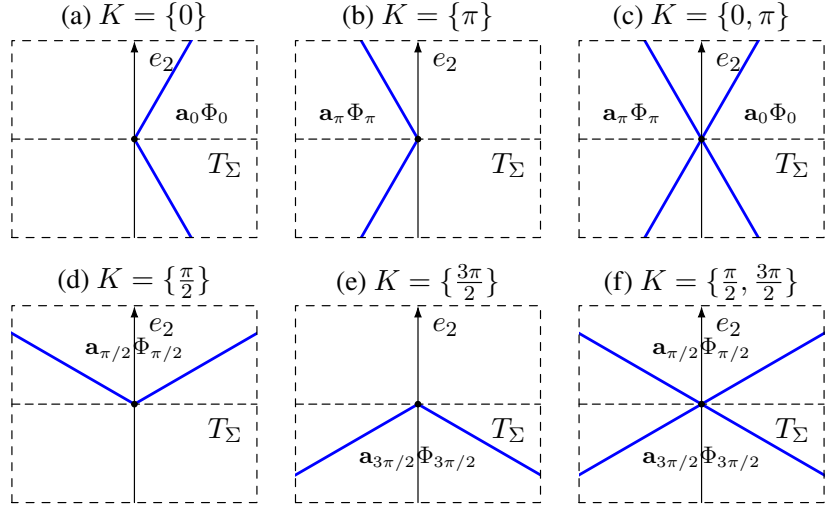

\begin{remark}
    We will list the configurations of planar blow-up limits when \(\gamma>1\) in Appendix \ref{sec:planar-blow-up-appendix} for the reader's convenience. 
\end{remark}

\begin{remark}
    Theorem~\ref{thm:planar-blow-up-classification} gives necessary forms of nontrivial blow-up limits. It does not assert that every listed configuration, with every permitted choice of coefficient vectors, is a global minimizer or is realized as a blow-up limit.
\end{remark}

\begin{remark}
    In the slow decay case \(0<\gamma\le 1\), our result means that either \(K\) is a singleton, or \(K=\{0,\pi\}\), or \(K=\{\pi/2,3\pi/2\}\). Thus \(\Omega_0\) has at most two connected components. In the two-branch cases, either
    \[
        \Omega_0
        =
        \left\{
            |x_2|<|x_1|\tan\beta
        \right\},
        \qquad
        \partial\Omega_0
        =
        \left\{
            x_2=\pm x_1\tan\beta
        \right\},
    \]
    or
    \[
        \Omega_0
        =
        \left\{
            |x_2|>|x_1|\cot\beta
        \right\},
        \qquad
        \partial\Omega_0
        =
        \left\{
            x_2=\pm x_1\cot\beta
        \right\}.
    \]
\end{remark}

In the last part of this section, we present two measure estimates for the free boundary \(\partial\Omega_{\mathbf{u}}\) near the contact set \(\Sigma\) and the non-degenerate contact set \(\Gamma_{\Sigma,\rm ND}(\mathbf{u})\). The first estimate is a logarithmic bound on the \((n-1)\)-dimensional Hausdorff measure of \(\partial\Omega_{\mathbf{u}}\) in a neighborhood of \(\Sigma\). The second estimate is a bound on the \(n\)-dimensional Lebesgue measure of a tubular neighborhood of \(\Gamma_{\Sigma,\rm ND}(\mathbf{u})\), which implies that \(\Gamma_{\Sigma,\rm ND}(\mathbf{u})\) is countably \((n-2)\)-rectifiable. We now state these two results. Recall that for a set \(E\subset\mathbb R^n\) and \(\rho>0\), we write the \(\rho\)-neighborhood of the set \(E\subset\R^n\) as 
\[
    N_\rho(E):=\{x\in\mathbb R^n:\operatorname{dist}(x,E)<\rho\}.
\]

\begin{maintheorem}{V}[Logarithmic measure estimate near \(\Sigma\)]\label{thm:measure-GammaSigma}
  Let \(\mathbf{u}\) be an \(\varepsilon_0\)-local minimizer of \(J\) in \(\mathbb{A}\). Then for every \(K\subset\subset D\), there exist constants \(r_K>0\) and \(C_K<\infty\) such that for every \(x_0\in K\cap\Sigma\) and every \(0<\rho<R<r_K\), 
    \[
        \cH^{n-1}\bigl(
            \partial\Omega_{\mathbf{u}}\cap (B_R(x_0)\setminus \cN_\rho(\Sigma))
        \bigr)\le C_K R^{n-1}\left( 1+\log\frac{R}{\rho} \right).
    \]
\end{maintheorem}

\begin{maintheorem}{VI}[Rectifiability of the non-degenerate contact set]\label{thm:measure-GammaND}
    Let \(\mathbf{u}\) be an \(\varepsilon_0\)-local minimizer of \(J\) in \(\mathbb{A}\). Then, for every \(K\subset\subset D\), there are constants \(C_K<\infty\) and \(\rho_K>0\) such that
    \[
        \cL^n(\cN_\rho(\Gamma_{\Sigma,\rm ND}(\mathbf{u})\cap K))
        \le C_K \rho^2
        \qquad\text{for every }0<\rho<\rho_K.
    \]
    Consequently,
    \[
        \mathcal H^{n-2}(\Gamma_{\Sigma,\rm ND}(\mathbf{u})\cap K)\le C_K,
    \]
    and \(\Gamma_{\Sigma,\rm ND}(\mathbf{u})\cap K\) is countably \((n-2)\)-rectifiable. In particular, if \(n=2\), then \(\Gamma_{\Sigma,\rm ND}(\mathbf{u})\cap K\) is finite.
\end{maintheorem}

\subsection{Remarks on branching points}

The terminology ``branching point'' is borrowed from the scalar two-phase Bernoulli problem. In that setting, one minimizes
\[
    J_{\rm tp}(u;D)
    :=
    \int_D \left( |\nabla u|^2+\lambda_+^2\chi_{\{u>0\}} +\lambda_-^2\chi_{\{u<0\}}   \right) 
    \,\dx{x},
\]
where \(\lambda_\pm>0\) are positive constants. A point \(x_0\in\partial\{u>0\}\cap\partial\{u<0\}\) is called a branching point if  $\cL^n\left( B_r(x_0)\cap\{u=0\}\right)>0$  for every \(r>0\), see \cite[Section~1.2]{DSV21}. Thus, for the two-phase Bernoulli problem, branching describes the transition between a region in which the positive and negative phases meet directly and a region in which they are separated by a cusp-like zero phase.

In the breakthrough work \cite{DSV21}, De Philippis, Spolaor, and Velichkov established a complete regularity theory near every two-phase point, including the branching points \cite[Theorem~1.1]{DSV21}. Moreover, every two-phase blow-up is a one-homogeneous two-plane solution
\[
    H_{\alpha,e}(x)
    =
    \alpha(x\cdot e)^+
    -\beta(x\cdot e)^-,
    \quad
    \alpha^2-\beta^2=\lambda_+^2-\lambda_-^2,
    \quad
    \alpha\ge\lambda_+,\quad \beta\ge\lambda_-,
\]
and the blow-up is unique \cite[Lemma~2.2 and Lemma~4.1]{DSV21}. Consequently, although the zero phase exists at every finite scale around a branching point, it disappears in the first-order blow-up in the sense that the two phases occupy complementary half-spaces.

\begin{remark}[The multi-phase result of Siclari--Velichkov]
    During the preparation of this paper, the authors noticed a recent interesting result on possible branching in vectorial problems studied by Siclari and Velichkov \cite{SV26b}. They divide the components of a vector-valued map \(\mathbf{v}\) into prescribed groups \(v_1,\ldots,v_m\), require the supports of distinct groups to be disjoint, and minimize
    \[
        \mathcal J_{\boldsymbol\Lambda}(\mathbf{v};D)
        =
        \int_D|\nabla \mathbf{v}|^2\,\dx{x}
        +\sum_{\ell=1}^m
            \Lambda_\ell|\Omega_{\mathbf{v}_\ell}|,
        \qquad \Lambda_\ell>0.
    \]
    Components belonging to the same group may have overlapping supports. A common boundary point of two different groups is called a two-phase point if the vacuum set \(\{|\mathbf{v}|=0\}\) is absent from some neighborhood, and a branching point if every neighborhood meets the nodal sets, see \cite[Definition~1.1]{SV26b}. This result is highly insightful, as it provides an alternative approach to studying variants of the cooperative system introduced in \cite{CSY18}.
\end{remark}

\subsection{Plan of the paper}

The paper is organized as follows. In Section \ref{sec:C1Dini-geometry}, we introduce the \(C^{1,\rm Dini}\) geometry of the contact set \(\Sigma\) and prove some basic geometric properties. In Section \ref{sec:behavior-near-contact-points}, we study the behavior of local minimizers near contact points as preliminary results for the blow-up analysis. In Section \ref{sec:measure-estimate}, using the growth estimates and non-degeneracy away from the contact set, we prove the logarithmic measure estimate of the free boundary near \(\Sigma\), namely Theorem \ref{thm:measure-GammaSigma}. In Section \ref{sec:contact-free-boundary-structure}, we establish a Weiss-type monotonicity formula in Lemma \ref{lem:Weiss-monotonicity}, prove the compactness and convergence of blow-up sequence in Proposition \ref{prop:blow-up-convergence}. In this section, we complete the proof of Theorem \ref{thm:main}, Theorem \ref{thm:structure-ND}, Theorem \ref{thm:structure-mp-br}, Theorem \ref{thm:blow-up-ND}, Corollary \ref{cor:blow-up-classification}, Corollary \ref{cor:blow-up-branching}, and Theorem \ref{thm:planar-blow-up-classification}. Finally in the section \ref{sec:quantitative-stratification}, we prove the rectifiability of the non-degenerate contact set \(\Gamma_{\Sigma,\rm ND}(\mathbf{u})\), namely Theorem \ref{thm:measure-GammaND}, by using the quantitative stratification method developed in \cite{NV17,EE19}.

\section{Preliminaries}\label{sec:C1Dini-geometry}

In this section, we gather basic geometric facts about the \(C^{1,\mathrm{Dini}}\) hypersurface \(\Sigma\). These results are independent of the minimization problem and its local minimizers, and will be used in our subsequent analysis. The material in this section is largely based on \cite{McC24}. All definitions and properties below are purely geometric, depending only on the Dini geometry of \(\Sigma\) and the Bernoulli weight function \(Q(x)=\msf d_\Sigma^\gamma(x)\). We begin by recalling the notion of a Dini modulus, the foundational regularity concept that underpins all subsequent definitions.

\begin{definition}[Dini modulus]\label{def.Dini-modulus}
    A non-decreasing function \(\omega:[0,r_0]\to[0,\infty)\) is a Dini modulus if \(\omega(0)=0\) and 
    \[
        \int_0^{r_0} \frac{\omega(t)}{t}\,\dx{t}<\infty.
    \]
\end{definition}

With this in place, we next define closed embedded \(C^{1,\rm Dini}\) hypersurfaces, the core geometric class of our problem. Recall that a subset \(\Sigma\subset\R^n\) is a closed embedded hypersurface if it is an embedded \((n-1)\)-dimensional submanifold of \(\R^n\) that is closed in the standard Euclidean topology.

\begin{definition}[\(C^{1,\mathrm{Dini}}\) hypersurface]\label{def.C1Dini-hypersurface}
    A nonempty closed embedded \(C^1\) hypersurface \(\Sigma\subset\R^n\) is of class \(C^{1,\rm Dini}\) if, at every \(x\in \Sigma\), there exists an orthonormal coordinate system centered at \(x\) with \(T_x\Sigma=\R^{n-1}\times\{0\}\) such that \(\Sigma\) admits the local graph representation
    \[
        \Sigma\cap\left( B_{r_x}'(0)\times(-r_x,r_x) \right)
        =
        \bigl\{
            (y,f_x(y)):y\in B_{r_x}'(0)
        \bigr\},
    \]
    for some \(r_x>0\) and \(f_x\in C^1(B_{r_x}'(0))\) satisfying \(f_x(0)=\nabla f_x(0)=0\), and 
     \[ 
        |\nabla f_x(y)-\nabla f_x(z)| \le \omega_x(|y-z|) \quad\text{ for every }y,z\in B_{r_x}'(0),
    \] 
    where \(\omega_x\) is a Dini modulus on \([0,2r_x]\). If the same radius \(r_\Sigma\) and modulus \(\omega_\Sigma\) can be used at every \(x\in \Sigma\), we call \(\Sigma\) a \((r_\Sigma,\omega_\Sigma)\)-\(C^{1,\rm Dini}\) hypersurface.
\end{definition}

\begin{remark}
    All applications of the quantitative geometric estimates below are local. On each fixed compact portion of \(\Sigma\), a finite graph covering provides a common radius \(r_\Sigma>0\) and a common Dini modulus \(\omega_\Sigma\), after reducing the radius. The constants and admissible radii in these local applications are understood to depend on this choice.
\end{remark}

\begin{remark}
    In Definition \ref{def.C1Dini-hypersurface}, \(B_r'(0)\) denotes the open ball of radius \(r\) in \(\R^{n-1}\). We also write \(B_r^{T_x\Sigma}(0):=\{y\in T_x\Sigma:|y|<r\}\) for the ball in the tangent space, and define the corresponding cylinder in the original coordinates by
    \[
        \operatorname{Cyl}_x(r) := \bigl\{ x + y + t\nu_x : y\in B_r^{T_x\Sigma}(0),\ |t| < r \bigr\}.
    \]
    Under the orthonormal coordinate system chosen in Definition \ref{def.C1Dini-hypersurface}, \(\operatorname{Cyl}_x(r_x)\) coincides exactly with \(B'_{r_x}(0)\times(-r_x,r_x)\).
\end{remark}

We then introduce the projection mapping onto \(\Sigma\), a standard geometric tool used repeatedly in our later analysis. 

\begin{definition}[Projection]
    For a nonempty closed hypersurface \(\Sigma\subset\R^n\), define the nearest-point projection
    \[
        \Pi_\Sigma(x) := \bigl\{ p\in\Sigma : |x-p| = \msf d_\Sigma(x) \bigr\}, \qquad x\in\R^n.
    \]
    This set is always nonempty. In the singleton case, the unique point is denoted \(\pi_\Sigma(x)\).
\end{definition}

\begin{remark}
    At any point \(x\notin\Sigma\) at which \(\msf d_\Sigma(x)\) is differentiable, the nearest-point projection is uniquely defined, and we have 
    \[
        \nabla \msf d_\Sigma(x)
        =
        \frac{x-\pi_\Sigma(x)}{|x-\pi_\Sigma(x)|}.
    \]
    As \(\msf d_\Sigma\) is a Lipschitz function with Lipschitz constant \(1\), this identity holds \(\cL^n\)-a.e. on \(\R^n\setminus\Sigma\).
\end{remark}

We conclude this section by collecting several basic geometric properties of \(C^{1,\rm Dini}\) hypersurfaces.

\begin{proposition}[Basic properties of a \(C^{1,\mathrm{Dini}}\) hypersurface]\label{prop:C1Dini-properties}
    Let \(\Sigma\subset\R^n\) be a \((r_\Sigma,\omega_\Sigma)\)-\(C^{1,\rm Dini}\) hypersurface, and set \(Q (x)= \msf d_\Sigma^\gamma\) for some \(\gamma>0\). Write \(L := \omega_\Sigma(r_\Sigma)\). Then the following properties hold.
    \begin{enumerate}[label=(\roman*)]
        \item \label{item:projection-estimate} For every \(x_0\in\Sigma\), every \(0<r\le r_\Sigma/2\), all \(x\in B_r(x_0)\setminus\Sigma\), and any \(p\in\Pi_\Sigma(x)\), 
        \begin{equation}\label{eq.projection-estimate}
            \left|
                (p-x_0)\cdot\frac{x-p}{|x-p|}
            \right|
            \le
            2r\,\omega_\Sigma(2r).
        \end{equation}
        \item \label{item:neighborhood-measure} For every \(x_0\in\Sigma\) and \(0<\rho\le R\le r_\Sigma/2\),
        \begin{equation}\label{eq.neighborhood-measure}
            \cL^n\bigl(N_\rho(\Sigma)\cap B_R(x_0)\bigr)
            \le
            C(n,L)\,\rho\,R^{n-1},
        \end{equation}
         where \(N_\rho(\Sigma)=\{x\in\R^n:\msf d_\Sigma(x)<\rho\}\) is the \(\rho\)-neighborhood of \(\Sigma\).
        \item \label{item:Q2-regularity} The Bernoulli weight function satisfies \(Q^2(x) \in C_{\rm loc}^{0,\min\{2\gamma,1\}}(\R^n)\). Moreover, at every point \(x\notin\Sigma\) where \(\msf d_\Sigma\) is differentiable,
        \[
            \nabla (Q^2(x))
            =
            2\gamma \msf d_\Sigma^{2\gamma-1}
            \,\frac{x-\pi_\Sigma(x)}{|x-\pi_\Sigma(x)|}.
        \]
        \item \label{item:gradient-integrability} For all \(x_0\in\Sigma\) and \(0<R\le r_\Sigma/2\),
        \begin{equation}\label{eq.gradient-integrability}
            \int_{B_R(x_0)} \msf d_\Sigma^{2\gamma-1}\,\dx{x} \le C(n,\gamma,L) R^{n+2\gamma-1}.
        \end{equation}
    \end{enumerate}
\end{proposition}

\begin{proof}
    By translation and rotation, we may assume without loss of generality that \(x_0 = 0\) and \(T_0\Sigma = \R^{n-1}\times\{0\}\). Then \(\Sigma\) admits the local graph representation
    \begin{equation}\label{eq.local-graph-representation}
        \Sigma\cap \operatorname{Cyl}_0(r_\Sigma) = \{(y,f(y)):y\in B_{r_\Sigma}'(0)\},
    \end{equation}
    with \(f(0)=\nabla f(0)=0\), and 
    \begin{equation}\label{eq.Dini-gradient-modulus}
        \|\nabla f\|_{L^\infty(B_{r_\Sigma}'(0))}\le L,\quad 
        |\nabla f(y)-\nabla f(z)|\le \omega_\Sigma(|y-z|).
    \end{equation}
    If \(x\in B_{r_\Sigma/2}(0)\) and \(p\in\Pi_\Sigma(x)\), then \(|p| \le |x| + \msf d_\Sigma(x) \le 2|x| < r_\Sigma\), so \(p\) lies entirely within the graph region above.

    To prove part \ref{item:projection-estimate}, write \(p = (y, f(y))\) and set \(\nu = (x-p)/|x-p|\). Since \(p\) is a nearest point, \(\nu\) is a unit normal to \(\Sigma\) at \(p\). By the fundamental theorem of calculus, 
    \[
        \begin{aligned}
            |p\cdot\nu|&\le |f(y)-\nabla f(y)\cdot y|\\ 
            &=\left|
                \int_0^1\left( \nabla f(ty)-\nabla f(y) \right)\cdot y\,\dx{t}
            \right|\\ 
            &\le |y|\omega_\Sigma(|y|)\le 2r\omega_\Sigma(2r),
        \end{aligned}
    \]
    where the last step uses \eqref{eq.Dini-gradient-modulus} and \(|y| \le |p| \le 2r\).
    
    For part \ref{item:neighborhood-measure}, let \(x = (x', x_n) \in N_\rho(\Sigma) \cap B_R(0)\) and pick \(p = (y, f(y)) \in \Pi_\Sigma(x)\). Then 
    \[
        |x_n-f(x')|\le |x_n-f(y)|+L|y-x'|\le (1+L)|x-p|<(1+L)\rho.
    \]
    This implies 
    \[
        N_\rho(\Sigma)\cap B_R(0)
        \subset
        \left\{
            (x',x_n):
            |x'|<R,\ |x_n-f(x')|<(1+L)\rho
        \right\}.
    \]
    Thus each vertical section of \(N_\rho(\Sigma)\cap B_R(0)\) has length at most \(2(1+L)\rho\). Integration over \(B_R'(0)\) yields estimate \eqref{eq.neighborhood-measure}, completing the proof of part \ref{item:neighborhood-measure}.

    Part \ref{item:Q2-regularity} follows from the Lipschitz continuity of \(\msf d_\Sigma\), and the regularity of \(t\mapsto t^{2\gamma}\) on \([0,\infty)\), so we omit the details.

    Finally, note that \(\msf d_\Sigma(x) \le R\) for all \(x\in B_R(0)\). When \(\gamma \ge 1/2\), we have \(2\gamma - 1 \ge 0\), so \(\msf d_\Sigma^{2\gamma-1} \le R^{2\gamma-1}\) on \(B_R(0)\), and estimate \eqref{eq.gradient-integrability} is immediate. When \(0 < \gamma < 1/2\), set \(\beta := 1 - 2\gamma \in (0,1)\). By Tonelli's theorem and estimate \eqref{eq.neighborhood-measure},
    \[
        \begin{aligned}
            \int_{B_R(0)}\msf d_\Sigma^{-\beta}\,\dx{x}={}&R^{-\beta}\cL^n(B_R(0))+\beta\int_0^R t^{-\beta-1}\cL^n\left( N_t(\Sigma)\cap B_R(0) \right)\,\dx{t}\\ 
            \le{}&C(n)R^{n-\beta} + C(n,L)\beta R^{n-1}\int_0^R t^{-\beta}\,\dx{t}\\ 
            \le{}&C(n,\gamma,L)R^{n-\beta}=C(n,\gamma,L)R^{n+2\gamma-1}.
        \end{aligned}
    \]
    This proves \eqref{eq.gradient-integrability}, and the proof of Proposition \ref{prop:C1Dini-properties} is complete.
\end{proof}

\section{Behavior of the local minimizer near contact points}\label{sec:behavior-near-contact-points}

In this section, we turn to the quantitative behavior of the local minimizer \(\mathbf{u}\) around contact points \(x_0\in \Gamma_\Sigma(\mathbf{u})\). Specifically, we derive growth, non-degeneracy, and gradient estimates near the contact set. We first recall the variational identities for local minimizers, and fix a uniform scale to validate our comparison arguments on balls.

\begin{lemma}\label{lem:EL}
    Let \(\mathbf{u}=(u_1,\dots,u_m)\) be an \(\varepsilon_0\)-local minimizer of \(J\) in \(\mathbb A\). Then:
    \begin{enumerate}[label=(\roman*)]
        \item Each component \(u_i\) is weakly subharmonic in \(D\):
        \[
            \int_D\nabla u_i\cdot\nabla \eta\,\dx{x}\le 0,
        \]
        for every nonnegative \(\eta\in C_0^\infty(D)\).
        \item \label{item:innerv}
        For every \(\Phi\in C_0^\infty(D;\R^n)\), 
        \begin{equation}\label{eq:inner-dist}
            \begin{aligned}
                \int_D
                \Bigg[
                &\left(|\nabla \mathbf{u}|^2+\msf d_\Sigma^{2\gamma}(x)\chi_{\Omega_{\mathbf{u}}}\right)\operatorname*{div}\Phi
                -2\sum_{i=1}^m \langle D\Phi\,\nabla u_i,\nabla u_i\rangle\\ 
                &+
                2\gamma \msf d_\Sigma^{2\gamma-1}(x)\chi_{\Omega_{\mathbf{u}}}
                \frac{x-\pi_\Sigma(x)}{|x-\pi_\Sigma(x)|}\cdot\Phi
                \Bigg]\,\dx{x}
                =0.
            \end{aligned}
        \end{equation}
        The projection formula in the last term is understood a.e. outside the hypersurface \(\Sigma\).
    \end{enumerate}
\end{lemma}

\begin{proof}
    The first assertion follows from one-sided perturbations \(v_i = (u_i - t\eta)_+\) (with all other components unchanged), which do not increase the phase term, upon dividing by \(t>0\) and letting \(t\to 0^+\). The second is obtained by differentiating the energy of \(\mathbf{u}\circ(\operatorname{Id}+t\Phi)^{-1}\) at \(t=0\). For small \(|t|\) these competitors satisfy the requirement in \eqref{eq.local-min}, and differentiability of the weight is justified by the next remark. 
\end{proof}

\begin{remark}
    By part \ref{item:gradient-integrability} of Proposition \ref{prop:C1Dini-properties}, we have \(\msf d_\Sigma^{2\gamma-1} \in L^1_{\text{loc}}(D)\). Applying the chain rule to the regularized functions \((\msf d_\Sigma + \varepsilon)^{2\gamma}\) and passing to the limit as \(\varepsilon \to 0^+\) shows that \(Q^2(x) \in W^{1,1}_{\text{loc}}(D)\). In particular, the weight term in \eqref{eq:inner-dist} is locally integrable.
\end{remark}

The following lemma establishes a common scale that validates both harmonic replacement and truncation arguments in small balls.

\begin{lemma}[Comparison at small scales]\label{lem:standard-scale}
    Let \(\mathbf{u}\) be an \(\varepsilon_0\)-local minimizer which satisfies \eqref{eq.local-min}, and set \(Q_D := \sup_D \msf d_\Sigma^\gamma < \infty\). There exists \(\bar r > 0\) such that for every ball \(B_\rho(x) \subset\subset D\) with \(0 < \rho \le \bar r\), we have
    \begin{equation}\label{eq:ball-minimality}
        J(\mathbf{u};B_\rho(x))\le J(\mathbf{v};B_\rho(x))
    \end{equation}
    for all \(\mathbf{v}\in\mathbb{A}\) with \(\mathbf{v}=\mathbf{u}\) a.e. in \(D\setminus B_\rho(x)\). In particular, harmonic replacement is valid on such balls and one may take any \(\bar r> 0\) satisfying 
    \begin{equation}\label{eq.r_0}
        \sup_{\substack{B_\rho(x)\subset\subset D\\ 0<\rho\le\bar r}}\left[ 
            4\int_{B_\rho(x)}|\nabla\mathbf{u}|^2\,\dx{y} + (2Q_D^2+1)\cL^n(B_\rho(x))
        \right]<\varepsilon_0.
    \end{equation}
\end{lemma}

\begin{proof}
    Such a radius exists by the absolute continuity of \(\int|\nabla\mathbf{u}|^2\,\dx{x}\). For brevity, write \(B = B_\rho(x)\). It suffices to consider only competitors satisfying \(J(\mathbf{v}; B) \le J(\mathbf{u}; B)\). For any such competitor,
    \[
        \int_B|\nabla\mathbf{v}|^2\,\dx{y}\le \int_B |\nabla\mathbf{u}|^2\,\dx{y}+Q_D^2\cL^n(B),
    \]
    and hence 
    \[
        \begin{aligned}
            &\|\nabla(\mathbf{u}-\mathbf{v})\|_{L^2(D)}^2+\|\chi_{\Omega_{\mathbf{u}}}-\chi_{\Omega_{\mathbf{v}}}\|_{L^1(D)}\\ 
            \le{}& 4\int_B|\nabla\mathbf{u}|^2\,\dx{y}+(2Q_D^2+1)\cL^n(B)<\varepsilon_0.
        \end{aligned}
    \]
    Thus \eqref{eq:ball-minimality} follows from local minimality of \(\mathbf{u}\).
\end{proof}

\begin{remark}\label{rem:continuity}
    Let \(B:=B_\rho(x)\) be such a ball, and let \(\mathbf{h}\) be the componentwise harmonic replacement of \(\mathbf{u}\) in \(B\), extended by \(\mathbf{u}\) outside \(B\). Comparison and Dirichlet orthogonality yield
    \begin{equation}\label{eq.esm1}
        \begin{aligned}
            \int_D|\nabla(\mathbf{u}-\mathbf{h})|^2\,\dx{y}
            &\le
            \int_{B\setminus\Omega_{\mathbf{u}}}
                \msf d_\Sigma(y)^{2\gamma}\,\dx{y}\\
            &\le
            \bigl(\msf d_\Sigma(x)+\rho\bigr)^{2\gamma}
            \cL^n(B\setminus\Omega_{\mathbf{u}}).
        \end{aligned}
    \end{equation}
    The harmonic replacement argument of \cite[Lemma 2]{CSY18}, using only an \emph{upper bound} for the Bernoulli weight function \(Q(x)\), gives \(\mathbf{u}\in C_{\rm loc}^\beta(D;\R^m)\) for all \(0<\beta<1\). We henceforth take this continuous representative, so \(\Omega_{\mathbf{u}}\) is open.
\end{remark}

We now present the main results of this section. When the free boundary makes contact with the prescribed hypersurface \(\Sigma\), the growth behavior of the minimizer \(\mathbf{u}\) is constrained by the geometry of \(\Sigma\), a scenario not treated in \cite{CSY18}. Our main findings are summarized in the following proposition, which establishes both growth estimates and non-degeneracy properties for \(\mathbf{u}\) near the contact set.

\begin{proposition}[Growth and non-degeneracy near contact points]\label{prop:optimal-growth-nondegeneracy}
    Let \(\mathbf{u}\) be an \(\varepsilon_0\)-local minimizer of \(J\) in the class \(\mathbb A\). Fix \(0<r_0\le\bar r\), where \(\bar r\) is given in Lemma \ref{lem:standard-scale}. Then the following hold.
    \begin{enumerate}[label=(\roman*)]
        \item \label{item:growth} For every \(x_0\in\partial\Omega_{\mathbf{u}}\cap D\) and every \(0<r\le r_0/2\) satisfying \(2r<\dist(x_0,\partial D)\),
        \begin{equation}\label{eq:boundary-growth}
            \sup_{B_r(x_0)}|\mathbf{u}|\le Cr\left( \msf d_\Sigma(x_0) + 2r \right)^\gamma,
        \end{equation}
        for some constant \(C=C(n,m,\gamma)>0\). In particular, at a contact point \(x_0\in\Gamma_\Sigma(\mathbf{u})\), this reduces to
        \begin{equation}\label{eq:optimal-growth}
            \sup_{B_r(x_0)}|\mathbf{u}|\le Cr^{1+\gamma}.
        \end{equation}
        \item \label{item:nondegeneracy} For every \(s\in(0,1)\), there exists a constant \(c=c(n,m,s)>0\) with the following property: if \(B_r(x)\subset\subset D\) with \(0<r\le r_0\) and \(B_{sr}(x)\cap\Omega_{\mathbf{u}}\neq\varnothing\), then
        \begin{equation}\label{eq:boundary-nondegeneracy}
            \sup_{B_r(x)}|\mathbf{u}|\ge cr\left( \msf d_\Sigma(x)-sr \right)_+^\gamma.
        \end{equation}
        Furthermore, fix \(0<\sigma<s<1/2\), let \(x_0\in\Gamma_\Sigma(\mathbf{u})\), and take \(0<r\le \frac{1}{2}\min\{r_0,\dist(x_0,\partial D)\}\). If 
        \[
            B_{r/2}(x_0)\cap \Omega_{\mathbf{u}}\cap\{\msf d_\Sigma\ge sr\}\ne\varnothing,
        \]
        then 
        \begin{equation}\label{eq:boundary-nondegeneracy1}
            \sup_{B_r(x_0)\cap \{ \msf d_\Sigma\ge \sigma r \}}|\mathbf{u}|\ge cr^{1+\gamma},
        \end{equation}
        where \(c=c(n,m,\gamma,s,\sigma)>0\).
    \end{enumerate}
\end{proposition}

\begin{proof}
    The proof follows the arguments in \cite[Theorems 2 and 3]{CSY18}, with key adjustments for the degeneracy of the Bernoulli weight at contact points. 

    For part \ref{item:growth}, let \(\mathbf{h}\) be the componentwise harmonic replacement of \(\mathbf{u}\) in \(B_\rho(x_0)\) with \(0<\rho\le 2r\). It follows from \eqref{eq.esm1} and the capacity estimates from \cite[Lemmas 3 and 5]{CSY18} that
    \begin{equation}\label{eq.esm2}
        \begin{aligned}
            &\cL^n(B_\rho(x_0)\setminus\Omega_{\mathbf{u}})\left( \frac 1\rho\barint_{\partial B_\rho(x_0)}|\mathbf{u}|\,\dH{n-1} \right)^2\\ 
            \le{}& C\int_{B_\rho(x_0)}|\nabla(\mathbf{u}-\mathbf{h})|^2\,\dx{y}\\ 
            \le{}& C\left( \msf d_\Sigma(x_0) + \rho \right)^{2\gamma}\cL^n(B_\rho(x_0)\setminus\Omega_{\mathbf{u}}).
        \end{aligned}
    \end{equation}
    Note that \(\cL^n(B_\rho(x_0)\setminus\Omega_{\mathbf{u}}) > 0\): otherwise \eqref{eq.esm1} would force \(\mathbf{u}=\mathbf{h}\) everywhere in \(B_\rho(x_0)\), contradicting \(x_0\in\partial\Omega_{\mathbf{u}}\) by the strong maximum principle. Cancelling this factor from both sides yields
    \begin{equation}\label{eq.esm3}
        \barint_{\partial B_\rho(x_0)}|\mathbf{u}|\,\dH{n-1}\le C\rho\left( \msf d_\Sigma(x_0)+\rho \right)^\gamma.
    \end{equation}
    The subharmonicity of the nonnegative components of \(\mathbf{u}\), combined with the Poisson estimate on \(B_{2r}(x_0)\), then gives
    \[
        \sup_{B_r(x_0)}|\mathbf{u}|\le C\barint_{\partial B_{2r}(x_0)}|\mathbf{u}|\,\dH{n-1}\le Cr\left( \msf d_\Sigma(x_0) + 2r \right)^\gamma.
    \]
    At a contact point, \(\msf d_\Sigma(x_0)=0\), which reduces the estimate to \eqref{eq:optimal-growth}.

    For part \ref{item:nondegeneracy}, define
    \[
        M:=\frac 1r\sup_{B_r(x)}|\mathbf{u}|,\qquad q:=\inf_{B_{sr}(x)}\msf d_\Sigma^\gamma.
    \]
    Since \(\msf d_\Sigma\) is 1-Lipschitz, for every \(y\in B_{sr}(x)\) we have 
    \[
        \msf d_\Sigma(y)\ge \msf d_\Sigma(x)-|y-x|\ge\msf d_\Sigma(x)-sr.
    \]
    Taking the infimum over \(B_{sr}(x)\) yields
    \begin{equation}\label{eq.inf_distance-lower-bound}
        q\ge \left( \msf d_\Sigma(x) - sr \right)_+^\gamma.
    \end{equation}
    If \(q = 0\), the desired estimate holds trivially. We may therefore assume \(q > 0\). Let \(\psi_s\) be the harmonic function in \(B_1\setminus\overline{B_s}\) that vanishes on \(\partial B_s\) and equals 1 on \(\partial B_1\). Extend \(\psi_s\) by zero to \(B_s\). Define for \(y\in B_r(x)\),
    \[
        w(y):=rM\psi_s\left( \frac{y-x}{r} \right),\qquad v_i:=\min\{u_i,w\}\text{ in }B_r(x),
    \]
    and extend \(\mathbf{v}\) to equal \(\mathbf{u}\) outside \(B_r(x)\). This competitor is admissible for the minimality condition \eqref{eq:ball-minimality}. Set \(A:=B_r(x)\setminus\overline{B_{sr}(x)}\) and 
    \begin{equation}\label{eq:energy-left}
        E:=\int_{B_{sr}(x)}\left( |\nabla\mathbf{u}|^2 + \msf d_\Sigma^{2\gamma}\chi_{\Omega_{\mathbf{u}}} \right)\,\dx{y}>0.
    \end{equation}
    Positivity follows from \(q > 0\) and the assumption \(B_{sr}(x)\cap\Omega_{\mathbf{u}} \neq \varnothing\). Since \(\mathbf{v} = 0\) in \(B_{sr}(x)\) and \(\chi_{\Omega_{\mathbf{v}}}=\chi_{\Omega_{\mathbf{u}}}\) a.e. in \(A\), the minimality of \(\mathbf{u}\) and integration by parts give 
    \begin{equation}\label{eq:E-bound-boundary}
        \begin{aligned}
            E&\le\sum_{i=1}^m\int_A\left( |\nabla v_i|^2-|\nabla u_i|^2 \right)\,\dx{y}\\ 
            &\le-2\sum_{i=1}^m\int_A\nabla w\cdot\nabla(u_i-w)_+\,\dx{y}\\ 
            &=2\sum_{i=1}^m\int_{\partial B_{sr}(x)}u_i\nabla w\cdot\nu\,\dH{n-1}\\ 
            &\le C(n,m,s)M\int_{\partial B_{sr}(x)}|\mathbf{u}|\,\dH{n-1},
        \end{aligned}
    \end{equation}
    where \(\nu\) is the outward normal to \(B_{sr}(x)\). 
    
    Combining the trace inequality, the pointwise bound \(|\mathbf{u}| \le rM\), the lower bound \(q^2 \cL^n(B_{sr}(x)\cap\Omega_{\mathbf{u}}) \le E\), and the Cauchy–Schwarz inequality, we arrive at
    \begin{equation}\label{eq:trace-inequality}
        E\le CM\int_{B_{sr}(x)\cap\Omega_{\mathbf{u}}}(M+|\nabla\mathbf{u}|)\,\dx{y}\le C\frac Mq\left( \frac Mq+1 \right)E.
    \end{equation}
    Since \(E>0\), we may divide both sides by \(E\) and rearrange to obtain
    \begin{equation}\label{eq:M-lower-bound-inf-distance}
        M\ge c(n,m,s)q.
    \end{equation}
    Combined with \eqref{eq.inf_distance-lower-bound}, this establishes estimate \eqref{eq:boundary-nondegeneracy}.

    Finally, choose \(p\in B_{r/2}(x_0)\cap\Omega_{\mathbf{u}}\) with \(\msf d_\Sigma(p)\ge sr\), and set \(\eta:=\tfrac 14\min\{1,s-\sigma\}\), \(\rho:=\eta r\). Then 
    \[
        B_\rho(p)\subset B_r(x_0)\cap\{\msf d_\Sigma\ge \sigma r\}.
    \]
    Applying \eqref{eq:boundary-nondegeneracy} in \(B_\rho(p)\) with \(s=\tfrac 12\) gives
    \[
        \sup_{B_\rho(p)}|\mathbf{u}|\ge c\rho\left( \msf d_\Sigma(p) - \frac\rho 2 \right)^\gamma,
    \]
    proving \eqref{eq:boundary-nondegeneracy1}.
\end{proof}

We now introduce a concept that will be used repeatedly throughout the subsequent sections.

\begin{definition}[Standard scale]\label{def:standard-scale}
    A radius \(r_0 > 0\) is called a \emph{standard scale} for \(\mathbf{u}\) if for every ball \(B_\rho(x) \subset\subset D\) with \(0 < \rho \le r_0\), the inequality
    \[
        J(\mathbf{u};B_\rho(x))\le J(\mathbf{v};B_{\rho}(x))
    \]
    holds for all \(\mathbf{v} \in \mathbb{A}\) such that \(\mathbf{v} = \mathbf{u}\) a.e. in \(D \setminus B_\rho(x)\). 
\end{definition}

\begin{remark}
    The existence of such radii follows from Lemma \ref{lem:standard-scale}. In particular, condition \eqref{eq.r_0} is sufficient. Moreover, any positive radius smaller than a standard scale is also a standard scale.
\end{remark}

The growth estimate, combined with interior harmonic estimates, yields uniform gradient control.

\begin{corollary}[Lipschitz regularity for local minimizers]\label{cor:gradient-estimate}
    Let \(\mathbf{u}\) be an \(\varepsilon_0\)-local minimizer, and let \(r_0\) be a standard scale. Then 
    \begin{enumerate}[label=(\roman*)]
        \item \label{item:harmonic} The set \(\Omega_{\mathbf{u}}\) is open, and each component \(u_i\) is harmonic in \(\Omega_{\mathbf{u}}\).
        \item \label{item:gradient} For every \(x_0\in\partial \Omega_{\mathbf{u}}\cap D\) and every \(0<r\le\tfrac 14\min\{r_0,\dist(x_0,\partial D)\}\),
        \begin{equation}\label{eq:Lipschitz}
            \|\nabla\mathbf{u}\|_{L^\infty(B_r(x_0))}\le C(n,m,\gamma)\left( \msf d_\Sigma(x_0) + r \right)^\gamma.
        \end{equation}
    \end{enumerate}
\end{corollary}

\begin{proof}
    The openness of \(\Omega_{\mathbf{u}}\) follows from the continuity of \(\mathbf{u}\). In any sufficiently small ball compactly contained in \(\Omega_{\mathbf{u}}\), estimate \eqref{eq.esm1} from Remark \ref{rem:continuity} forces \(\mathbf{u}=\mathbf{h}\), which proves harmonicity of each component.

    Now take any \(y\in B_r(x_0)\cap\Omega_{\mathbf{u}}\) and set \(\delta:=\dist(y, \partial\Omega_{\mathbf{u}})\). If \(\delta\ge r/4\), then \(B_{r/4}(y)\subset\Omega_{\mathbf{u}}\cap B_{3r/2}(x_0)\). Since \(3r<\dist(x_0,\partial D)\) and \(3r/2\le r_0/2\), interior harmonic estimates and \eqref{eq:boundary-growth} give
    \[
        \begin{aligned}
            |\nabla\mathbf{u}(y)|
            &\le\frac Cr\sup_{B_{r/4}(y)}|\mathbf{u}|
            \le\frac Cr\sup_{B_{3r/2}(x_0)}|\mathbf{u}|\\
            &\le C\bigl(\msf d_\Sigma(x_0)+r\bigr)^\gamma.
        \end{aligned}
    \]
    If \(\delta < r/4\), choose \(z\in\partial\Omega_{\mathbf{u}}\) with \(|z-y| = \delta\). Then 
    \[
        |\nabla\mathbf{u}(y)|\le\frac C\delta\sup_{B_{\delta/2}(y)}|\mathbf{u}|\le\frac C\delta\sup_{B_{2\delta}(z)}|\mathbf{u}|.
    \]
    Since \(|z-x_0| \le 5r/4\), the radius \(2\delta\) satisfies the hypotheses of \eqref{eq:boundary-growth} at \(z\). Thus
    \begin{equation}\label{esm.grad1}
        \begin{aligned}
            |\nabla\mathbf{u}(y)|&\le \frac C\delta\sup_{B_{2\delta}(z)}|\mathbf{u}|\le C\left( \msf d_\Sigma(z) + 4\delta \right)^\gamma\\ 
            &\le C\left( \msf d_\Sigma(x_0)+r \right)^\gamma,
        \end{aligned}
    \end{equation}
    where in the last inequality, we apply the Lipschitz continuity of \(\msf d_\Sigma\) and absorb \((9/4)^\gamma\) into the constant \(C\):
    \[
        \begin{aligned}
            \msf d_\Sigma(z) + 4\delta &\le \msf d_\Sigma(x_0) + |z-x_0| +4\delta\\ 
            &\le \msf d_\Sigma(x_0) + \frac 94 r\\ 
            &\le \frac 94 \left( \msf d_\Sigma(x_0) + r \right).
        \end{aligned}
    \]
    Finally, \(\nabla\mathbf{u} = 0\) almost everywhere on \(\{\mathbf{u} = 0\}\), so estimate \eqref{eq:Lipschitz} holds on the entire ball.
\end{proof}

\begin{remark}
    At a contact point \(x_0\in \Gamma_\Sigma(\mathbf{u})\), estimate \eqref{eq:Lipschitz} reduces to
    \begin{equation}\label{eq:Lipschitz-contact}
        \|\nabla\mathbf{u}\|_{L^\infty(B_r(x_0))}\le C_0r^\gamma,
    \end{equation}
    where \(C_0=C_0(n,m,\gamma)\).
\end{remark}

Combining the gradient bound with the non-degeneracy property in Proposition \ref{prop:optimal-growth-nondegeneracy} yields interior balls of radius comparable to \(r\) where \(\mathbf{u}\) remains uniformly positive and maintains a prescribed distance from \(\Sigma\).

\begin{corollary}[Interior balls in a non-degenerate strip]\label{cor:interior-balls-strip}
    Fix \(0<\sigma<s<1/2\), let \(x_0\in \Gamma_\Sigma(\mathbf{u})\), and assume \(0<r\le\frac14\min\{r_0,\dist(x_0,\partial D)\}\). If 
    \[
        B_{r/2}(x_0)\cap \Omega_{\mathbf{u}}\cap\{\msf d_\Sigma\ge sr\}\ne\varnothing,
    \]
    then there exist constants \(\theta,c>0\) depending only on \(n,m,\gamma,s,\sigma\), and a point \(y\in \Omega_{\mathbf{u}}\), such that
    \[
        B_{\theta r}(y)\subset B_r(x_0)\cap\Omega_{\mathbf{u}}\cap\{\msf d_\Sigma\ge \sigma r\},
    \]
    and 
    \[
        |\mathbf{u}|\ge cr^{1+\gamma}\quad\text{ in }B_{\theta r}(y).
    \]
\end{corollary}

\begin{proof}
    Pick a point \(p\in B_{r/2}(x_0)\cap \Omega_{\mathbf{u}}\) with \(\msf d_\Sigma(p) \ge sr\). Set \(\eta := \frac{1}{4}\min\{1, s-\sigma\}\) and \(\rho := \eta r\). By the Lipschitz continuity of \(\msf d_\Sigma\), we have \(B_\rho(p) \subset B_r(x_0) \cap \{\msf d_\Sigma \ge \sigma r\}\). 
    
    Applying the non-degeneracy estimate \eqref{eq:boundary-nondegeneracy} to \(B_{\rho/2}(p)\) with scale parameter \(s=\frac{1}{2}\), we find a point \(y\in B_{\rho/2}(p)\) such that
    \[
        |\mathbf{u}(y)|\ge c_1r^{1+\gamma},
    \]
    where \(c_1 = c_1(n,m,\gamma,s,\sigma) > 0\). Setting \(\theta := \min\left\{\eta/4, c_1/(2C_0)\right\}\) and using the gradient bound \eqref{eq:Lipschitz-contact}, we obtain
    \[
        |\mathbf{u}(z)|\ge |\mathbf{u}(y)|-C_0r^\gamma|z-y|\ge \frac{c_1}2r^{1+\gamma} 
    \]
    for \(z\in B_{\theta r}(y)\). Since \(B_{\theta r}(y)\subset B_\rho(p)\), the proof is complete.
\end{proof}

\section{Measure estimates for the free boundary near contact points}\label{sec:measure-estimate}

In this section, we develop measure properties of the positivity set \(\Omega_{\mathbf{u}}\) and prove quantitative measure estimates for the free boundary near contact points with \(\Sigma\). Fix a standard scale \(r_0\) as in Definition \ref{def:standard-scale}. We first recall the measure estimate for balls where \(Q(x)\) admits a uniform lower bound, and then establish uniform Hausdorff measure bounds for the free boundary on dyadic distance layers near the hypersurface \(\Sigma\).

\begin{lemma}[Quantitative estimate in non-degenerate balls]\label{lem:measure-nondegenerate-ball}
    Let \(z\in\partial\Omega_{\mathbf{u}}\cap D\), suppose \(B_{2\ell}(z)\subset\subset D\setminus\Sigma\) and \(0<2\ell\le r_0\). Assume that there are positive constants \(Q_1, Q_2 > 0\) such that
    \[
        0<Q_1\le \msf d_\Sigma^\gamma \le Q_2\quad\text{ in }B_{2\ell}(z),
    \]
    and set \(\Lambda:=Q_2/Q_1\). Then
    \[
        c\ell^{n-1}\le \cH^{n-1}(\partial \Omega_{\mathbf{u}}\cap B_\ell(z))\le C\ell^{n-1},
    \]
    where \(c,C>0\) depend only on \(n,m,\Lambda\). Moreover, the singular part of the free boundary satisfies
    \[
        \cH^{n-1}\left( (\partial\Omega_{\mathbf{u}}\setminus\partial^*\Omega_{\mathbf{u}}) \cap B_\ell(z) \right)=0.
    \]
\end{lemma}

\begin{proof}
    By the standard scale property, \(\mathbf{u}\) is an absolute minimizer in \(B_{2\ell}(z)\). We consider the following rescaling and normalization:
    \[
        \widetilde{\mathbf{u}}(y):=\frac{\mathbf{u}(z+\ell y)}{\ell Q_1},\qquad \widetilde Q(y):=\frac{\msf d_\Sigma^\gamma(z+\ell y)}{Q_1}.
    \]
    Under this transformation, \(\widetilde{\mathbf{u}}\) is an absolute minimizer in \(B_2(0)\) of the normalized energy functional 
    \[
        \tilde J(\mathbf{v}):=\int_{B_2(0)}\left(|\nabla \mathbf{v}|^2+\widetilde Q^2(x)\chi_{\Omega_{\mathbf{v}}}\right)\,\dx{y}.
    \]
    Moreover, \(\widetilde Q(x)\) satisfies \(1 \le \widetilde Q(x) \le \Lambda\) in \(B_2(0)\) and \(\Omega_{\widetilde{\mathbf{u}}}=((\Omega_{\mathbf{u}}-z)/\ell)\cap B_2(0)\). Therefore, we can apply \cite[Lemma 9]{CSY18}, which gives 
    \[
        c\le \cH^{n-1}\left( \partial\Omega_{\widetilde{\mathbf{u}}}\cap B_1(0) \right)\le C,
    \]
    for some positive constants \(c\) and \(C\). Rescaling back to the original variables, and using \cite[Lemma 10]{CSY18}, we obtain the desired result.
\end{proof}

On each dyadic distance layer, the weight function has a uniformly bounded ratio between its upper and lower values. The tubular covering estimate from Section \ref{sec:C1Dini-geometry} (cf. Part \ref{item:neighborhood-measure} in Proposition \ref{prop:C1Dini-properties}) provides uniform control of the number of balls needed to cover any such layer, as the following Lemma shows.

\begin{lemma}[Dyadic layer measure bound near contact points]\label{lem:dyadic-layer-measure}
    Let \(\Sigma\) be a \((r_\Sigma,\omega_\Sigma)\)-\(C^{1,\mathrm{Dini}}\) hypersurface, let \(\mathbf{u}\) be an \(\varepsilon_0\)-local minimizer, and fix \(x_0\in \Sigma\cap D\). Suppose that \(0<R\le \frac{1}{4}\min\{r_0, r_\Sigma\}\) and \(B_{4R}(x_0)\subset\subset D\). For each integer \(j\ge 0\), define the dyadic distance layer
    \[
        A_j:=B_R(x_0)\cap\{2^{-j-1}R<\msf d_\Sigma\le 2^{-j}R\}.
    \]
    Then 
    \[
        \cH^{n-1}(\partial\Omega_{\mathbf{u}}\cap A_j)\le CR^{n-1},
    \]
    where \(C\) depends only on \(n,m,\gamma,\omega_\Sigma(r_\Sigma)\).
\end{lemma}

\begin{proof}
    Define \(\delta_j:=2^{-j}R\) and \(F_j:=\partial\Omega_{\mathbf{u}}\cap A_j\). Consider a maximal disjoint family \(\{B_{\delta_j/64}(z_\alpha)\}_{\alpha\in I_j}\) of balls centered at points of \(F_j\). By maximality,
    \[
        F_j\subset\bigcup_{\alpha\in I_j}B_{\delta_j/32}(z_\alpha).
    \]
    Each of these disjoint balls is contained in \(B_{2R}(x_0) \cap N_{2\delta_j}(\Sigma)\), where \(N_{2\delta_j}(\Sigma):=\{x\in\R^n:\msf d_\Sigma(x)<2\delta_j\}\) denotes the \(2\delta_j\)-neighborhood of \(\Sigma\). Applying the tubular neighborhood volume estimate \eqref{eq.neighborhood-measure} yields
    \[
        \#\,(I_j)\cdot \delta_j^n\le CR^{n-1}\delta_j,
    \]
    and hence
    \begin{equation}\label{eq:cardinality-estimate}
        \#\,(I_j)\le C(R/\delta_j)^{n-1}.
    \end{equation}

    On each ball \(B_{\delta_j/16}(z_\alpha)\), the distance function satisfies
    \[
        \frac 7{16}\delta_j\le \msf d_\Sigma \le\frac{17}{16}\delta_j.
    \]
    Applying Lemma \ref{lem:measure-nondegenerate-ball} with \(\ell = \delta_j/32\) and \(\Lambda = (17/7)^\gamma\), we obtain
    \[
        \cH^{n-1}\left( \partial\Omega_{\mathbf{u}}\cap B_{\delta_j/32}(z_\alpha) \right)\le C\delta_j^{n-1}.
    \]
    Summing over the covering balls and using \eqref{eq:cardinality-estimate}, we obtain 
    \[
        \begin{aligned}
            \cH^{n-1}(\partial\Omega_{\mathbf{u}}\cap A_j)&=\cH^{n-1}(F_j)\\ 
            &\le \sum_{\alpha\in I_j}\cH^{n-1}(\partial\Omega_{\mathbf{u}}\cap B_{\delta_j/32}(z_\alpha))\\ 
            &\le C\#\,(I_j)\cdot\delta_j^{n-1}\\ 
            &\le CR^{n-1},
        \end{aligned}
    \]
    which completes the proof.
\end{proof}

Summing the dyadic layer estimates established in Lemma \ref{lem:dyadic-layer-measure} immediately yields the following logarithmic measure bound for the free boundary.

\begin{corollary}\label{cor:free-boundary-measure-estimate}
    Under the assumptions of Lemma \ref{lem:dyadic-layer-measure}, for every \(0<\rho<R\), 
    \[
        \cH^{n-1}\left( \partial\Omega_{\mathbf{u}}\cap \left( B_R(x_0)\setminus N_\rho(\Sigma) \right)\right)\le CR^{n-1}\left( 1+\log\frac R\rho \right).
    \]
\end{corollary}

\begin{proof}
    Choose an integer \(k\ge 0\) such that \(2^{-k-1}R<\rho\le 2^{-k}R\). Since \(\msf d_\Sigma(x)<R\) for all \(x\in B_R(x_0)\), we have the inclusion
    \[
        B_R(x_0)\setminus N_\rho(\Sigma)\subset \bigcup_{j=0}^kA_j,
    \]
   where \(A_j\) are the dyadic distance layers defined in Lemma \ref{lem:dyadic-layer-measure}. By the preceding lemma, each layer satisfies \(\cH^{n-1}(\partial\Omega_{\mathbf{u}}\cap A_j)\le CR^{n-1}\). Summing over \(j=0,\dots,k\) and noting that \(k\le\log_2(R/\rho)\), we obtain
    \[
        \begin{aligned}
            \cH^{n-1}\left( \partial\Omega_{\mathbf{u}}\cap \left( B_R(x_0)\setminus N_\rho(\Sigma) \right) \right)&\le\sum_{j=0}^k\cH^{n-1}(\partial\Omega_{\mathbf{u}}\cap A_j)\\ 
            &\le C(k+1)R^{n-1}\\ 
            &\le CR^{n-1}\left( 1+\log_2\frac R\rho \right)\\ 
            &\le CR^{n-1}\left( 1+\log\frac R\rho \right).
        \end{aligned}
    \]
    This concludes the proof.
\end{proof}

\begin{remark}
    The estimate applies at every contact point \(x_0\in \Gamma_\Sigma(\mathbf{u})\), including degenerate contact points. Its logarithmic dependence on \(\rho\) supplies no uniform perimeter bound as the excluded neighborhood of \(\Sigma\) shrinks.
\end{remark}

\section{Structure of the contact free boundary points}\label{sec:contact-free-boundary-structure}

In this section, we investigate the fine structure of the free boundary at its contact points with the prescribed hypersurface \(\Sigma\). We begin by establishing an almost-monotonicity formula and compactness of rescalings. This gives rise to homogeneous blow-up limits and a density gap property, from which the main structural results for the contact set follow.

\subsection{A Weiss-type monotonicity formula at contact points}

Fix a contact point \(x_0\in \Gamma_\Sigma(\mathbf{u})\). Decreasing the standard scale \(r_0\) if necessary, we may assume \(4r_0 \le r_\Sigma\) and \(B_{4r_0}(x_0)\subset\subset D\). For \(0<r<r_0\), define the total interior energy, the boundary trace integral, and the Weiss-type functional respectively by

\begin{equation}\label{eq.Dr}
    D_{x_0,\mathbf{u}}(r)=D(r)=\int_{B_r(x_0)}\left( 
        |\nabla\mathbf{u}|^2 + \msf d_\Sigma^{2\gamma}\chi_{\Omega_{\mathbf{u}}}
    \right)\,\dx{x},
\end{equation}
\begin{equation}\label{eq.Hr}
    H_{x_0,\mathbf{u}}(r)=H(r):=\int_{\partial B_r(x_0)}|\mathbf{u}|^2\,\dH{n-1},
\end{equation}
and 
\begin{equation}\label{eq.Weiss}
    W_{x_0,\mathbf{u}}(r)=W(r):=r^{-n-2\gamma}D(r)-(\gamma+1)r^{-n-2\gamma-1}H(r).
\end{equation}

\begin{lemma}[Almost monotonicity for Weiss's energy]\label{lem:Weiss-monotonicity}
    The function \(r\mapsto W(r)\) is locally absolutely continuous in \((0,r_0)\), and for a.e. \(r\in(0,r_0)\),
    \begin{equation}\label{eq.Wr'}
        \begin{aligned}
            W'(r)={}&\frac 2{r^{n+2(\gamma+1)}}\int_{\partial B_r(x_0)}\left|
                \nabla\mathbf{u}\cdot(x-x_0)-(\gamma+1)\mathbf{u}
            \right|^2\,\dH{n-1}\\ 
            &+\frac{2\gamma}{r^{n+2\gamma+1}}\int_{B_r(x_0)}\msf d_\Sigma^{2\gamma-1}\chi_{\Omega_{\mathbf{u}}}\frac{x-\pi_\Sigma(x)}{|x-\pi_\Sigma(x)|}\cdot(\pi_\Sigma(x)-x_0)\,\dx{x}.
        \end{aligned}
    \end{equation}
\end{lemma}

\begin{proof}
    By translating coordinates so that \(x_0=0\), we define 
    \[
        G(r):=\int_{B_r(0)}|\nabla\mathbf{u}|^2\,\dx{x}.
    \]
    Using radial differentiation, we obtain for a.e. \(r\in(0,r_0)\),
    \begin{equation}\label{eq.H'}
        H'(r)=\frac{n-1}rH(r)+2\int_{\partial B_r(0)}\mathbf{u}\cdot\nabla\mathbf{u}\cdot\nu\,\dH{n-1}.
    \end{equation}
    The harmonicity and truncation argument used in \cite[Lemma 20]{CSY18} gives, for a.e. \(r\in(0,r_0)\),
    \[
        G(r) = \int_{\partial B_r(0)}\mathbf{u}\cdot\nabla \mathbf{u}\cdot\nu\,\dH{n-1}.
    \]
    Following the argument of \cite[Lemma 21]{CSY18}, choose a smooth radial cutoff \(\phi_\varepsilon\) equal to one on \(B_r(0)\) and zero outside \(B_{r+\varepsilon}(0)\), with \(|\nabla\phi_\varepsilon|\le C/\varepsilon\). Testing \eqref{eq:inner-dist} with \(\Phi(x)=x\phi_\varepsilon(x)\) and letting \(\varepsilon\to 0^+\) yields, for a.e. \(r\),
    \begin{equation}\label{eq,pohoid1}
        \begin{aligned}
            rD'(r) ={}&(n-2)G(r)+n\int_{B_r(0)}\msf d_\Sigma^{2\gamma}\chi_{\Omega_{\mathbf{u}}}\,\dx{x}+2r\int_{\partial B_r(0)}|\nabla\mathbf{u}\cdot\nu|^2\,\dH{n-1}\\ 
            &+2\gamma\int_{B_r(0)}\msf d_\Sigma^{2\gamma-1}(x)\chi_{\Omega_{\mathbf{u}}}\frac{(x-\pi_\Sigma(x))}{|x-\pi_\Sigma(x)|}\cdot x\,\dx{x},
        \end{aligned}
    \end{equation}
    Let us write \(x=(x-\pi_\Sigma(x))+\pi_\Sigma(x)\). Then it follows that 
    \[
        \frac{x-\pi_\Sigma(x)}{|x-\pi_\Sigma(x)|}\cdot (x-\pi_\Sigma(x))=|x-\pi_\Sigma(x)|= \msf d_\Sigma(x).
    \]
    Thus we may rewrite the last term in equation \eqref{eq,pohoid1} as
    \[
        \begin{aligned}
            &\int_{B_r(0)}\msf d_\Sigma^{2\gamma-1}(x)\chi_{\Omega_{\mathbf{u}}}\frac{(x-\pi_\Sigma(x))}{|x-\pi_\Sigma(x)|}\cdot x\,\dx{x}\\ 
            ={}&\int_{B_r(0)}\msf d_\Sigma^{2\gamma}(x)\chi_{\Omega_{\mathbf{u}}}\,\dx{x}+
            \int_{B_r(0)}\msf d_\Sigma^{2\gamma-1}(x)\chi_{\Omega_{\mathbf{u}}}\frac{(x-\pi_\Sigma(x))}{|x-\pi_\Sigma(x)|}\cdot \pi_\Sigma(x)\,\dx{x}\\ 
            ={}&\left( D(r) - G(r) \right) +
            \int_{B_r(0)}\msf d_\Sigma^{2\gamma-1}(x)\chi_{\Omega_{\mathbf{u}}}\frac{(x-\pi_\Sigma(x))}{|x-\pi_\Sigma(x)|}\cdot \pi_\Sigma(x)\,\dx{x}.
        \end{aligned}
    \]
    Therefore, the equation \eqref{eq,pohoid1} becomes
    \begin{equation}\label{eq.pohoid_rewritten}
        \begin{aligned}
            rD'(r)={}&(n+2\gamma)D(r)-2(\gamma+1)G(r)+2r\int_{\partial B_r(0)}\left|\nabla\mathbf{u}\cdot\nu \right|^2\,\dH{n-1}\\ 
            &+2\gamma \int_{B_r(0)}\msf d_\Sigma^{2\gamma-1}(x)\chi_{\Omega_{\mathbf{u}}}\frac{(x-\pi_\Sigma(x))}{|x-\pi_\Sigma(x)|}\cdot \pi_\Sigma(x)\,\dx{x}.
        \end{aligned}
    \end{equation}
    Substituting these identities into the derivative of the Weiss functional \eqref{eq.Weiss} gives
    \[
        \begin{aligned}
            W'(r)
            ={}&2r^{-n-2\gamma}
            \int_{\partial B_r(0)}
            \left|
                \nabla\mathbf{u}\cdot\nu
                -\frac{1+\gamma}{r}\mathbf{u}
            \right|^2\,\dH{n-1}\\
            &+2\gamma r^{-n-2\gamma-1}
            \int_{B_r(0)}
            \msf d_\Sigma^{2\gamma-1}
            \chi_{\Omega_{\mathbf{u}}}
            \frac{x-\pi_\Sigma(x)}{|x-\pi_\Sigma(x)|}
            \cdot\pi_\Sigma(x)\,\dx{x}.
        \end{aligned}
    \]
    which proves \eqref{eq.Wr'} and concludes the proof.
\end{proof}

\begin{remark}
    If \(\Sigma\) is a flat hyperplane, then \((x-\pi_\Sigma(x))\cdot(\pi_\Sigma(x)-x_0)\equiv 0\) and 
    \[
        W'(r)
        =2r^{-n-2\gamma}
        \int_{\partial B_r(x_0)}
        \left|
            \nabla\mathbf{u}\cdot\nu
            -\frac{1+\gamma}{r}\mathbf{u}
        \right|^2\,\dH{n-1}
        \ge0,
    \]
    where \(\nu\) is the outer unit normal to \(B_r(x_0)\). For a general \(C^{1,\rm Dini}\) hypersurface, the geometric remainder measures the deviation from the tangent plane. The Dini condition makes this remainder integrable in the radius and yields almost monotonicity.
\end{remark}

The next estimate shows that the geometric remainder in \eqref{eq.Wr'} is integrable in the radius. This will give existence of the Weiss limit and homogeneity of blow-up limits.

\begin{lemma}\label{lem:esmaddi}
    There exists a constant \(C = C(n,\gamma,\omega_\Sigma(r_\Sigma)) > 0\) such that for all \(0<r<r_0\),
    \[
        r^{-n-2\gamma-1}
            \left|
                \int_{B_r(x_0)}
                \msf d_\Sigma^{2\gamma-1} \chi_{\Omega_{\mathbf{u}}}
                \frac{x-\pi_\Sigma(x)}{|x-\pi_\Sigma(x)|}
                \cdot \bigl(\pi_\Sigma(x) - x_0\bigr)
                \,\dx{x}
            \right| \le C \frac{\omega_\Sigma(2r)}{r}.
    \]
\end{lemma}

\begin{proof}
    By the projection estimate \eqref{eq.projection-estimate} and the gradient integrability bound \eqref{eq.gradient-integrability} in Proposition \ref{prop:C1Dini-properties}, we have 
    \[  
        \begin{aligned}
            &\left|
                \int_{B_r(x_0)}
                \msf d_\Sigma^{2\gamma-1} \chi_{\Omega_{\mathbf{u}}}
                \frac{x-\pi_\Sigma(x)}{|x-\pi_\Sigma(x)|}
                \cdot \bigl(\pi_\Sigma(x) - x_0\bigr)
                \,\dx{x}
            \right|\\ 
            \le{}& 2r\omega_\Sigma(2r)\int_{B_r(x_0)}\msf d_\Sigma^{2\gamma-1}\,\dx{x}\le C\omega_\Sigma(2r)r^{n+2\gamma},
        \end{aligned}
    \]
    and the result follows after dividing by \(r^{n+2\gamma+1}\).
\end{proof}

We are now in a position to establish the existence of the limit of the Weiss functional \(W(r)\) as \(r\to 0^+\).

\begin{proposition}[Existence of the Weiss limit at contact points]
\label{prop:Weiss-limit}
    For every contact point \(x_0\in \Gamma_\Sigma(\mathbf{u})\), the limit \(W(0^+) := \lim_{r\to 0^+} W(r)\) exists and is finite.
\end{proposition}

\begin{proof}
    Combining the derivative formula \eqref{eq.Wr'} with the geometric remainder estimate from Lemma \ref{lem:esmaddi}, we deduce that there exists a constant \(C>0\) such that the function 
    \[
        r\longmapsto W(r)+C\int_0^r\frac{\omega_\Sigma(2t)}t\,\dx{t}
    \]
    is non-decreasing for \(0<r<r_0\). Since \(\omega_\Sigma\) satisfies the Dini condition, the correction integral tends to zero as \(r\to 0^+\). Moreover, the optimal growth estimate \eqref{eq:optimal-growth} and the Lipschitz bound \eqref{eq:Lipschitz-contact} imply \(D(r) \le C r^{n+2\gamma}\) and \(H(r) \le C r^{n+1+2\gamma}\), so \(W(r)\) remains uniformly bounded as \(r\to 0^+\). By the monotone convergence principle, the limit \(W(0^+)\) exists and is finite.
\end{proof}

\subsection{Compactness and convergence of blow-up sequences at contact points}

The growth estimates at contact points yield compactness of blow-up sequences under the natural scaling \(r^{1+\gamma}\). We now show that the rescaled energies converge to a limit functional whose weight is given by the distance to the tangent plane of \(\Sigma\) at \(x_0\).

For a contact point \(x_0\in \Gamma_\Sigma(\mathbf{u})\) and \(r>0\), define the rescaling
\begin{equation}\label{eq.blow-up-sequence}
    \mathbf{u}_{x_0,r}(x):= \frac{\mathbf{u}(x_0+rx)}{r^{1+\gamma}}, \qquad x\in D_{x_0,r} := \frac{D-x_0}{r} = \left\lbrace \frac{x-x_0}{r}:x\in D\right\rbrace.
\end{equation}
The corresponding positivity set, rescaled hypersurface, and distance function satisfy
\[
    \Omega_{\mathbf{u}_{x_0,r}}=\frac{\Omega_{\mathbf{u}}-x_0}r,\quad \Sigma_{x_0,r}:=\frac{\Sigma-x_0}r,\quad\msf d_{\Sigma_{x_0,r}}(x)=\frac{\msf d_\Sigma(x_0+rx)}r.
\]

\begin{proposition}[Compactness and convergence of blow-up sequences at contact points]
\label{prop:blow-up-convergence}
    Let \(\mathbf{u}\) be an \(\varepsilon_0\)-local minimizer of \(J\) in \(\mathbb{A}\), let \(x_0\in\Gamma_\Sigma(\mathbf{u})\), and let \(r_k \to 0^+\). Then up to a subsequence, there exists a nonnegative locally Lipschitz function \(\mathbf{u}_0\in W_{\rm loc}^{1,2}(\R^n;\R^m)\) such that
    \[
        \mathbf{u}_{x_0,r_k}\to \mathbf{u}_0\quad\text{ locally uniformly and strongly in }W_{\rm loc}^{1,2}(\R^n;\R^m).
    \]
    Let \(T_\Sigma := T_{x_0}\Sigma\). Then 
    \[
        \msf d_{\Sigma_{x_0,r_k}}^{2\gamma} \chi_{\Omega_{\mathbf{u}_{x_0,r_k}}} \to \msf d_{T_\Sigma}^{2\gamma} \chi_{\Omega_{\mathbf{u}_0}}\quad\text{ strongly in } L^1_{\rm loc}(\R^n),
    \]
    and 
    \[
        \chi_{\Omega_{\mathbf{u}_{x_0,r_k}}} \to \chi_{\Omega_{\mathbf{u}_0}}
        \quad\text{ strongly in } L^1_{\rm loc}(\R^n).
    \]
    Finally, \(\mathbf{u}_0\) is a global minimizer of
    \[
        J_T(\mathbf{v};B_R(0)):=\int_{B_R(0)}\left( 
            |\nabla\mathbf{v}|^2 + \msf d_{T_\Sigma}^{2\gamma}\chi_{\Omega_{\mathbf{v}}}
        \right)\,\dx{x}.
    \]
    That is, for every \(R>0\) and every \(\mathbf{v}\in W^{1,2}(B_R;\R_+^m)\) with \(\mathbf{v}-\mathbf{u}_0\in W_0^{1,2}(B_R;\R^m)\), there holds 
    \[
        J_T(\mathbf{u}_0;B_R(0))\le J_T(\mathbf{v};B_R(0)).
    \]
\end{proposition}

\begin{proof}
    By translating and rotating coordinates, we may assume without loss of generality that \(x_0=0\) and \(T:=T_0\Sigma =\R^{n-1}\times\{0\}\). We adopt the following abbreviated notation:
    \begin{equation}\label{eq.abbre_notations}
        \mathbf{u}_k:=\mathbf{u}_{0,r_k},\qquad \Sigma_k:=r_k^{-1}\Sigma,\qquad \msf d_k:=\msf d_{\Sigma_k},\qquad \msf d_0:=\msf d_T.
    \end{equation}
    We further write
    \begin{equation}\label{eq.abbre_notations2}
        \Omega_k:=\Omega_{\mathbf{u}_k},\qquad \Omega_0:=\Omega_{\mathbf{u}_0}.
    \end{equation}

    By the graph representation in Definition \ref{def.C1Dini-hypersurface}, the rescaled graph functions \(f_k(y):=r_k^{-1}f(r_ky)\) converge locally to zero in \(C^1\). Consequently,
    \begin{equation}\label{eq.claim}
        \msf d_k\to \msf d_0,\qquad \msf d_k^{2\gamma}\to \msf d_0^{2\gamma}
        \quad\text{locally uniformly in }\R^n.
    \end{equation}
    Combined with the optimal growth estimate \eqref{eq:optimal-growth} and the Lipschitz bound \eqref{eq:Lipschitz-contact}, the Arzelà--Ascoli theorem and Rellich compactness yield a subsequence, still denoted \(\{\mathbf{u}_k\}\), that converges locally uniformly and weakly in \(W_{\rm loc}^{1,2}(\R^n;\R^m)\) to a nonnegative locally Lipschitz function \(\mathbf{u}_0\). Each component \(u_{0,i}\) is harmonic in \(\{u_{0,i}>0\}\), since \(u_{k,i}\) is harmonic on every compact subset of this set for all sufficiently large \(k\).

    To obtain strong convergence, let \(w\) denote either \(u_{k,i}\) or \(u_{0,i}\). In both cases, \(\Delta w=0\) weakly in \(\{w>0\}\). For \(0\le\varphi\in C_0^\infty(\R^n)\) and \(\varepsilon>0\), the function \(\varphi(w-\varepsilon)_+\) is compactly supported in \(\{w>0\}\) and is an admissible Sobolev test. For \(u_{k,i}\), we take \(k\) sufficiently large that this support lies in the rescaled domain. Thus
    \[
        \int\varphi\chi_{\{w>\varepsilon\}}|\nabla w|^2\,\dx{x}
        =-\int(w-\varepsilon)_+\nabla w\cdot\nabla\varphi\,\dx{x}.
    \]
    Letting \(\varepsilon\to 0^+\), using \(\nabla w=0\) a.e. on \(\{w=0\}\), and summing over the components, we obtain, for \(j=k,0\),
    \[
        \int\varphi|\nabla\mathbf{u}_j|^2\,\dx{x}
        =-\sum_{i=1}^m
            \int u_{j,i}\nabla u_{j,i}\cdot\nabla\varphi\,\dx{x}.
    \]
    The right-hand side converges as \(k\to\infty\) by strong local \(L^2\) convergence of the functions and weak local \(L^2\) convergence of their gradients. The corresponding weighted gradient norms therefore converge. Choosing \(\varphi=1\) on any prescribed compact ball proves strong local \(W^{1,2}\) convergence.

    We now establish both the minimality of \(\mathbf{u}_0\) and the convergence of the characteristic functions simultaneously. Define the rescaled energy functional
    \[
        J_k(\mathbf{w};B):=\int_B\bigl(|\nabla \mathbf{w}|^2 + \msf d_k^{2\gamma}\chi_{\Omega_{\mathbf{w}}}\bigr)\,\dx{x}.
    \]
    Fix \(R>0\) and \(\delta>0\), and let \(\mathbf{v}\) be an admissible competitor for \(\mathbf{u}_0\) in \(B_R(0)\), extended by \(\mathbf{u}_0\) outside \(B_R(0)\). Choose a cutoff function \(\eta\in C_0^\infty(B_{R+\delta}(0))\) with \(0\le\eta\le 1\) and \(\eta=1\) on \(B_R(0)\), and set \(\mathbf{v}_k := \eta\mathbf{v} + (1-\eta)\mathbf{u}_k\). By the standard scale property (cf. Definition \ref{def:standard-scale}), for sufficiently large \(k\), \(\mathbf{u}_k\) minimizes \(J_k\) in \(B_{R+\delta}(0)\). On the annulus \(A_\delta:=B_{R+\delta}(0) \setminus \overline{B_R(0)}\), both \(\mathbf{v}_k\) and \(\mathbf{u}_k\) converge strongly to \(\mathbf{u}_0\) in \(W^{1,2}\). The comparison with \(\mathbf{v}_k\) therefore yields
    \[
        J_k(\mathbf{u}_k;B_R(0))
        \le J_k(\mathbf{v};B_R(0))+\int_{A_\delta}\msf d_k^{2\gamma}\,\dx{x}
        +o_k(1).
    \]
    Here \(o_k(1)\to0\) as \(k\to\infty\) with \(R\), \(\delta\), \(\mathbf{v}\), and \(\eta\) fixed. More precisely, one may take
    \[
        o_k(1)
        :=\int_{A_\delta}
            \bigl(
                |\nabla\mathbf{v}_k|^2
                -|\nabla\mathbf{u}_k|^2
            \bigr)\,\dx{x},
    \]
    which tends to zero because both functions converge strongly
    to \(\mathbf{u}_0\) in \(W^{1,2}(A_\delta)\).

    Since \(\chi_{\Omega_0}\le\liminf_k\chi_{\Omega_k}\) a.e., it follows from Fatou's lemma that
    \[
        J_T(\mathbf{u}_0;B_R(0))\le\liminf_kJ_k(\mathbf{u}_k;B_R(0)).
    \]
    Combining this with the previous inequality, sending \(k\to\infty\) first and then \(\delta\to 0^+\), we obtain
    \[
        \begin{aligned}
            J_T(\mathbf{u}_0;B_R(0))
            &\le \liminf_{k\to\infty} J_k(\mathbf{u}_k;B_R(0))\\
            &\le\limsup_kJ_k(\mathbf{u}_k;B_R(0))
            \le J_T(\mathbf{v};B_R(0)),
        \end{aligned}
    \]
    which proves the global minimality of \(\mathbf{u}_0\). Taking \(\mathbf{v}=\mathbf{u}_0\) in the preceding comparison gives
    \[
        J_k(\mathbf{u}_k;B_R(0))
        \longrightarrow J_T(\mathbf{u}_0;B_R(0)).
    \]
    Subtracting the Dirichlet energies, which converge by the strong local \(W^{1,2}\) convergence already proved, yields
    \[
        \int_{B_R(0)}
            \msf d_k^{2\gamma}\chi_{\Omega_k}\,\dx{x}
        \longrightarrow
        \int_{B_R(0)}
            \msf d_0^{2\gamma}\chi_{\Omega_0}\,\dx{x}.
    \]
    Set
    \[
        q_k:=\msf d_k^{2\gamma},\qquad
        q_0:=\msf d_0^{2\gamma},\qquad
        a_k:=q_k\chi_{\Omega_k},\qquad
        a_0:=q_0\chi_{\Omega_0}.
    \]
    The weight convergence in \eqref{eq.claim} means that
    \[
        \|q_k-q_0\|_{L^\infty(B_R(0))}\longrightarrow0,
        \qquad
        q_k(x)=r_k^{-2\gamma}
            \msf d_\Sigma(r_kx)^{2\gamma}.
    \]
    At every point of \(\Omega_0\), local uniform convergence gives \(\chi_{\Omega_k}=1\) for all sufficiently large \(k\). Consequently, \((a_0-a_k)_+\to0\) pointwise, with \(0\le(a_0-a_k)_+\le a_0\). The convergence of the integrals established above therefore implies
    \[
        \begin{aligned}
            \|a_k-a_0\|_{L^1(B_R(0))}
            &=
            \int_{B_R(0)}(a_k-a_0)\,\dx{x}
            +2\int_{B_R(0)}(a_0-a_k)_+\,\dx{x}\\
            &\longrightarrow0.
        \end{aligned}
    \]
    Moreover,
    \[
        \begin{aligned}
            &\int_{B_R(0)}
                q_0|\chi_{\Omega_k}-\chi_{\Omega_0}|\,\dx{x}\\
            &\qquad\le
            \|a_k-a_0\|_{L^1(B_R(0))}
            +|B_R(0)|\|q_k-q_0\|_{L^\infty(B_R(0))}
            \longrightarrow0.
        \end{aligned}
    \]
    This gives the convergence of the weighted phase. Finally, for any \(0<\delta<R\), we have the interpolation estimate
    \[
        \|\chi_{\Omega_k}-\chi_{\Omega_0}\|_{L^1(B_R(0))}
        \le\delta^{-2\gamma}\int_{B_R(0)}\msf d_0^{2\gamma}
        |\chi_{\Omega_k}-\chi_{\Omega_0}|\,\dx{x}
        +C(n)R^{n-1}\delta.
    \]
    The integral on the right tends to zero by the convergence just established. Letting first \(k\to\infty\) and then \(\delta\to 0^+\) removes the weight and completes the proof.
\end{proof}

We are now ready to pass the Weiss monotonicity formula to the blow-up limit. With the notation introduced in \eqref{eq.abbre_notations}--\eqref{eq.abbre_notations2}, define the rescaled Weiss functional
\begin{equation}\label{eq.scaled_Weiss}
    \begin{aligned}
        W_k(\rho):={}&\rho^{-n-2\gamma} \int_{B_\rho(0)} \left( |\nabla \mathbf{u}_k|^2 + \msf d_k^{2\gamma} \chi_{\Omega_k} \right)\,\dx{x}\\ 
        &-(1+\gamma)\rho^{-n-1-2\gamma} \int_{\partial B_\rho(0)} |\mathbf{u}_k|^2\,\dH{n-1},
    \end{aligned}
\end{equation}
and let \(W_0\) denote the analogous expression for \(\mathbf{u}_0\) and \(\msf d_0\) derived in Proposition \ref{prop:blow-up-convergence}. By the scaling property of the Weiss functional and the convergence results in Proposition \ref{prop:blow-up-convergence}, for every \(\rho>0\) we have
\begin{equation}\label{eq.scaled_Weiss_identity}
    W_k(\rho) = W(r_k \rho) \longrightarrow W_0(\rho) = W(0^+).
\end{equation}
Moreover, combining the monotonicity formula (cf. Lemma \ref{lem:Weiss-monotonicity}) with the geometric remainder estimate (cf. Lemma \ref{lem:esmaddi}), we obtain for \(0<\varrho<\sigma\) and all sufficiently large \(k\),
\begin{equation}\label{eq:scaled-homogeneity-defect-estimate}
    \begin{aligned}
        &2\int_\varrho^\sigma r^{-n-2(1+\gamma)}
        \int_{\partial B_r(0)}
        |\nabla \mathbf{u}_k\cdot x-(1+\gamma)\mathbf{u}_k|^2
        \,\dH{n-1}\dx{r}\\
        &\qquad\le W_k(\sigma)-W_k(\varrho)+C\int_\varrho^\sigma\frac{\omega_\Sigma(2r_kr)}r\,\dx{r}.
    \end{aligned}
\end{equation}
As \(k\to\infty\), the right-hand side of \eqref{eq:scaled-homogeneity-defect-estimate} converges to zero, which yields the following homogeneity result for blow-up limits.

\begin{corollary}[Homogeneity of blow-up limits]\label{cor:homogeneity-blow-up-limits}
    Every blow-up limit \(\mathbf{u}_0\) from Proposition \ref{prop:blow-up-convergence} is homogeneous of degree \((1+\gamma)\). Furthermore, for every \(\rho>0\),
    \begin{equation}\label{eq:W0rho}
        W_0(\rho)=\rho^{-n-2\gamma}\int_{B_\rho(0)}\msf d_0^{2\gamma}\chi_{\Omega_0}\,\dx{x}.
    \end{equation}
\end{corollary}

\begin{proof}
    Fix \(0<\varrho<\sigma\). By \eqref{eq.scaled_Weiss_identity} and the Dini condition (cf. Definition \ref{def.Dini-modulus}), the right-hand side of \eqref{eq:scaled-homogeneity-defect-estimate} tends to zero as \(k\to\infty\). Strong local \(W^{1,2}\) convergence of \(\{\mathbf{u}_k\}\) then allows us to pass to the limit in the left-hand side, giving
    \begin{equation}\label{eq.homogeneity-defect-vanishing}
        \int_{B_\sigma(0)\setminus B_\varrho(0)}|x|^{-n-2(\gamma+1)} \left|
            \nabla \mathbf{u}_0 \cdot x - (\gamma+1)\mathbf{u}_0 
        \right|^2\,\dx{x} = 0.
    \end{equation}
    Since \(\varrho\) and \(\sigma\) are arbitrary, the radial derivative of \(r^{-1-\gamma}\mathbf{u}_0(r\theta)\) vanishes a.e., which proves the homogeneity of \(\mathbf{u}_0\). By harmonicity and homogeneity,
    \[
        \begin{aligned}
            \int_{B_\rho(0)} |\nabla \mathbf{u}_0|^2 \,\dx{x} &= \int_{\partial B_\rho(0)} \mathbf{u}_0 \cdot \nabla\mathbf{u}_0\cdot\nu \,\dH{n-1} \\ 
            &= \frac{\gamma+1}{\rho} \int_{\partial B_\rho(0)} |\mathbf{u}_0|^2 \,\dH{n-1}.
        \end{aligned}
    \]
    Substituting this into the definition of \(W_0(\rho)\) yields \eqref{eq:W0rho}.
\end{proof}

We collect here the variational identities satisfied by the flat blow-up limits, which will play a key role in the classification analysis that follows.

\begin{lemma}[Variational properties of flat blow-up limits]\label{lem:basic_property_U0}
    Let \(\mathbf{u}_0=(u_{0,1},\dots,u_{0,m})\) be a blow-up limit as in Proposition \ref{prop:blow-up-convergence}, with tangent plane \(T = T_{x_0}\Sigma\), distance function \(\msf d_0=\msf d_T\), and positivity set \(\Omega_0 = \{|\mathbf{u}_0| > 0\}\). Then
    \begin{enumerate}[label=(\roman*)]
    \item \label{item:harmonic_U0}
    Each component of \(\mathbf{u}_0\) is harmonic in \(\Omega_0\).
    \item \label{item:inner_variation_U0}
    For every vector field \(\Phi\in C_0^\infty(\R^n;\R^n)\), there holds
    \[
            \begin{aligned}
                \int_{\R^n}\Bigg[
                &\bigl(|\nabla \mathbf{u}_0|^2+\msf d_0^{2\gamma}\chi_{\Omega_0}\bigr)
                  \operatorname{div}\Phi
                -2\sum_{i=1}^m
                  \langle D\Phi\,\nabla u_{0,i},\nabla u_{0,i}\rangle\\
                &+2\gamma \msf d_0^{2\gamma-1}\chi_{\Omega_0}
                  \frac{x-\pi_T(x)}{|x-\pi_T(x)|}\cdot\Phi
                \Bigg]\,\dx{x}=0.
            \end{aligned}
        \]
    \item \label{item:Weiss_monotonicity_U0}
    The flat Weiss energy \(W_0\) satisfies the monotonicity formula
    \begin{equation}\label{eq.flat-Weiss-U0}
        W_0'(\rho)
        =\frac2{\rho^{n+2(\gamma+1)}}
            \int_{\partial B_\rho(0)}
            \left|\nabla \mathbf{u}_0\cdot x-(\gamma+1)\mathbf{u}_0\right|^2
            \,\dH{n-1}\ge 0
    \end{equation}
    for a.e. \(\rho>0\).
    \end{enumerate}
\end{lemma}
\begin{proof}
    Harmonicity and the inner variation identity follow directly from the global minimality of \(\mathbf{u}_0\). The derivation of the Weiss monotonicity formula in Lemma \ref{lem:Weiss-monotonicity} carries over verbatim to the case \(\Sigma = T\); the geometric remainder term vanishes identically since \(\nabla \msf d_T(x) \cdot \pi_T(x) = 0\) for almost every \(x \notin T\).
\end{proof}

\begin{remark}
    For the homogeneous blow-up limits studied above, the Weiss functional \(W_0\) is constant and equal to the contact-point Weiss limit \(W(0^+)\).
\end{remark}

\subsection{Analysis of the weighted density at contact points}

Our compactness and homogeneity results allow us to identify the weighted density at contact points \(x_0\in \Gamma_\Sigma(\mathbf{u})\) with the limiting Weiss energy defined in \eqref{eq.Weiss}. We now establish that this density satisfies a universal gap property: its positive values are uniformly bounded away from zero.

After translating the contact point to the origin, we abbreviate the weighted density in \eqref{eq.weighted-density} as
\[
    \Theta_\Sigma(\mathbf{u};r):=r^{-n-2\gamma}\int_{B_r(0)}\msf d_\Sigma^{2\gamma}\chi_{\Omega_{\mathbf{u}}}\,\dx{x}.
\]
We continue to use the blow-up notation introduced in \eqref{eq.abbre_notations} and \eqref{eq.abbre_notations2}.
\begin{proposition}[Weighted density gap at contact points]\label{prop:weighted-density-contact}
    Let \(\mathbf{u}\) be an \(\varepsilon_0\)-local minimizer of \(J\) in \(\mathbb{A}\), and let \(x_0\in \Gamma_\Sigma(\mathbf{u})\). Then the limit \(\Theta_\Sigma(\mathbf{u}; x_0, 0^+)\) exists and satisfies
    \begin{equation}\label{eq.weighted-density1}
        \Theta_\Sigma(\mathbf{u};x_0,0^+) = W_0(1)
    \end{equation}
    for every blow-up limit \(\mathbf{u}_0\) at \(x_0\). Moreover,
    \begin{equation}\label{eq.weighted-density2}
        W_0(1) \in \{0\} \cup [\theta_1(n,\gamma), \theta_2(n,\gamma)],
    \end{equation}
    where \(0 < \theta_1(n,\gamma) \le \theta_2(n,\gamma) < \infty\) are universal constants depending only on \(n\) and \(\gamma\).
\end{proposition}

\begin{proof}
    Assume without loss of generality that \(x_0 = 0\). Given any sequence \(r_k \to 0^+\), Proposition \ref{prop:blow-up-convergence} yields a subsequence along which \(\mathbf{u}_k \to \mathbf{u}_0\) with the stated energy and phase convergence. By \eqref{eq:W0rho} and \eqref{eq.scaled_Weiss_identity},
    \[
        \Theta_\Sigma(\mathbf{u}; r_k)
        = \int_{B_1(0)} \msf d_k^{2\gamma} \chi_{\Omega_k} \,\dx{x}
        \longrightarrow
        \int_{B_1(0)} \msf d_0^{2\gamma} \chi_{\Omega_0} \,\dx{x}
        = W_0(1) = W(0^+),
    \]
    where the last equality follows by passing to the limit in \(W_k(1) = W(r_k)\). Since \(W(0^+)\) exists by Proposition \ref{prop:Weiss-limit}, the full limit exists and satisfies \eqref{eq.weighted-density1}. Now suppose \(\mathbf{u}_0 \not\equiv 0\), and let \(\cC\) be any connected component of \(\Omega_0\). By homogeneity, \(\cC\) is a cone with vertex at the origin, so its intersection with the unit sphere \(\cS:= \cC \cap \mathbb{S}^{n-1}\) is a connected open subset of \(\mathbb{S}^{n-1}\). By harmonicity and the strong maximum principle, at least one component \(u_{0,i}\) is strictly positive in \(\cC\). Write
    \[
        u_{0,i}(r\theta)
        =r^{1+\gamma}\varphi_i(\theta),
        \qquad r>0,\quad\theta\in\cS.
    \]
    The function \(\varphi_i\) is Lipschitz and vanishes on \(\partial\cS\). Its truncations \((\varphi_i-\varepsilon)_+\) have compact support in \(\cS\) and converge to \(\varphi_i\) in \(H^1(\cS)\). Hence \(\varphi_i\in H_0^1(\cS)\) and satisfies the eigenvalue problem
    \begin{equation}\label{eq.eigenvalue_problem}
        -\Delta_{\mathbb{S}^{n-1}} \varphi_i
        = (1+\gamma)(n+\gamma-1) \varphi_i
        \quad\text{in } \cS,
        \qquad \varphi_i > 0
        \quad\text{in } \cS.
    \end{equation}
    Therefore, the first Dirichlet eigenvalue of \(-\Delta_{\mathbb{S}^{n-1}}\) on \(\cS\) satisfies
    \begin{equation}\label{eq.first_eigenvalue_esm}
        \lambda_1(\cS)
        \le \frac{\int_{\cS} |\nabla_{\mathbb{S}^{n-1}} \varphi_i|^2\,\mathrm d\mathcal H^{n-1}}
        {\int_{\cS} \varphi_i^2\,\mathrm d\mathcal H^{n-1}}
        = (1+\gamma)(n+\gamma-1).
    \end{equation}
    Set \(\lambda:=(1+\gamma)(n+\gamma-1)\) and \(\psi:=\frac{\varphi_i}{\|\varphi_i\|_{L^2(\cS)}}\). Extend \(\psi\) by zero to \(\mathbb S^{n-1}\). Then
    \[
        \|\psi\|_{L^2(\mathbb S^{n-1})}^2=1,
        \qquad
        \|\nabla_{\mathbb S^{n-1}}\psi\|_{L^2(\mathbb S^{n-1})}^2
        =\lambda.
    \]
    For \(q=2n/(n-1)>2\), the Sobolev inequality on \(\mathbb S^{n-1}\) and H\"older's inequality yield
    \[
        1\le\bigl[\mathcal H^{n-1}(\cS)\bigr]^{1/n}\|\psi\|_{L^q(\mathbb S^{n-1})}^2\le C(n)(1+\lambda)\bigl[\mathcal H^{n-1}(\cS)\bigr]^{1/n}.
    \]
    Therefore,
    \[
        \mathcal H^{n-1}(\cS)
        \ge [C(n)(1+\lambda)]^{-n}
        =:\mu_0(n,\gamma)>0.
    \]
    
    Let \(\nu\) be a unit normal vector to the tangent plane \(T\). Choose \(\delta_0 = \delta_0(n,\gamma) \in (0,1)\) such that
    \[
        \cH^{n-1} \bigl( \{\theta\in\mathbb{S}^{n-1} : |\theta\cdot\nu| < \delta_0\} \bigr) \le \mu_0 / 2.
    \]
    Since \(\msf d_0(\theta) = |\theta\cdot\nu|\) on \(\mathbb{S}^{n-1}\), passing to polar coordinates yields
    \[
        W_0(1)
        \ge
        \frac{1}{n+2\gamma}
        \int_{\cS}|\theta\cdot\nu|^{2\gamma}
            \,\mathrm d\mathcal H^{n-1}\ge
        \frac{\delta_0^{2\gamma}\mu_0}{2(n+2\gamma)}
        =:\theta_1(n,\gamma)>0.
    \]
    The universal upper bound follows trivially by extending the integral to the full unit ball:
    \[
        W_0(1) \le \int_{B_1} \msf d_0^{2\gamma} \,\mathrm{d}x
        = \int_{B_1} |x_n|^{2\gamma} \,\mathrm{d}x
        =: \theta_2(n,\gamma) < \infty.
    \]
    If \(\mathbf{u}_0 \equiv 0\), then \(W_0(1) = 0\). This completes the proof.
\end{proof}

We are now in a position to establish the density classification theorem and to determine the vectorial structure of connected blow-up limits at nondegenerate contact points.

\begin{proof}[Proof of Theorem \ref{thm:main}]
    Proposition \ref{prop:weighted-density-contact} establishes both the existence of the weighted density \(\Theta_\Sigma(\mathbf{u}; x_0, 0^+)\) and the universal density gap at every contact point. The dichotomy between zero and positive values of this density yields precisely the disjoint decomposition
    \[
        \Gamma_\Sigma(\mathbf{u}) = \Gamma_{\Sigma,\rm D}(\mathbf{u}) \cup \Gamma_{\Sigma,\rm ND}(\mathbf{u}),
    \]
    where \(\Gamma_{\Sigma,\rm D}(\mathbf{u})\) and \(\Gamma_{\Sigma,\rm ND}(\mathbf{u})\) are defined in \eqref{eq:GammaD-definition-density} and \eqref{eq:GammaND-definition-density}, respectively.
\end{proof}

\begin{proof}[Proof of Theorem \ref{thm:blow-up-ND}]
    Compactness of blow-up sequences, global minimality of the limits, and their \((1+\gamma)\)-homogeneity follow directly from Proposition \ref{prop:blow-up-convergence} and Corollary \ref{cor:homogeneity-blow-up-limits}. Since \(x_0 \in \Gamma_{\Sigma,\rm ND}(\mathbf{u})\), the identity
    \[
        0 < \Theta_\Sigma(\mathbf{u}; x_0, 0^+)
        = \int_{B_1(0)} \msf d_0^{2\gamma} \chi_{\Omega_0} \,\dx{x}
    \]
    guarantees that every blow-up limit at \(x_0\) is nontrivial.
    
    Assume now that \(\Omega_0\) is connected. Every nonzero component of \(\mathbf{u}_0\) is strictly positive in \(\Omega_0\), and its angular part is a positive Dirichlet eigenfunction on the connected spherical domain \(\mathcal{S} := \Omega_0 \cap \mathbb{S}^{n-1}\), as shown in \eqref{eq.eigenvalue_problem}. All such angular eigenfunctions belong to the first eigenspace, which is one-dimensional on a connected domain (cf. \cite[Lemma 25]{CSY18}). Consequently, the nonzero components of \(\mathbf{u}_0\) are mutually proportional, and after normalization we obtain
    \[
        \mathbf{u}_0 = \mathbf{a} v_0, \qquad v_0 := |\mathbf{u}_0|, \qquad \mathbf{a} \in \mathbb{S}^{m-1} \cap \mathbb{R}_+^m.
    \]
    The direction vector \(\mathbf{a}\) is unique since \(v_0 > 0\) throughout \(\Omega_0\).
    Finally, for every nonnegative scalar competitor \(w\) with \(w- v_0 \in H_0^1(B_R)\), the vector-valued function \(\mathbf{a} w\) is an admissible competitor for \(\mathbf{u}_0\). Since \(|\mathbf{a}| = 1\),
    \[
        \cJ_T(v_0; B_R)
        = J_T(\mathbf{u}_0; B_R)
        \le J_T(\mathbf{a} w; B_R)
        = \cJ_T(w; B_R).
    \]
    Therefore \(v_0\) is a scalar global minimizer, with the same positivity set and homogeneity as \(\mathbf{u}_0\).
\end{proof}

\subsection{Classification of non-degenerate contact points}

In this section, we establish the local structural assertions in Theorems \ref{thm:structure-ND} and \ref{thm:structure-mp-br}. The basic observation is that each component \(u_i\) is either positive or identically zero on a connected component of \(\Omega_{\mathbf{u}}\).

\begin{lemma}[Structure of the free boundary]\label{lem:structure-free-boundary}
    Let \(\mathbf{u}\) be an \(\varepsilon_0\)-local minimizer of \(J\) in \(\mathbb A\). Then
    \[
        \partial\Omega_{\mathbf{u}}
        =\bigcup_{i=1}^m\partial\Omega_{u_i}
        \quad\text{in }D.
    \]
    In particular, \(\partial\Omega_{u_i}\cap\Omega_{\mathbf{u}}=\varnothing\) for every \(i\in\{1,\dots,m\}\).
\end{lemma}

\begin{proof}
    By harmonicity and the strong maximum principle, each \(u_i\) is either positive or identically zero on every connected component of \(\Omega_{\mathbf{u}}\). Thus \(\partial\Omega_{u_i}\) cannot meet \(\Omega_{\mathbf{u}}\), and \(\partial\Omega_{u_i}\cap D\subset\partial\Omega_{\mathbf{u}}\cap D\). The reverse inclusion follows from the finite union \(\Omega_{\mathbf{u}}=\bigcup_{i=1}^m\Omega_{u_i}\).
\end{proof}

\begin{remark}
    This identity holds throughout \(D\); it does not require
    non-degeneracy or proximity to \(\Sigma\).
\end{remark}

Translate the contact point to the origin and recall the active index set \(\mathcal I_{\mathbf{u}}(0)\) from \eqref{eq:local-active-index-set}. Once this set is nonempty, the trichotomy follows from its finiteness and the nesting of the separation sets.

\begin{lemma}[Trichotomy at non-degenerate contact points]
    Let \(\mathbf{u}\) be an \(\varepsilon_0\)-local minimizer of \(J\) in \(\mathbb A\), and let \(0\in\Gamma_{\Sigma,\rm ND}(\mathbf{u})\). Assume that \(\mathcal I_{\mathbf{u}}(0)\ne\varnothing\). Then exactly one of
    \[
        0\in\Gamma_{\rm sp}(\mathbf{u}),\qquad
        0\in\Gamma_{\rm mp}^{\rm nb}(\mathbf{u}),\qquad
        0\in\Gamma_{\rm mp}^{\rm br}(\mathbf{u})
    \]
    holds.
\end{lemma}

\begin{proof}
    If \(\#\,\mathcal I_{\mathbf{u}}(0)=1\), the point is single-phase. Otherwise, either some pair \(i\ne j\in\mathcal I_{\mathbf{u}}(0)\) has \(\triangle_{i,j}(r)\ne\varnothing\) at every sufficiently small radius, or each pair has an empty separation set at some radius. In the latter case, take the minimum of these finitely many radii. Since \(\triangle_{i,j}(r)\) increases with \(r\), every separation set is empty at this common radius. The alternatives are precisely \eqref{eq.multi-br} and \eqref{eq.multi-nb}, and are mutually exclusive.
\end{proof}

Away from \(\Sigma\), the local connectivity theorem identifies \(\Omega_{\mathbf{u}}\) with the positivity set of any component whose free boundary passes through the point.

\begin{lemma}[Local scalar reduction away from \(\Sigma\)]\label{lem:local-scalar-reduction-away-from-Sigma}
    Let \(\mathbf{u}\) be an \(\varepsilon_0\)-local minimizer of \(J\) in \(\mathbb A\), and let \(z\in\partial\Omega_{u_i}\cap(D\setminus\Sigma)\). There exists \(\rho_z>0\) such that
    \[
        \Omega_{\mathbf{u}}\cap B_{\rho_z}(z)
        =\Omega_{u_i}\cap B_{\rho_z}(z).
    \]
    Consequently, \(\partial\Omega_{\mathbf{u}}\cap B_{\rho_z}(z)=\partial\Omega_{u_i}\cap B_{\rho_z}(z)\).
\end{lemma}

\begin{proof}
    Choose \(B_{2\rho}(z)\subset\subset D\setminus\Sigma\) below the standard scale. Here \(Q(x):=\msf d_\Sigma^\gamma(x)\) is bounded above and below by positive constants, so \cite[Theorem 5]{CSY18} gives a connected open set \(G\) with
    \[
        \Omega_{\mathbf{u}}\cap B_{\varepsilon\rho}(z)
        \subset G\subset\Omega_{\mathbf{u}}\cap B_\rho(z)
    \]
    for some \(\varepsilon>0\). Since \(z\in\partial\Omega_{u_i}\), the function \(u_i\) is nonzero in \(G\). The strong maximum principle gives \(u_i>0\) throughout \(G\), proving the assertion with \(\rho_z=\varepsilon\rho\).
\end{proof}

At a multi-phase non-branching point, the active component free boundaries agree. Applying the preceding lemma to this common boundary proves the local assertion of Theorem \ref{thm:structure-ND}.

\begin{lemma}[Local structure at a non-branching point]
    Let \(0\in\Gamma_{\rm mp}^{\rm nb}(\mathbf{u})\), and choose \(\bar r>0\) below the active-set stabilization radius such that \(B_{\bar r}(0)\subset\subset D\) and \(\triangle_{i,j}(\bar r)=\varnothing\) for all distinct \(i,j\in\mathcal I_{\mathbf{u}}(0)\). If \(i_0\in\mathcal I_{\mathbf{u}}(0)\) and
    \[
        z\in
        \bigl(\partial\Omega_{u_{i_0}}\cap B_{\bar r}(0)\bigr)
        \setminus\Sigma,
    \]
    then there exists \(\rho_z>0\), with
    \(B_{\rho_z}(z)\subset\subset B_{\bar r}(0)\cap(D\setminus\Sigma)\),
    such that
    \[
        \Omega_{\mathbf{u}}\cap B_{\rho_z}(z)
        =\Omega_{u_i}\cap B_{\rho_z}(z)
        \quad\text{for every }i\in\mathcal I_{\mathbf{u}}(0).
    \]
    The corresponding free boundaries therefore agree in this ball.
\end{lemma}

\begin{proof}
    At a multi-phase non-branching point, the separation sets are empty, so \(z \in \partial\Omega_{u_i}\) for every active index \(i \in \mathcal I_{\mathbf{u}}(0)\). Since there are only finitely many such components, we may apply Lemma \ref{lem:local-scalar-reduction-away-from-Sigma} to each \(u_i\) and select a common smaller radius that satisfies the required containment property for all indices simultaneously.
\end{proof}

For multi-phase branching points, we first establish countability of the local branches and then describe what happens where two component free boundaries separate.

\begin{lemma}[Countability of the local branches touching \(x_0\)]
\label{lem:countability-local-branches}
    Let \(0\in\Gamma_{\rm mp}^{\rm br}(\mathbf{u})\), and fix \(0<\bar r<r_0\) with \(B_{\bar r}(0)\subset\subset D\). The family \(\mathscr C_{\mathbf{u}}(\bar r)\) is at most countable. Thus, for some \(\Lambda_{\mathbf{u}}(\bar r)\subset\mathbb N\),
    \[
        \mathscr C_{\mathbf{u}}(\bar r)
        =\{\Omega_{\mathbf{u}}^{(\ell)}:\ell\in\Lambda_{\mathbf{u}}(\bar r)\},
        \qquad 0\in\partial\Omega_{\mathbf{u}}^{(\ell)}.
    \]
\end{lemma}

\begin{proof}
    The connected components of an open subset of \(\R^n\) are disjoint open sets, each containing a rational point. Hence there are at most countably many.
\end{proof}

\begin{remark}
    This family lists only the components touching the origin; its disjoint union may be a proper subset of \(\Omega_{\mathbf{u}}\cap B_{\bar r}(0)\).
\end{remark}

\begin{lemma}[Local structure at a separating point]
\label{lem:local-reduction-separating-point}
    Let \(\mathbf{u}\) be an \(\varepsilon_0\)-local minimizer of \(J\) in \(\mathbb A\), let \(x_0\in\Gamma_{\rm mp}^{\rm br}(\mathbf{u})\), and let \(B_r(x_0)\subset\subset D\). Suppose \(i\ne j\) and
    \[
        z\in
        \bigl((\partial\Omega_{u_i}\setminus\partial\Omega_{u_j})
        \cap B_r(x_0)\bigr)\setminus\Sigma.
    \]
    Then there exists \(\rho_z>0\), with \(B_{\rho_z}(z)\subset\subset B_r(x_0)\cap(D\setminus\Sigma)\), such that
    \begin{equation}\label{eq:local-reduction-separating-point}
        \partial\Omega_{\mathbf{u}}\cap B_{\rho_z}(z)
        =\partial\Omega_{u_i}\cap B_{\rho_z}(z),
    \end{equation}
    whereas
    \begin{equation}\label{eq:local-vanishing-separated-component}
        \partial\Omega_{u_j}\cap B_{\rho_z}(z)=\varnothing.
    \end{equation}
\end{lemma}

\begin{proof}
    By Lemma \ref{lem:structure-free-boundary}, we have \(z\in\partial\Omega_{\mathbf{u}}\), so in particular \(z\notin\Omega_{u_j}\). Since by assumption \(z\notin\partial\Omega_{u_j}\) as well, we conclude that \(z\notin\overline{\Omega_{u_j}}\). By openness of the complement of \(\overline{\Omega_{u_j}}\), there exists a neighborhood of \(z\) entirely disjoint from \(\overline{\Omega_{u_j}}\).
    
    We now apply Lemma \ref{lem:local-scalar-reduction-away-from-Sigma} to the component \(u_i\), shrinking the radius appropriately so that
    \begin{equation}\label{eq:local-vanishing-separated-component-closure}
        \Omega_{\mathbf{u}}\cap B_{\rho_z}(z)
        = \Omega_{u_i}\cap B_{\rho_z}(z),
        \qquad
        \overline{\Omega_{u_j}}\cap B_{\rho_z}(z) = \varnothing.
    \end{equation}
    The asserted boundary identities follow directly.
\end{proof}

\begin{proof}[Proof of Theorem \ref{thm:structure-mp-br}]
    Part \textnormal{(i)} is proved by Lemma \ref{lem:countability-local-branches}. For part \textnormal{(ii)}, interchange \(i,j\) if necessary and apply Lemma \ref{lem:local-reduction-separating-point} at the point \(z\in\triangle_{i,j}(x_0,r)\setminus\Sigma\). Moreover, if \(z\in\partial\Omega_{\mathbf{u}}^{(\ell)}\), this branch meets \(B_{\rho_z}(z)\), where \(u_i>0\) on \(\Omega_{\mathbf{u}}\) and \(u_j=0\). The strong maximum principle on the connected branch gives \(u_i>0\) and \(u_j\equiv0\) throughout it. Thus \(i\in\mathcal I_{\mathbf{u}}^{(\ell)}(x_0;\bar r)\) and \(j\notin\mathcal I_{\mathbf{u}}^{(\ell)}(x_0;\bar r)\), as required.
\end{proof}

For homogeneous blow-up limits, the geometry is more rigid: each branch is a cone, and the target-space direction is constant on it.

\begin{proof}[Proof of Corollary \ref{cor:blow-up-classification}]
    Under the connectedness assumption, Theorem \ref{thm:blow-up-ND} gives \(\mathbf{u}_0=Av_0\), with \(v_0=|\mathbf{u}_0|\). Therefore
    \[
        \Omega_{u_{0,i}}=
        \begin{cases}
            \Omega_0,&i\in\operatorname{spt}\mathbf{a},\\[3pt]
            \varnothing,&i\notin\operatorname{spt}\mathbf{a}.
        \end{cases}
    \]
    Hence \(\mathcal I_{\mathbf{u}_0}(0)=\operatorname{spt}\mathbf{a}\). If this set is \(\{i_0\}\), then \(|\mathbf{a}|=1\) gives \(\mathbf{a}=e_{i_0}\). Otherwise all active component free boundaries coincide, so the origin is multi-phase non-branching.
\end{proof}

\begin{proof}[Proof of Corollary \ref{cor:blow-up-branching}]
    Decompose the open cone \(\Omega_0\) into its at most countably many connected components \(\Omega_0^{(\ell)}\). For \(x\in\Omega_0^{(\ell)}\), homogeneity places the connected ray \(\{tx:t>0\}\) in the same component. Thus every component is a cone touching the origin, and \(\mathcal S_\ell:=\Omega_0^{(\ell)}\cap\mathbb S^{n-1}\) is connected.

    Each nonzero angular component of \(\mathbf{u}_0\) on \(\mathcal S_\ell\) is a positive first Dirichlet eigenfunction with eigenvalue \((1+\gamma)(n+\gamma-1)\). Simplicity of this eigenvalue, exactly as in the proof of Theorem \ref{thm:blow-up-ND}, yields
    \[
        \mathbf{u}_0=\mathbf{a}_\ell v_\ell
        \quad\text{in }\Omega_0^{(\ell)},
        \qquad
        \mathbf{a}_\ell\in\mathbb S^{m-1}\cap\mathbb R_+^m,
        \qquad v_\ell=|\mathbf{u}_0|.
    \]
    The vector \(\mathbf{a}_\ell\) is unique, and \(v_\ell\) is positive, harmonic, and \((1+\gamma)\)-homogeneous in this component. The spherical measure estimate in the proof of Proposition \ref{prop:weighted-density-contact} applies to each \(\mathcal S_\ell\), giving
    \[
        \mathcal H^{n-1}(\mathcal S_\ell)\ge\mu_0(n,\gamma)>0,
        \qquad
        \#\,\Lambda\le
        \frac{\mathcal H^{n-1}(\mathbb S^{n-1})}{\mu_0(n,\gamma)}.
    \]
    This proves part \textnormal{(i)}, including finiteness.

    For part \textnormal{(ii)}, the representation gives
    \begin{equation}\label{eq:component-positivity-decomposition-blow-up}
        \Omega_{u_{0,i}}
        =\bigcup_{\substack{\ell\in\Lambda\\
            i\in\operatorname{spt}\mathbf{a}_\ell}}\Omega_0^{(\ell)},
        \qquad
        \mathcal I_{\mathbf{u}_0}(0)
        =\bigcup_{\ell\in\Lambda}\operatorname{spt}\mathbf{a}_\ell.
    \end{equation}
    If all \(\operatorname{spt}\mathbf{a}_\ell\) were equal, every active component would have positivity set \(\Omega_0\); no branching pair could then exist. Thus branching requires two distinct branchwise supports.
\end{proof}

\subsection{The planar structure of non-degenerate contact points}

Let \(n=2\), and let \(\mathbf{u}_0\) be a nontrivial blow-up limit at \(x_0\in\Gamma_{\Sigma,\rm ND}(\mathbf{u})\). After a rotation, write \(T:=T_{x_0}\Sigma=\{x_2=0\}\) and \(\msf d_0(x)=|x_2|\). Set
\begin{equation}\label{eq:planar-p-beta}
    p:=1+\gamma,
    \qquad
    \beta:=\frac{\pi}{2p}
    =\frac{\pi}{2(1+\gamma)}.
\end{equation}
Thus \(2\beta=\pi/(1+\gamma)\). Recall the axial directions \(\mathfrak A\), sectors \(S_\mu\), profiles \(\Phi_\mu\), and amplitudes \(\kappa_\mu\) defined in equations \eqref{eq:planar-axial-directions}--\eqref{eq:planar-kappa}.

The angular eigenvalue equation first determines the opening of each branch; the Bernoulli condition then determines its direction and amplitude.

\begin{lemma}[Planar sector profile]\label{lem:planar-sector-profile}
    For every connected component \(\Omega_0^{(\ell)}\) of \(\Omega_0\), there exists \(\alpha_\ell\in\mathbb R\) such that
    \begin{equation}\label{eq:planar-sector-component}
        \Omega_0^{(\ell)}=\{r(\cos\theta,\sin\theta):r>0,\ \alpha_\ell<\theta<\alpha_\ell+2\beta\}.
    \end{equation}
    Moreover, there are a unique
    \(\mathbf{a}_\ell\in\mathbb S^{m-1}\cap\mathbb R_+^m\) and \(c_\ell>0\) such that
    \begin{equation}\label{eq:planar-preliminary-profile}
        \mathbf{u}_0(r,\theta)
        =\mathbf{a}_\ell c_\ell r^p
            \sin\bigl(p(\theta-\alpha_\ell)\bigr)
        \quad\text{ in }\Omega_0^{(\ell)}.
    \end{equation}
    Thus every branch has opening \(2\beta\), and there are only finitely many branches.
\end{lemma}

\begin{proof}
    By Corollary \ref{cor:blow-up-branching}, the angular section of each branch is a connected proper subset of \(\mathbb S^1\), and its positive scalar profile solves the ODE \(-\varphi''=p^2\varphi\) with zero endpoint values. Positivity forces the endpoints to be consecutive zeros, giving an interval of length \(\pi/p=2\beta\) and the stated sine profile. Finiteness and uniqueness of \(\mathbf{a}_\ell\) also follow from Corollary \ref{cor:blow-up-branching}.
\end{proof}

Before applying the Bernoulli law on these rays, we exclude rays on the degeneracy line and rays shared by two branches. This 

\begin{lemma}[Separation of planar free-boundary rays]
\label{lem:planar-ray-separation}
    Every boundary ray \(R\) of a connected component of \(\Omega_0\) satisfies \(R\cap T=\{0\}\). Distinct components cannot share a boundary ray.
\end{lemma}

\begin{proof}
    If \(R\subset T\), choose \(z\in R\setminus\{0\}\). We apply the contact gradient estimate \eqref{eq:Lipschitz-contact} to the flat minimizer \(\mathbf{u}_0\) at \(z\), which yields
    \begin{equation}\label{eq:planar-ray-separation-gradient-vanishing}
        \|\nabla \mathbf{u}_0\|_{L^\infty(B_{\rho/2}(z))}
        \le C\rho^\gamma\longrightarrow0
        \quad\text{ as }\rho\to 0^+.
    \end{equation}
    This contradicts \eqref{eq:planar-preliminary-profile}, whose gradient has one-sided magnitude \(c_\ell p|z|^\gamma>0\) on \(R\).

    If two sectors share a ray, choose a ball centered at a nonzero point of it that meets no other ray. Then \(\chi_{\Omega_0}=1\) almost everywhere in the ball. The componentwise harmonic replacement \(\mathbf{h}\) is nonnegative and has \(|\mathbf{h}|>0\) there by the strong maximum principle. Hence the phase terms agree, while harmonic replacement decreases the Dirichlet energy. Minimality forces \(\mathbf{u}_0=\mathbf{h}\), contradicting the vanishing of \(\mathbf{u}_0\) on the interior ray.
\end{proof}

We now conclude the following explicit structure for planar blow-up limits.

\begin{proposition}[Blow-up limits in \(\mathbb R^2\)]
\label{prop:explicit-planar-blow-up}
    Let \(\mathbf{u}_0\) be a nontrivial planar blow-up limit. There are a nonempty set \(K\subset\mathfrak A\) and unique vectors \(\mathbf{a}_\mu\in\mathbb S^{m-1}\cap\mathbb R_+^m\), \(\mu\in K\), such that the sectors \(S_\mu\) are pairwise disjoint, no two share a boundary ray, and
    \[
        \Omega_0=\bigcup_{\mu\in K}S_\mu,
        \qquad
        \mathbf{u}_0=\sum_{\mu\in K}\mathbf{a}_\mu\Phi_\mu
        \quad\text{ in }\mathbb R^2.
    \]
    Moreover, the union is disjoint.
\end{proposition}

\begin{proof}
    Fix a branch, and set \(\mu_\ell=\alpha_\ell+\beta\). Since \(p\beta=\pi/2\), its profile is
    \[
        \mathbf{u}_0(r,\theta)
        =\mathbf{a}_\ell c_\ell r^p\cos\bigl(p(\theta-\mu_\ell)\bigr).
    \]
    By Lemma \ref{lem:planar-ray-separation}, every nonzero point of either boundary ray is separated from \(T\) and from the other branches. Inner variations of the flat minimizing functional at this smooth interface give the Bernoulli law
    \begin{equation}\label{eq:planar-Bernoulli-rays}
        |\nabla \mathbf{u}_0|^2=\msf d_0^{2\gamma}
        \quad\text{on each boundary ray away from the origin}.
    \end{equation}
    Evaluating it on the free boundary \(\theta=\mu_\ell\pm\beta\) yields
    \[
        c_\ell p
        =|\sin(\mu_\ell-\beta)|^\gamma
        =|\sin(\mu_\ell+\beta)|^\gamma.
    \]
    Therefore we have 
    \[
        \sin(2\mu_\ell)\sin(2\beta)=0.
    \]
    Since \(0<2\beta<\pi\), this gives \(\mu_\ell\in\mathfrak A\) modulo \(2\pi\), and the same identity determines \(c_\ell=\kappa_{\mu_\ell}\). Summing the branchwise profiles proves the assertion.
\end{proof}

The remaining restriction is geometric: adjacent axial directions are separated by \(\pi/2\), whereas each sector has opening angle \(2\beta\).

\begin{corollary}[Number of planar branches]
\label{cor:number-planar-branches}
    Let \(N:=\#\,K\). Then \(1\le N\le4\), with the following refinements.
    \begin{enumerate}[label=(\roman*)]
        \item If \(0<\gamma\le1\), then \(K\) is a singleton or
        \[
            K=\{0,\pi\},
            \qquad\text{or}\qquad
            K=\left\{\frac\pi2,\frac{3\pi}{2}\right\}.
        \]
        Thus \(N\le2\), and a two-branch blow-up consists of opposite tangential sectors or opposite normal sectors. In the tangential case,
        \[
            \Omega_0=\{|x_2|<|x_1|\tan\beta\},
            \qquad
            \partial\Omega_0=\{x_2=\pm x_1\tan\beta\};
        \]
        in the normal case,
        \[
            \Omega_0=\{|x_2|>|x_1|\cot\beta\},
            \qquad
            \partial\Omega_0=\{x_2=\pm x_1\cot\beta\}.
        \]
        \item If \(\gamma>1\), then \(N\le4\). The four axial sectors have positive angular gaps, so every nonempty subset of \(\mathfrak A\) satisfies the geometric packing condition.
    \end{enumerate}
\end{corollary}

\begin{proof}
    If \(0<\gamma<1\), then \(2\beta>\pi/2\), so adjacent sectors overlap. If \(\gamma=1\), they share a boundary ray, contrary to Lemma \ref{lem:planar-ray-separation}. Hence only opposite pairs are possible when \(0<\gamma\le1\). For \(\gamma>1\), one has \(2\beta<\pi/2\), leaving positive gaps between all four axial sectors.
\end{proof}

\begin{proof}[Proof of Theorem \ref{thm:planar-blow-up-classification}]
    Combine Proposition \ref{prop:explicit-planar-blow-up} and Corollary \ref{cor:number-planar-branches}.
\end{proof}

\section{Quantitative stratification of the non-degenerate contact set}
\label{sec:quantitative-stratification}
We now apply quantitative stratification to obtain the measure and rectifiability conclusions of Theorem \ref{thm:measure-GammaND}. The argument follows Naber--Valtorta \cite{NV17}, in the free-boundary form developed by Edelen--Engelstein \cite[Theorems~1.11--1.12]{EE19}, and McCurdy's treatment of a scalar degenerate weight \cite[Theorem~1.3, Lemma~1.4, and Corollary~1.5]{McC24}. Edelen--Engelstein already include the vector-valued one-phase functional in their extension \cite[Section~1.4 and Theorem~1.15]{EE19}. Thus the geometric method, including the Jones estimate and the covering construction, is supplied by these works. Our task is to verify its analytic hypotheses for the vectorial minimizers and \(C^{1,\rm Dini}\) hypersurfaces considered here.

The mechanism is quantitative dimension reduction. A small change in the Weiss energy makes a rescaling nearly homogeneous. Homogeneity about several independent centers forces approximate translation invariance, so the centers where an additional symmetry fails must concentrate near a lower-dimensional plane. The Jones estimate measures this concentration by the energy drop. The Reifenberg and covering arguments then turn the sum of these drops into packing and rectifiability. Finally, the absence of nonzero one-dimensional homogeneous models places the non-degenerate contact set \(\Gamma_{\Sigma,\rm ND}(\mathbf{u})\) in the \((n-2)\)-stratum.

We keep the normalization and the analytic estimates in the main text. Appendix \ref{app:quantitative-stratification-inputs} records the precise rigidity, density-propagation, and covering statements being transferred from \cite{EE19,McC24}, together with the verification of the Dini errors in our geometric setting.

\subsection{Normalization and compactness}\label{subsec:QS-normalization-compactness}

Uniform estimates require a common interior scale and common control of the \(C^{1,\rm Dini}\) geometry. The following normalization is obtained locally by the contact-point rescaling; it does not impose a separate normalization on the amplitude of \(\mathbf{u}\).

\begin{definition}[Normalized class]\label{def:normalized-QS-class}
    Fix \(n\ge2\), \(m\ge1\), \(\gamma>0\), a graph radius \(r_\Sigma>0\), and a Dini modulus \(\omega_\Sigma\) on \([0,2r_\Sigma]\). Let \(\mathscr A\) be the class of pairs \((\mathbf{u},\Sigma)\) satisfying the following conditions:
    \begin{enumerate}[label=\textnormal{(\roman*)}]
        \item \(\mathbf{u}\in W^{1,2}(B_{10}(0);\R_+^m)\cap C(B_{10}(0);\R^m)\) is an \(\varepsilon_0\)-local minimizer of \(J\) in \(B_{10}(0)\) for some \(\varepsilon_0>0\), with \(0\in\Gamma_\Sigma(\mathbf{u})\), and \(10\) is a standard scale in the sense of Definition~\ref{def:standard-scale}.
        \item \(\Sigma\subset\R^n\) is a nonempty closed embedded \(C^{1,\rm Dini}\) hypersurface. At every \(q\in\Sigma\cap B_4(0)\), it has the graph representation of Definition~\ref{def.C1Dini-hypersurface}, with graph radius \(r_\Sigma\) and gradient modulus bounded by the same function \(\omega_\Sigma\).
    \end{enumerate}
\end{definition}

Throughout this section, write \(p:=1+\gamma\), consistently with \eqref{eq:planar-p-beta}. Unless otherwise stated, constants are uniform over \(\mathscr A\) and may depend on \(n,m,\gamma,r_\Sigma,\omega_\Sigma\). The following remarks collect the consequences of the normalization that will be used repeatedly.

\begin{remark}[Absolute comparisons and natural rescaling]
\label{rem:QS-uniform-minimality}
    Since \(10\) is a standard scale, every ball compactly contained in \(B_{10}(0)\) admits absolute energy comparisons with nonnegative competitors having the same boundary trace. Thus the comparison arguments below require no common value of \(\varepsilon_0\). The proofs of the growth and non-degeneracy estimates use their small-scale restriction only through these absolute comparisons; hence those arguments apply here with standard scale \(10\). By Lemma~\ref{lem:EL} and Corollary~\ref{cor:gradient-estimate}, each component is subharmonic in \(B_{10}(0)\) and harmonic in \(\Omega_{\mathbf{u}}\).

    For a contact point \(x\), set
    \[
        \mathbf{u}_{x,r}(y):=r^{-p}\mathbf{u}(x+ry),
        \qquad
        \Sigma_{x,r}:=\frac{\Sigma-x}{r}.
    \]
    The identity
    \[
        \msf d_{\Sigma_{x,r}}(y)
        =r^{-1}\msf d_\Sigma(x+ry)
    \]
    shows that this rescaling preserves the form of the functional: both bulk terms acquire the same factor \(r^{n+2\gamma}\). Consequently, absolute comparisons in \(B_{Rr}(x)\) become absolute comparisons for \(\mathbf{u}_{x,r}\) in \(B_R(0)\), with the Bernoulli weight function \(\msf d_{\Sigma_{x,r}}^{2\gamma}(x)\).

    For an original minimizer with standard scale \(r_0\), choosing \(s>0\) so that \(10s\le r_0\) and \(B_{10s}(x)\subset\subset D\) gives the normalized comparison scale. Under this change of variables, a common graph radius is divided by \(s\), and a common modulus \(\omega\) becomes \(t\mapsto\omega(st)\). This explains how the normalized class arises from the local problem.
\end{remark}

\begin{remark}[Uniform contact estimates]
\label{rem:QS-uniform-contact-estimates}
    If \(x\in\Gamma_\Sigma(\mathbf{u})\cap B_2(0)\), then we have \(\dist(x,\partial B_{10}(0))>8\). Hence every \(0<s\le1\) satisfies the radius restrictions in \eqref{eq:optimal-growth} and \eqref{eq:Lipschitz-contact}. Since \(\msf d_\Sigma(x)=0\), those estimates give
    \begin{equation}\label{eq:QS-uniform-contact-estimates}
        \|\mathbf{u}\|_{L^\infty(B_s(x))}\le Cs^p,
        \qquad
        \|\nabla\mathbf{u}\|_{L^\infty(B_s(x))}\le Cs^{p-1},
    \end{equation}
    where \(C=C(n,m,\gamma)\), independently of the geometric data. Therefore, whenever \(Rr\le1\),
    \begin{equation}\label{eq:QS-rescaled-contact-estimates}
        \|\mathbf{u}_{x,r}\|_{L^\infty(B_R)}\le CR^p,
        \qquad
        \|\nabla\mathbf{u}_{x,r}\|_{L^\infty(B_R)}\le CR^{p-1}.
    \end{equation}
    These estimates also imply
    \[
        s^{-n-2\gamma}D_{x,\mathbf{u}}(s)
        +s^{-n-1-2\gamma}H_{x,\mathbf{u}}(s)\le C,
        \qquad |W_{x,\mathbf{u}}(s)|\le C.
    \]
    Here the phase term is bounded using \(\msf d_\Sigma(z)\le|z-x|\). The estimates hold at both degenerate and non-degenerate contact points. In addition, the non-degeneracy estimates of Proposition~\ref{prop:optimal-growth-nondegeneracy} remain available on balls satisfying their stated hypotheses, with standard scale \(10\).
\end{remark}

\begin{remark}[Uniform geometry and a common radius]
\label{rem:QS-uniform-geometry}
    Fix
    \begin{equation}\label{eq:QS-common-radius}
        r_*:=\min\{1/4,r_\Sigma/8\}.
    \end{equation}
    For \(x\in\Sigma\cap B_2(0)\), \(0<s\le r_*\), and \(0<t\le s\), Proposition~\ref{prop:C1Dini-properties} gives, with common constants,
    \[
        \cL^n\bigl(B_s(x)\cap\{\msf d_\Sigma<t\}\bigr)
        \le Cs^{n-1}t,
        \qquad
        \int_{B_s(x)}\msf d_\Sigma^{2\gamma-1}\,\dx{z}
        \le Cs^{n+2\gamma-1}.
    \]
    The projection estimate is uniform on the same balls. Indeed, a nearest point to \(z\in B_s(x)\) lies in \(B_{2s}(x)\subset B_4(0)\), and the common graph radius covers this neighborhood. These facts justify the inner-variation formula and the uniform remainder bound in Lemma~\ref{lem:esmaddi}. In particular, the same correction constant can be used at every contact center in \(B_2(0)\).
\end{remark}

We now make that correction explicit. Define
\begin{equation}\label{eq.Dini-modulus-integral}
    \mathscr D_\Sigma(a,b)
    :=\int_a^b\frac{\omega_\Sigma(2t)}t\,\dx{t},
    \qquad 0\le a<b\le r_\Sigma/2,
\end{equation}
and write \(\mathscr D_\Sigma(b):=\mathscr D_\Sigma(0,b)\). Choose a common constant \(C_W\) so that the absolute value of the full geometric remainder in \eqref{eq.Wr'} is bounded by \(C_W\omega_\Sigma(2r)/r\). Thus \(C_W\) includes the factor \(2\gamma\) in that formula. For \(x\in\Gamma_\Sigma(\mathbf{u})\cap B_2(0)\) and \(0<r\le r_*\), set
\[
    \widetilde W_{x,\mathbf{u}}(r)
    :=W_{x,\mathbf{u}}(r)+C_W\mathscr D_\Sigma(r),
\]
and 
\[
    \vartheta_{\mathbf{u}}(x)
    :=\widetilde W_{x,\mathbf{u}}(0^+)=W_{x,\mathbf{u}}(0^+).
\]
The two limits agree because \(\mathscr D_\Sigma(r)\to0\) as \(r\to 0^+\).

\begin{remark}[Corrected monotonicity and the homogeneity defect]
\label{rem:QS-corrected-monotonicity}
    The function \(r\mapsto\widetilde W_{x,\mathbf{u}}(r)\) is nondecreasing on \((0,r_*]\). Combining Lemmas~\ref{lem:Weiss-monotonicity} and~\ref{lem:esmaddi}, integrating in the radius, and using polar coordinates gives, for \(0<a<b\le r_*\),
    \begin{equation}\label{eq:QS-homogeneity-defect}
        \begin{aligned}
            &\int_{B_b(x)\setminus B_a(x)}
            \frac{|(z-x)\cdot\nabla\mathbf{u}(z)-p\mathbf{u}(z)|^2}
                 {|z-x|^{n+2p}}\,\dx{z}\\
            &\qquad\le\frac12\bigl[
                \widetilde W_{x,\mathbf{u}}(b)
                -\widetilde W_{x,\mathbf{u}}(a)
            \bigr].
        \end{aligned}
    \end{equation}
    Thus a small corrected energy drop forces a small homogeneity defect on the corresponding annulus. Moreover, Proposition~\ref{prop:weighted-density-contact} gives
    \[
        \vartheta_{\mathbf{u}}(x)
        =\Theta_\Sigma(\mathbf{u};x,0^+)
        \in\{0\}\cup[\theta_1(n,\gamma),\theta_2(n,\gamma)].
    \]
    The density-gap constants are independent of the pair in \(\mathscr A\).
\end{remark}

For later use with varying hypersurfaces, we write the weight explicitly as a superscript:
\begin{equation}\label{eq:QS-Weiss-weight-notation}
    \begin{aligned}
        W^S_{a,\mathbf{w}}(R)
        :={}&R^{-n-2\gamma}
        \int_{B_R(a)}\bigl(
            |\nabla\mathbf{w}|^2
            +\msf d_S^{2\gamma}\chi_{\Omega_{\mathbf{w}}}
        \bigr)\,\dx{y}\\
        &-pR^{-n-1-2\gamma}
        \int_{\partial B_R(a)}|\mathbf{w}|^2\,\dH{n-1}.
    \end{aligned}
\end{equation}
The superscript \(S\) identifies the hypersurface defining the distance weight; the subscript \(a,\mathbf{w}\) identifies the center and the map. We suppress the superscript when \(S=\Sigma\) and when working with the original pair. A change of variables gives
\begin{equation}\label{eq:QS-Weiss-scaling}
    W^{\Sigma_{x,r}}_{0,\mathbf{u}_{x,r}}(R)
    =W_{x,\mathbf{u}}(rR).
\end{equation}
Define the scale-invariant boundary height by
\[
    \mathsf H_{x,\mathbf{u}}(r)
    :=\int_{\partial B_1(0)}|\mathbf{u}_{x,r}|^2\,\dH{n-1}
    =r^{-n-1-2\gamma}H_{x,\mathbf{u}}(r).
\]
To compare shapes, we will divide the natural rescaling by the square root of this height. The next lemma ensures that this operation does not amplify a vanishing profile. It is the counterpart of \cite[Lemma~6.4 and Corollary~6.5]{McC24}. The proof below obtains the lower bound directly from positive Weiss energy and a cutoff comparison.

\begin{lemma}[Non-collapse of the boundary height]
\label{lem:T-and-homogeneous-rescaling-comparable}
    There exist radius \(r_T>0\) that depends on \(n,m,\gamma,r_\Sigma,\omega_\Sigma\) and constants \(0<c_H\le C_H<\infty\), depending only on \(n,m,\gamma\), such that, for every \((\mathbf{u},\Sigma)\in\mathscr A\),
    \[
        c_H\le\mathsf H_{x,\mathbf{u}}(r)\le C_H
    \]
    whenever \(x\in\Gamma_{\Sigma,\rm ND}(\mathbf{u})\cap B_2(0)\) and \(0<r\le r_T\). We may choose \(r_T\le\min\{1/64,r_\Sigma/256\}\).
\end{lemma}

\begin{proof}
    Choose \(0<r_T\le\min\{1/64,r_\Sigma/256\}\) such that \(C_W\mathscr D_\Sigma(r_T)\le\theta_1/2\), where \(\theta_1=\theta_1(n,\gamma)>0\) is the density gap from Proposition~\ref{prop:weighted-density-contact}. Fix \(x\in\Gamma_{\Sigma,\rm ND}(\mathbf{u})\cap B_2(0)\) and \(0<r\le r_T\), and set
    \[
        \mathbf{v}(y):=r^{-p}\mathbf{u}(x+ry),\qquad
        \Sigma_{x,r}:=\frac{\Sigma-x}{r},
    \]
    where \(p:=1+\gamma\). Recalling \eqref{eq:QS-Weiss-weight-notation}, we write
    \[
        W^{\Sigma_{x,r}}_{0,\mathbf{v}}(1)
        :=\int_{B_1(0)}
            \bigl(|\nabla\mathbf{v}|^2+
            \msf d_{\Sigma_{x,r}}^{2\gamma}
            \chi_{\Omega_{\mathbf{v}}}\bigr)\,\dx{y}
            -p\int_{\partial B_1(0)}|\mathbf{v}|^2\,\dH{n-1}.
    \]
    Since \(\msf d_{\Sigma_{x,r}}(y)=r^{-1}\msf d_\Sigma(x+ry)\), a change of variables gives \(W^{\Sigma_{x,r}}_{0,\mathbf{v}}(1)=W_{x,\mathbf{u}}(r)\). Corrected monotonicity and the density gap therefore imply
    \begin{equation}\label{eq:QS-positive-Weiss}
        W^{\Sigma_{x,r}}_{0,\mathbf{v}}(1)
        \ge\vartheta_{\mathbf{u}}(x)-C_W\mathscr D_\Sigma(r)
        \ge\theta_1/2.
    \end{equation}
    As \(x\in B_2(0)\), its distance to \(\partial B_{10}(0)\) is larger than \(8\). Thus the radius \(2r\) satisfies the restrictions in \eqref{eq:optimal-growth} and \eqref{eq:Lipschitz-contact}, with standard scale \(10\). Since \(x\in\Sigma\), these estimates give
    \[
        \sup_{B_{2r}(x)}|\mathbf{u}|\le C(2r)^p,
        \qquad
        \|\nabla\mathbf{u}\|_{L^\infty(B_{2r}(x))}\le C(2r)^\gamma.
    \]
    Using \(\nabla\mathbf{v}(y)=r^{-\gamma}\nabla\mathbf{u}(x+ry)\), we obtain
    \begin{equation}\label{eq:QS-uniform-rescaled-bounds}
        \|\mathbf{v}\|_{L^\infty(B_2(0))}
        +\|\nabla\mathbf{v}\|_{L^\infty(B_2(0))}\le C,
    \end{equation}
    where \(C=C(n,m,\gamma)\). In particular, \(\mathsf H_{x,\mathbf{u}}(r)\le C_H\).

    We now prove the lower bound, put \(M:=\max_{\partial B_1(0)}|\mathbf{v}|\). Subharmonicity and nonnegativity of the components yield \(0\le v_i\le M\) in \(B_1(0)\), and hence \(|\mathbf{v}|\le\sqrt m\,M\) there. For \(0<\delta<1/2\), choose a Lipschitz cutoff \(0\le\eta_\delta\le1\) equal to \(0\) on \(B_{1-\delta}(0)\) and to \(1\) near \(\partial B_1(0)\), with \(|\nabla\eta_\delta|\le C/\delta\). Set \(A_\delta:=B_1(0)\setminus\overline{B_{1-\delta}(0)}\). The standard-scale property then permits comparison with \(\eta_\delta\mathbf{v}\) in \(B_1(0)\). Because \(0\in\Sigma_{x,r}\), we have \(\msf d_{\Sigma_{x,r}}(y)\le|y|\le1\) in this ball. A direct computation using the product rule gives
    \[
        |\nabla(\eta_\delta\mathbf{v})|^2
        \le2\eta_\delta^2|\nabla\mathbf{v}|^2
        +2|\mathbf{v}|^2|\nabla\eta_\delta|^2.
    \]
    Dropping the nonpositive boundary term in the Weiss energy
    and using this bound for the competitor gives
    \[
        \begin{aligned}
            W^{\Sigma_{x,r}}_{0,\mathbf{v}}(1)
            &\le J_{\Sigma_{x,r}}(\mathbf{v};B_1(0))
            \le J_{\Sigma_{x,r}}(\eta_\delta\mathbf{v};B_1(0))\\
            &\le\int_{A_\delta}
                \left(2|\nabla\mathbf{v}|^2
                +2|\mathbf{v}|^2|\nabla\eta_\delta|^2+1\right)
                \,\dx{y}\\
            &\le C|A_\delta|\left(1+\frac{M^2}{\delta^2}\right)
            \le C\left(\delta+\frac{M^2}{\delta}\right).
        \end{aligned}
    \]
    Here \(J_{\Sigma_{x,r}}\) denotes the functional with Bernoulli weight function \(\msf d_{\Sigma_{x,r}}^{2\gamma}\), and the last two inequalities use \eqref{eq:QS-uniform-rescaled-bounds} and \(|A_\delta|\le C(n)\delta\). If \(M=0\), the maximum principle gives \(\mathbf{v}=0\) on \(B_1(0)\), contradicting \eqref{eq:QS-positive-Weiss}.

    If \(0<M<1/2\), choosing \(\delta=M\) yields \(\theta_1/2\le2CM\). Together with the case \(M\ge1/2\), this gives
    \[
        M\ge c_*:=\min\{1/2,\theta_1/(4C)\}>0.
    \]
    Let \(\xi\in\partial B_1(0)\) satisfy \(|\mathbf{v}(\xi)|=M\), and let \(L=L(n,m,\gamma)\) be a Lipschitz bound given by \eqref{eq:QS-uniform-rescaled-bounds}. Set
    \[
        \rho:=\min\left\{\frac12,\frac{c_*}{2\max\{L,1\}}\right\}.
    \]
    For every \(\zeta\in\partial B_1(0)\cap B_\rho(\xi)\),
    \[
        |\mathbf{v}(\zeta)|
        \ge |\mathbf{v}(\xi)|-L|\zeta-\xi|
        \ge c_*/2.
    \]
    Since the spherical cap has measure at least \(c(n)\rho^{n-1}\), we have
    \[
        \mathsf H_{x,\mathbf{u}}(r)
        =\int_{\partial B_1(0)}|\mathbf{v}|^2\,\dH{n-1}
        \ge\frac{c_*^2}{4}c(n)\rho^{n-1}
        =:c_H>0.
    \]
    Both height constants depend only on \(n,m,\gamma\); the common geometric data enter only through the choice of \(r_T\).
\end{proof}

We may consequently define
\[
    T_{x,r}\mathbf{u}
    :=\frac{\mathbf{u}_{x,r}}{\sqrt{\mathsf H_{x,\mathbf{u}}(r)}},
    \qquad
    \|T_{x,r}\mathbf{u}\|_{L^2(\partial B_1)}=1.
\]
These normalized maps measure symmetry; energy comparisons use the natural rescalings \(\mathbf{u}_{x,r}\). More precisely, if \(h:=\mathsf H_{x,\mathbf{u}}(r)\), then \(T_{x,r}\mathbf{u}\) minimizes the rescaled functional with phase coefficient \(h^{-1}\msf d_{\Sigma_{x,r}}^{2\gamma}\), rather than the original coefficient \(\msf d_{\Sigma_{x,r}}^{2\gamma}\).

The compactness needed in quantitative stratification allows both the minimizer and the contact center to vary. Common graph control makes the hypersurfaces flatten uniformly at vanishing scales, while the height bound prevents the limit from vanishing. The next lemma is the vectorial Dini counterpart of the compactness used in \cite[Theorem~5.4 and Corollary~6.5]{McC24}; we give the details to identify which estimates remain uniform along the sequence.

\begin{lemma}[Compactness at moving contact points]\label{lem:sequential-compactness-moving-centers}
    Let \((\mathbf{u}_j,\Sigma_j)\in\mathscr A\), let \(x_j\in\Gamma_{\Sigma_j,\rm ND}(\mathbf{u}_j)\cap B_2(0)\), and let \(r_j>0\) with \(r_j\to0\). After passing to a subsequence,
    \[
        \widehat\Sigma_j:=\frac{\Sigma_j-x_j}{r_j}
        \longrightarrow T
        \quad\text{locally in }C^1\text{ graph coordinates},
    \]
    for a hyperplane \(T\) through the origin, and
    \[
        \mathbf{v}_j:=(\mathbf{u}_j)_{x_j,r_j}
        \longrightarrow\mathbf{v}
        \quad\text{locally uniformly and strongly in }
        W^{1,2}_{\rm loc}(\R^n;\R^m).
    \]
    The limit \(\mathbf{v}\) is a nonnegative, nonzero global minimizer of
    \begin{equation}\label{eq.flat-functional-T}
        J_T(\mathbf{w};B_R)
        :=\int_{B_R}
        \bigl(|\nabla\mathbf{w}|^2
            +\msf d_T^{2\gamma}\chi_{\Omega_{\mathbf{w}}}\bigr)
        \,\dx{y}.
    \end{equation}
    Moreover,
    \[
        \msf d_{\widehat\Sigma_j}^{2\gamma}
            \chi_{\Omega_{\mathbf{v}_j}}
        \longrightarrow
        \msf d_T^{2\gamma}\chi_{\Omega_{\mathbf{v}}}
        \quad\text{strongly in }L^1_{\rm loc}(\R^n),
    \]
    and, for every fixed \(R>0\),
    \begin{equation}\label{eq:QS-moving-Weiss-convergence}
        W^{\widehat\Sigma_j}_{0,\mathbf{v}_j}(R)
        \longrightarrow W^T_{0,\mathbf{v}}(R).
    \end{equation}
\end{lemma}

\begin{proof}
    After extracting convergent tangent planes, the rescaled graph functions \(g_j(\xi)=r_j^{-1}f_{x_j}(r_j\xi)\) satisfy the estimate
    \[
        \|g_j\|_{L^\infty(B_R)}
        +R\|\nabla g_j\|_{L^\infty(B_R)}
        \le2R\,\omega_\Sigma(Rr_j)\longrightarrow0.
    \]
    The uniform growth and gradient estimates give local uniform and weak \(W^{1,2}\) compactness. The proof of Proposition \ref{prop:blow-up-convergence} then applies to this sequence: its localized harmonic identity gives strong \(W^{1,2}\) convergence, and its gluing comparison gives limit minimality and convergence of the weighted phase terms. The boundary height bound makes \(\mathbf{v}\) nonzero. Convergence of the bulk integrals and local uniform convergence on spheres give convergence of the Weiss energies, also for moving centers and radii as stated.
\end{proof}

\begin{remark}[Boundary normalization and homogeneity]
\label{rem:QS-normalized-limit}
    Local uniform convergence on \(\partial B_1\) gives
    \[
        \mathsf H_{x_j,\mathbf{u}_j}(r_j)
        \longrightarrow h:=\int_{\partial B_1}|\mathbf{v}|^2\,\dH{n-1}
        \in[c_H,C_H].
    \]
    Hence the same subsequence satisfies
    \[
        T_{x_j,r_j}\mathbf{u}_j
        \longrightarrow\frac{\mathbf{v}}{\sqrt h}
        \quad\text{locally uniformly and strongly in }W^{1,2}_{\rm loc}.
    \]
    This is the boundary-normalized compactness corresponding to \cite[Corollary~6.5]{McC24}.

    Neither limit is asserted to be homogeneous. The maps \(\mathbf{u}_j\) vary with \(j\), so a fixed ratio of rescaling radii need not have vanishing energy drop. If, in addition, for every fixed \(0<a<b<\infty\),
    \[
        \widetilde W_{x_j,\mathbf{u}_j}(br_j)
        -\widetilde W_{x_j,\mathbf{u}_j}(ar_j)
        \longrightarrow0,
    \]
    then \eqref{eq:QS-homogeneity-defect}, after rescaling and passage to the limit, gives \(y\cdot\nabla\mathbf{v}=p\mathbf{v}\) a.e. on every annulus. In that case \(\mathbf{v}\) is \(p\)-homogeneous. The displayed energies are defined for all sufficiently large \(j\), since \(br_j\le r_*\) eventually.
\end{remark}

We finally prove continuity with respect to the center. The upper semicontinuity of the limiting density is the same observation used in \cite[Lemma~6.3]{McC24}. We state its uniform version on the interior region where the common geometric data are available.

\begin{lemma}[Continuity in the center]
\label{lem:center-continuity-modified-Weiss}
    Let \((\mathbf{u},\Sigma)\in\mathscr A\). For every fixed \(0<s\le r_*\), the map \(x\mapsto\widetilde W_{x,\mathbf{u}}(s)\) is continuous on \(\Gamma_\Sigma(\mathbf{u})\cap B_2(0)\). The function \(x\mapsto\vartheta_{\mathbf{u}}(x)\) is upper semicontinuous there, and \(\Gamma_{\Sigma,\rm ND}(\mathbf{u})\cap B_2(0)\) is relatively closed in \(B_2(0)\). More generally, for a local minimizer in the original domain \(D\), \(\Gamma_{\Sigma,\rm ND}(\mathbf{u})\) is relatively closed in \(D\).
\end{lemma}

\begin{proof}
    Let \(x_k\to x\) in \(\Gamma_\Sigma(\mathbf{u})\cap B_2(0)\), and fix \(0<s\le r_*\). The bulk energy density
    \[
        E(z):=|\nabla\mathbf{u}(z)|^2+\msf d_\Sigma(z)^{2\gamma}\chi_{\Omega_{\mathbf{u}}}(z)
    \]
    is locally integrable. All the balls \(B_s(x_k)\) lie in one compact subset of \(B_{10}(0)\), and their characteristic functions converge a.e. to that of \(B_s(x)\). Dominated convergence therefore gives
    \[
        \int_{B_s(x_k)}E\,\dx{z}\longrightarrow\int_{B_s(x)}E\,\dx{z}.
    \]
    For the boundary term, use
    \[
        H_{x,\mathbf{u}}(s)=s^{n-1}\int_{\mathbb S^{n-1}}|\mathbf{u}(x+s\theta)|^2\,\dH{n-1}.
    \]
    Uniform continuity of \(\mathbf{u}\) on a common compact neighborhood of these spheres proves continuity of this integral in \(x\). Since the correction \(C_W\mathscr D_\Sigma(s)\) is independent of the center, the asserted continuity of \(\widetilde W\) follows.

    Corrected monotonicity gives
    \[
        \vartheta_{\mathbf{u}}(x)=\inf_{0<s\le r_*}\widetilde W_{x,\mathbf{u}}(s).
    \]
    An infimum of continuous functions is upper semicontinuous. The contact set is relatively closed, and the density gap gives
    \[
        \Gamma_{\Sigma,\rm ND}(\mathbf{u})\cap B_2(0)
        =\{x\in\Gamma_\Sigma(\mathbf{u})\cap B_2(0):
            \vartheta_{\mathbf{u}}(x)\ge\theta_1\}.
    \]
    This proves relative closedness in \(B_2(0)\).

    For the final assertion, let non-degenerate contact points \(x_k\) converge to \(x\in D\). Relative closedness of the contact set gives \(x\in\Gamma_\Sigma(\mathbf{u})\). Choose a neighborhood of \(x\) with a common interior scale and common local Dini data, and apply the preceding argument there. Then
    \[
        \vartheta_{\mathbf{u}}(x)\ge\limsup_{k\to\infty}\vartheta_{\mathbf{u}}(x_k)\ge\theta_1>0,
    \]
    so \(x\in\Gamma_{\Sigma,\rm ND}(\mathbf{u})\).
\end{proof}

\subsection{Effective symmetry and the Jones estimate}\label{subsec:QS-effective-symmetry}

Throughout, \((\mathbf{u},\Sigma)\in\mathscr A\), and \(p=1+\gamma\) and \(r_T\) are as in Subsection~\ref{subsec:QS-normalization-compactness}. A \(k\)-symmetric profile can vary in at most \(n-k\) directions. The effective \(k\)-stratum consists of points where one cannot approximate \(\mathbf{u}\) by a profile with \(k+1\) invariant directions at any of the specified scales. We use the boundary normalization of \cite[Definitions~2.11--2.12]{McC24}; the comparison maps here have the contact homogeneity \(p\).

\begin{definition}[Vector-valued symmetry]\label{def:QS-vector-symmetry}
    Let \(0\le j\le n\) be an integer and let \(\Phi:\R^n\to\R^m\) be a nonzero map with a square-integrable angular part. We call it \(j\)-symmetric if it is \(p\)-homogeneous and invariant under translations along a \(j\)-dimensional linear subspace. For \(x\in\Gamma_{\Sigma,\rm ND}(\mathbf{u})\cap B_2(0)\) and \(0<r\le r_T\), we say that \(\mathbf{u}\) is
    \((j,\varepsilon)\)-symmetric in \(B_r(x)\) if
    \[
        \|T_{x,r}\mathbf{u}-\Phi\|_{L^2(B_1(0);\R^m)}<\varepsilon
    \]
    for some \(j\)-symmetric \(\Phi\) satisfying \(\|\Phi\|_{L^2(\partial B_1(0))}=1\).
\end{definition}
For \(0\le k\le n-1\), \(\varepsilon>0\), and \(0<\rho\le R\le r_T/8\), define
\begin{equation}\label{eq:QS-effective-strata}
    \begin{aligned}
        \Gamma^k_{\varepsilon,\rho,R}(\mathbf{u})
        :=\bigl\{x\in\Gamma_{\Sigma,\rm ND}(\mathbf{u})\cap B_2(0):\;
        &\mathbf{u}\text{ is not }(k+1,\varepsilon)\text{-symmetric}\\[-1mm]
        &\text{in }B_s(x)\text{ for every }s\in[\rho,R]\bigr\},
    \end{aligned}
\end{equation}
and
\begin{equation}\label{eq:QS-limiting-effective-stratum}
    \Gamma^k_{\varepsilon,R}(\mathbf{u})
    :=\bigcap_{0<\rho<R}\Gamma^k_{\varepsilon,\rho,R}(\mathbf{u}).
\end{equation}
The estimates below concern these effective strata. This is stronger than a dimension estimate for a set defined only through exact blow-up limits: the same bound holds before the lower scale \(\rho\) tends to zero. For a finite Borel measure \(\mu\), write
\[
    [\beta_{\mu,2}^k(x,r)]^2
    :=r^{-k-2}\inf_{L^k}
    \int_{B_r(x)}\dist(y,L^k)^2\,\dx{\mu}(y),
\]
where \(L^k\) ranges over affine \(k\)-planes.\footnote{An affine \(k\)-plane is a translate \(a+V:=\{a+v:v\in V\}\) of a \(k\)-dimensional linear subspace \(V\subset\R^n\), with \(a\in\R^n\). It need not contain the origin; an affine \(0\)-plane is a single point.} This number measures how far the distribution of centers is from lying in one plane. The following estimate controls that geometric deviation by the homogeneity defects at the centers.
\begin{remark}
    The connection with the Weiss density is geometric: if \(\mathbf{u}\) is almost homogeneous about several centers, subtracting the corresponding homogeneity identities gives small derivatives in the directions joining those centers. Unless \(\mathbf{u}\) has an additional symmetry, there can be at most \(k\) independent such directions. The covariance argument below makes this observation quantitative.
\end{remark}
This is the argument of \cite[Theorem~5.1]{EE19}, whose vector-valued extension is described in \cite[Section~1.4 and Theorem~1.15]{EE19}. For the scalar degenerate functional, the corresponding estimate is \cite[Lemma~7.4]{McC24}. We provide the calculation to specify the vector-valued defect and the Dini correction used here.
\begin{lemma}[Vectorial Jones estimate]\label{lem:beta-estimate-vectorial}
    Let \(0\le k\le n-1\) and \(\varepsilon>0\). There exist \(\delta>0\) and \(C<\infty\), depending only on \(n,m,\gamma,k,\varepsilon\), with the following property. Suppose \(x\in\Gamma_{\Sigma,\rm ND}(\mathbf{u})\cap B_2(0)\), \(0<8r\le r_T\), and \(\mathbf{u}\) is \((0,\delta)\)-symmetric but not \((k+1,\varepsilon)\)-symmetric in \(B_{8r}(x)\).
    Then every finite Borel measure \(\mu\) supported in \(\Gamma_{\Sigma,\rm ND}(\mathbf{u})\cap B_2(0)\) satisfies
    \begin{equation}\label{eq:beta-estimate-vectorial}
        [\beta_{\mu,2}^k(x,r)]^2
        \le \frac{C}{r^k}\int_{B_r(x)}
        \bigl[
        \widetilde W_{y,\mathbf{u}}(8r)
        -\widetilde W_{y,\mathbf{u}}(r)
        \bigr]\,\dx{\mu}(y).
    \end{equation}
\end{lemma}
\begin{proof}
    We first explain the role of the missing \((k+1)\)-symmetry. For sufficiently small \(\delta\), every orthonormal family \(v_1,\ldots,v_{k+1}\) satisfies
    \begin{equation}\label{eq:directional-energy-lower-short}
        r^{-n-2\gamma}\sum_{i=1}^{k+1}\int_{B_{4r}(x)\setminus B_{3r}(x)}|\partial_{v_i}\mathbf{u}|^2\,\dx{z}\ge c\,\mathsf H_{x,\mathbf{u}}(8r)\ge c\,c_H.
    \end{equation}
    Indeed, otherwise there would be normalized maps \(F_j:=T_{x_j,8r_j}\mathbf{u}_j\), symmetry errors tending to zero, and orthonormal families \(v_{j,1},\ldots,v_{j,k+1}\) for which the corresponding directional energies on \(B_{1/2}(0)\setminus B_{3/8}(0)\) tend to zero. Estimate~\eqref{eq:QS-rescaled-contact-estimates} and Lemma~\ref{lem:T-and-homogeneous-rescaling-comparable} give a subsequence converging uniformly on \(\overline{B_1}\) and weakly in \(W^{1,2}(B_1)\) to a map \(F\) with \(\|F\|_{L^2(\partial B_1)}=1\). Its vanishing distance from \(p\)-homogeneous maps implies that \(F\) is \(p\)-homogeneous in \(B_1\). Passing to a subsequence in the orthonormal families gives \(\partial_{v_i}F=0\) on the annulus; homogeneity propagates these identities throughout \(B_1\). Consequently, the homogeneous extension of \(F\) is a normalized \((k+1)\)-symmetric map. This contradicts the assumed failure of \((k+1,\varepsilon)\)-symmetry and proves \eqref{eq:directional-energy-lower-short}. We may assume \(\mu(B_r(x))>0\). Let
    \[
        X:=\frac{1}{\mu(B_r(x))}\int_{B_r(x)}y\,\dx{\mu}(y),
    \]
    and define
    \[
        B(v,w):=\int_{B_r(x)}
        (v\cdot(y-X))(w\cdot(y-X))\,\dx{\mu}(y).
    \]
    Write \(\lambda_1\ge\cdots\ge\lambda_n\ge0\) for the eigenvalues of \(B\), with orthonormal eigenvectors \(v_1,\ldots,v_n\). Centering the measure and using the eigenvector identity gives, for each component \(u_\alpha\),
    \[
        \lambda_i\partial_{v_i}u_\alpha(z)
        =-\int_{B_r(x)}(v_i\cdot(y-X))
        \bigl[(z-y)\cdot\nabla u_\alpha(z)-pu_\alpha(z)\bigr]
        \,\dx{\mu}(y).
    \]
    Cauchy--Schwarz, followed by summation over the components, therefore yields
    \begin{equation}\label{eq:vectorial-covariance-identity}
        \lambda_i|\partial_{v_i}\mathbf{u}(z)|^2
        \le\int_{B_r(x)}
        |(z-y)\cdot\nabla \mathbf{u}(z)-p\mathbf{u}(z)|^2\,\dx{\mu}(y).
    \end{equation}
    The inequality is also immediate when \(\lambda_i=0\). For \(z\in A_r:=B_{4r}(x)\setminus B_{3r}(x)\) and \(y\in B_r(x)\cap\operatorname{spt}\mu\), we have \(2r<|z-y|<5r\). Hence \eqref{eq:QS-homogeneity-defect} gives
    \[
        \int_{A_r}|(z-y)\cdot\nabla\mathbf{u}(z)-p\mathbf{u}(z)|^2\,\dx{z}
        \le Cr^{n+2p}\bigl[
        \widetilde W_{y,\mathbf{u}}(8r)-\widetilde W_{y,\mathbf{u}}(r)
        \bigr].
    \]
    Integrating \eqref{eq:vectorial-covariance-identity} over \(A_r\), summing over \(i=1,\ldots,k+1\), and using \eqref{eq:directional-energy-lower-short}, we obtain
    \[
        \lambda_{k+1}
        \le Cr^2\int_{B_r(x)}
        \bigl[
        \widetilde W_{y,\mathbf{u}}(8r)-\widetilde W_{y,\mathbf{u}}(r)
        \bigr]\,\dx{\mu}(y).
    \]
    Finally, the affine plane through \(X\) spanned by \(v_1,\ldots,v_k\) minimizes the quadratic distance, so
    \[
        [\beta_{\mu,2}^k(x,r)]^2
        =r^{-k-2}\sum_{i=k+1}^n\lambda_i
        \le(n-k)r^{-k-2}\lambda_{k+1}.
    \]
    Combining the last two inequalities proves \eqref{eq:beta-estimate-vectorial}.
\end{proof}

\subsection{Packing and rectifiability}\label{subsec:QS-packing-rectifiability}

Packing and rectifiability use two different consequences of the Jones estimate. Packing bounds the number of balls needed to cover the stratum. Rectifiability also requires control of how its approximating planes change with scale. The discrete and rectifiable Reifenberg theorems of Naber--Valtorta \cite{NV17}, in the forms recalled in \cite[Theorems~2.1--2.2]{EE19}, provide these two implications.

The following result is the present version of \cite[Theorems~1.11--1.12]{EE19} and \cite[Theorem~1.3]{McC24}. Its covering input is verified in Proposition \ref{prop:EE-covering-package-Dini}. The rectifiability argument below follows \cite[Section~8, proof of Theorem~1.12]{EE19}.

\begin{theorem}[Effective-stratum estimates]\label{thm:minkowski-rectifiability-effective-strata}
    Let \(0\le k\le n-1\), \(\varepsilon>0\), and \(0<R\le r_T/8\). There is \(C=C(n,m,\gamma,k,\varepsilon,R,r_\Sigma,\omega_\Sigma)\) such that, for \(0<\rho\le r\le R\),
    \begin{equation}\label{eq:minkowski-effective-strata}
        \cL^n\!\left(
        N_r\bigl(
        \Gamma^k_{\varepsilon,\rho,R}(\mathbf{u})\cap B_1(0)
        \bigr)\right)
        \le Cr^{n-k}.
    \end{equation}
    Moreover, \(\Gamma^k_{\varepsilon,R}(\mathbf{u})\cap B_1(0)\) is countably \(k\)-rectifiable, and
    \begin{equation}\label{eq:Hk-effective-strata}
        \cH^k\bigl(
        \Gamma^k_{\varepsilon,R}(\mathbf{u})\cap B_1(0)
        \bigr)\le C.
    \end{equation}
\end{theorem}
\begin{proof}
    Remarks~\ref{rem:QS-uniform-contact-estimates} and \ref{rem:QS-corrected-monotonicity} give the uniform range
    \begin{equation}\label{eq:density-range-effective-strata-proof}
        0<\theta_1\le\vartheta_{\mathbf{u}}(x)
        \le\widetilde W_{x,\mathbf{u}}(s)\le E_0,
    \end{equation}
    for \(x\in\Gamma_{\Sigma,\rm ND}(\mathbf{u})\cap B_2(0)\) and \(0<s\le r_T\). Let \(r_{\rm qs}\) be the radius in Proposition~\ref{prop:EE-covering-package-Dini}. Choose a fixed \(a_0>0\) with \(20a_0\le R\), \(a_0\le r_{\rm qs}\), and all root balls \(B_{20a_0}(a)\), \(a\in B_1(0)\), contained in \(B_2(0)\). Cover \(B_1(0)\) by finitely many balls of radius \(a_0\). For \(\rho\le r\le a_0\), Proposition \ref{prop:EE-covering-package-Dini} covers \(\Gamma^k_{\varepsilon,\rho,R}(\mathbf{u})\cap B_1(0)\) by at most \(Cr^{-k}\) balls of radius \(r\). Their concentric doubles cover its \(r\)-neighborhood, proving \eqref{eq:minkowski-effective-strata}. For \(r>a_0\), the estimate follows by enlarging \(C\).
    
    Set \(S:=\Gamma^k_{\varepsilon,R}(\mathbf{u})\cap B_1(0)\). Applying the same cover for every \(r\to 0^+\) gives
    \begin{equation}\label{eq:Hk-bound-effective-proof}
        \cH^k(S)\le C.
    \end{equation}
    If \(k=0\), this is a bound on the number of points, so rectifiability is immediate. Assume \(k\ge1\). Let \(\delta_J\) be the constant in Lemma \ref{lem:beta-estimate-vectorial}, and let \(\eta_{\rm rig}\) and \(r_{\rm rig}\) be supplied by Lemma~\ref{lem:quantitative-rigidity-Dini} with \(\delta=\delta_J\). Following \cite[Section~8]{EE19}, divide \(S\) into density slices
    \[
        S_q:=\{x\in S:q\zeta\le\vartheta_{\mathbf{u}}(x)<(q+1)\zeta\},
        \qquad q\in\mathbb Z,
    \]
    where \(0<2\zeta\le\eta_{\rm rig}\) will be chosen small. There are only finitely many nonempty slices. For each \(x\in S_q\), Lemma~\ref{lem:center-continuity-modified-Weiss} and the definition of \(\vartheta_{\mathbf{u}}(x)\) allow us to choose \(a>0\) so small that
    \begin{equation}\label{eq:local-Weiss-oscillation-bound}
        0\le\widetilde W_{y,\mathbf{u}}(s)-\vartheta_{\mathbf{u}}(y)\le2\zeta
    \end{equation}
    for \(y\in S_q\cap B_{4a}(x)\) and \(0<s\le128a\). We also require \(128a<R\) and \(4a\le\min\{r_{\rm qs},r_{\rm rig}\}\). Write \(E:=S_q\cap B_{2a}(x)\) and \(\mu:=\cH^k\mres E\). Estimate~\eqref{eq:local-Hk-effective-stratum} in Proposition~\ref{prop:EE-covering-package-Dini} gives
    \begin{equation}\label{eq:local-upper-growth-mu}
        \mu(B_s(z))\le Cs^k,\qquad z\in\R^n,\quad 0<s\le4a.
    \end{equation}
    For \(z\in E\) and \(0<s\le a/16\), apply Lemma~\ref{lem:quantitative-rigidity-Dini} at radius \(8s\); by \eqref{eq:local-Weiss-oscillation-bound}, the required drop satisfies
    \[
        \widetilde W_{z,\mathbf{u}}(64s)
        -\widetilde W_{z,\mathbf{u}}(8\eta_{\rm rig}s)\le2\zeta.
    \]
    Thus \(\mathbf{u}\) is \((0,\delta_J)\)-symmetric in \(B_{8s}(z)\). Membership in \(S\) excludes \((k+1,\varepsilon)\)-symmetry there, so Lemma~\ref{lem:beta-estimate-vectorial} yields
    \begin{equation}\label{eq:Jones-on-density-slice}
        [\beta_{\mu,2}^k(z,s)]^2
        \le Cs^{-k}\int_{B_s(z)}
        \bigl[\widetilde W_{y,\mathbf{u}}(8s)
        -\widetilde W_{y,\mathbf{u}}(s)\bigr]\,\dx{\mu}(y).
    \end{equation}
    For \(w\in B_{a/4}(x)\) and \(0<t\le a/64\), integrate this inequality, use Fubini and \eqref{eq:local-upper-growth-mu}, and telescope the corrected energy using Remark~\ref{rem:QS-corrected-monotonicity}:
    \[
        \begin{aligned}
            &\int_0^{2t}\int_{B_t(w)}
            [\beta_{\mu,2}^k(z,s)]^2
            \,\dx{\mu}(z)\,\frac{\dx{s}}s\\
            &\quad\le C\int_{B_{3t}(w)}
            \int_0^{2t}
            \bigl[\widetilde W_{y,\mathbf{u}}(8s)
            -\widetilde W_{y,\mathbf{u}}(s)\bigr]
            \frac{\dx{s}}s\,\dx{\mu}(y)\\
            &\quad\le C\int_{B_{3t}(w)}
            \bigl[\widetilde W_{y,\mathbf{u}}(16t)
            -\vartheta_{\mathbf{u}}(y)\bigr]\,\dx{\mu}(y)
            \le C\zeta t^k.
        \end{aligned}
    \]
    Choose \(C\zeta\le\delta_{\rm RR}\), the threshold in \cite[Theorem~2.2]{EE19}. After rescaling a ball \(B_{a/128}(x)\), that theorem makes \(S_q\cap B_{a/128}(x)\) countably \(k\)-rectifiable. A countable covering in \(x\), followed by the union over \(q\), proves rectifiability of \(S\).
\end{proof}
The density slices have a simple purpose: on each sufficiently small portion of a slice, the total remaining energy drop is uniformly small. The Jones estimate turns this into the small square-function bound required by rectifiable Reifenberg. The preceding Hausdorff-measure estimate alone would not give this conclusion.

\subsection{The dimension of the non-degenerate contact set}\label{subsec:QS-contact-set-dimension}
The effective-stratum theorem applies to every \(k\). To identify the relevant value here, we exclude a profile with \(n-1\) invariant directions. Such a profile depends on one variable; harmonicity makes it affine on each positivity interval, which is incompatible with degree \(p>1\). This is the argument of \cite[Lemma~8.1]{McC24}, applied to each component.

\begin{lemma}[Exclusion of codimension-one symmetry]\label{lem:no-nminusone-symmetric-blowup}
    Let \(T\subset\R^n\) be a linear hyperplane and let
    \(\mathbf{v}:B_2(0)\to\R_+^m\) minimize \(J_T\) locally, with \(\mathbf{v}(0)=0\).
    If \(\mathbf{v}\) is \(p\)-homogeneous in \(B_1(0)\) and invariant there
    under translations along an \((n-1)\)-dimensional linear subspace,
    then \(\mathbf{v}\equiv0\) in \(B_1(0)\).
\end{lemma}
\begin{proof}
    Write \(\mathbf{v}(x)=\Psi(x\cdot e)\), where \(e\) is normal to the invariant plane. Homogeneity gives
    \[
        \Psi(t)=A^+t_+^p+A^-(-t)_+^p.
    \]
    By part~\ref{item:harmonic} of Corollary~\ref{cor:gradient-estimate}, the components of \(\Psi\) are harmonic, and hence affine, on either half-interval where \(\Psi\ne0\). Since \(p>1\), this forces \(A^+=A^-=0\).
\end{proof}

Compactness makes this exclusion uniform at small scales. Each normalized rescaling centered on \(\Gamma_{\Sigma,\rm ND}(\mathbf{u})\) then stays a fixed distance from the forbidden models. This is the step that permits a measure bound for the entire contact set, rather than a countable union of effective strata.

\begin{theorem}[Measure bound for the non-degenerate contact set]\label{thm:ND-quantitative-bound}
    There are \(\varepsilon_*=\varepsilon_*(n,m,\gamma)>0\) and \(0<\widetilde r\le r_T/8\), with \(\widetilde r=\widetilde r(n,m,\gamma,r_\Sigma,\omega_\Sigma)\), such that
    \begin{equation}\label{eq:ND-contained-effective-stratum}
        \Gamma_{\Sigma,\rm ND}(\mathbf{u})\cap B_1(0)
        \subset
        \Gamma^{n-2}_{\varepsilon_*,\rho,\widetilde r}(\mathbf{u})\cap B_1(0)
        \qquad\text{for every }0<\rho<\widetilde r.
    \end{equation}
    Consequently, this set is countably \((n-2)\)-rectifiable and
    \begin{equation}\label{eq:ND-final-measure-bound}
        \cH^{n-2}\bigl(\Gamma_{\Sigma,\rm ND}(\mathbf{u})\cap B_1(0)\bigr)
        \le C.
    \end{equation}
    Moreover,
    \begin{equation}\label{eq:ND-final-minkowski-bound}
        \cL^n\!\left(
        N_r(\Gamma_{\Sigma,\rm ND}(\mathbf{u})\cap B_1(0))
        \right)\le Cr^2,
        \qquad 0<r\le\widetilde r,
    \end{equation}
    where \(C=C(n,m,\gamma,r_\Sigma,\omega_\Sigma)\).
\end{theorem}
\begin{proof}
    Consider maps minimizing \(J_T\) against every ball-supported comparison in \(B_2(0)\), allowing \(T\) to vary over hyperplanes through the origin. Require that the maps vanish at the origin and satisfy fixed universal growth and Lipschitz bounds on \(B_2(0)\), with boundary height in \([c_H,C_H]\). Their boundary-normalized restrictions form a compact family in \(L^2(B_1(0))\), by the compactness argument in the proof of Lemma~\ref{lem:sequential-compactness-moving-centers}, applied to flat weights, and Remark~\ref{rem:QS-normalized-limit}. The normalized \((n-1)\)-symmetric maps also form a compact family. Indeed, they have the form
    \[
        \Phi(x)=A^+(x\cdot e)_+^p+A^-(-x\cdot e)_+^p,
        \qquad e\in\mathbb S^{n-1},
    \]
    and the boundary normalization bounds \(A^\pm\). Lemma \ref{lem:no-nminusone-symmetric-blowup} makes these two compact families disjoint. Their \(L^2(B_1(0))\)-distance is therefore positive and depends only on \(n,m,\gamma\); choose \(\varepsilon_*\) smaller than one third of it. By Lemma~\ref{lem:sequential-compactness-moving-centers} and Remark~\ref{rem:QS-normalized-limit}, all \(T_{x,r}\mathbf{u}\), with \(x\in\Gamma_{\Sigma,\rm ND}(\mathbf{u})\cap B_1(0)\), lie within \(\varepsilon_*\) of the first family once \(r\le\widetilde r\), uniformly over \(\mathscr A\). Otherwise a sequence of vanishing scales would converge to a member of that family. Thus \(\mathbf{u}\) is not \((n-1,\varepsilon_*)\)-symmetric in any such \(B_r(x)\). This proves \eqref{eq:ND-contained-effective-stratum}.
    
    By \eqref{eq:QS-limiting-effective-stratum}, the inclusions \eqref{eq:ND-contained-effective-stratum} place \(\Gamma_{\Sigma,\rm ND}(\mathbf{u})\cap B_1(0)\) in \(\Gamma^{n-2}_{\varepsilon_*,\widetilde r}(\mathbf{u})\cap B_1(0)\). The Hausdorff-measure and rectifiability conclusions now follow from Theorem~\ref{thm:minkowski-rectifiability-effective-strata} with \(k=n-2\), \(\varepsilon=\varepsilon_*\), and \(R=\widetilde r\). For the neighborhood estimate, apply \eqref{eq:minkowski-effective-strata} with any \(0<\rho<\min\{r,\widetilde r\}\).
\end{proof}
\begin{corollary}
    If \(n=2\), then \(\Gamma_{\Sigma,\rm ND}(\mathbf{u})\cap B_1(0)\) is finite.
\end{corollary}
\begin{proof}
    In this case \(\cH^{n-2}=\cH^0\) counts points, so \eqref{eq:ND-final-measure-bound} bounds their number.
\end{proof}
Finally, let \(K\subset\subset D\) be compact. Lemma~\ref{lem:center-continuity-modified-Weiss} makes \(\Gamma_{\Sigma,\rm ND}(\mathbf{u})\cap K\) compact, and Remark~\ref{rem:QS-uniform-minimality} supplies a finite cover by contact balls that can be rescaled into the normalized class. Rescaling \eqref{eq:ND-final-measure-bound} gives
\[
    \cH^{n-2}\bigl(\Gamma_{\Sigma,\rm ND}(\mathbf{u})\cap K\bigr)<\infty,
\]
together with countable \((n-2)\)-rectifiability. This proves Theorem \ref{thm:measure-GammaND}.

\subsubsection*{Acknowledgement}
\hyphenpenalty=10
\sloppy
L.~Du is supported by National Natural Science Foundation of China under Grants 12125102,12526202, 12671245, 12631009. C.~Yang is supported by National Natural Science Foundation of China under Grants 12601435.

\subsubsection*{AI declarations}
OpenAI’s GPT was used only for grammar, spelling and citation checks, with all edits reviewed to safeguard mathematical reasoning. 

\subsubsection*{Data availability} No data was used in this research.

\appendix

\section{Stabilization of the local active set}\label{appendix:local-active-stabilization}

The main aim of this section is to prove the following lemma. 

\begin{lemma}[Stabilization of the local active set]
    Let \(x_0\in \Gamma_{\Sigma,\rm ND}(\mathbf{u})\) satisfy \eqref{eq:touching-branch-assumption}, and set 
    \[
        \bar R_{x_0}:=\frac 14\min\left\{
            r_0,r_\Sigma,R_{x_0},\dist(x_0,\partial D)
        \right\}.
    \]
    There exist \(r_{x_0}\in(0,\bar R_{x_0})\) and a nonempty set \(\mathcal I_{\mathbf{u}}(x_0)\subset\{1,\dots,m\}\) such that
    \[
        \mathcal I_{\mathbf{u}}(x_0,r)=\mathcal I_{\mathbf{u}}(x_0)\quad\text{ for every }0<r<r_{x_0}.
    \]
\end{lemma}

\begin{proof}
    Note that every touching component has a nonempty active index set. Indeed, at least one component of \(\mathbf{u}\) is positive at some point, and positivity propagates throughout that component by the strong maximum principle.

    If \(0<r_1<r_2<\bar R_{x_0}\), every \(\cO_1\in\mathscr C_{\mathbf{u}}(x_0;r_1)\) is contained in some \(\cO_2\in\mathscr C_{\mathbf{u}}(x_0;r_2)\). The strong maximum principle gives \(I_{\mathbf{u}}(\cO_1)\subseteq I_{\mathbf{u}}(\cO_2)\). Consequently,
    \[
        \varnothing\ne\mathcal I_{\mathbf{u}}(x_0,r_1)
        \subseteq\mathcal I_{\mathbf{u}}(x_0,r_2)
        \subseteq\{1,\dots,m\}.
    \]
    Choose \(\rho\in(0,\bar R_{x_0})\) for which \(\#\,\mathcal I_{\mathbf{u}}(x_0,\rho)\) is minimal. Nestedness and minimality imply \(\mathcal I_{\mathbf{u}}(x_0,r)=\mathcal I_{\mathbf{u}}(x_0,\rho)\) for every \(0<r<\rho\). Taking \(r_{x_0}=\rho/2\) and \(\mathcal I_{\mathbf{u}}(x_0)=\mathcal I_{\mathbf{u}}(x_0,\rho)\) proves the assertion.
\end{proof}

\section{\texorpdfstring{Planar configurations for \(\gamma>1\)}{Planar configurations for gamma>1}}\label{sec:planar-blow-up-appendix}

In the case \(n=2\) and \(\gamma>1\), the nontrivial blow-up limits \(\mathbf{u}_0\) are linear combinations of the four homogeneous solutions \(\Phi_0,\Phi_{\pi/2},\Phi_\pi,\Phi_{3\pi/2}\) defined in \eqref{eq:planar-canonical-profile}. The fifteen possible structural forms of \(\mathbf{u}_0\) are listed in Table~\ref{tab:planar-blowups-gamma-large}. Each form is determined by a subset \(K\subset\{0,\pi/2,\pi,3\pi/2\}\) of the angles of the branches, and the corresponding coefficients \(\mathbf{a}_\theta\in\mathbb S^{m-1}\cap\R_+^m\) for \(\theta\in K\). The number of branches is \(\#\,K=1,2,3,4\).

\begin{table}[!htbp]
    \centering
    \caption{The fifteen possible forms of a nontrivial planar blow-up limit when \(\gamma>1\).}
    \label{tab:planar-blowups-gamma-large}

    \scriptsize
    \setlength{\tabcolsep}{4pt}
    \renewcommand{\arraystretch}{1.15}

    \begin{tabular}{
        @{}
        c
        >{\centering\arraybackslash}p{0.26\textwidth}
        >{\raggedright\arraybackslash}p{0.62\textwidth}
        @{}
    }
        \toprule
        \(\#\,K\)
        &
        \(K\)
        &
        Corresponding form of \(\mathbf{u}_0\)
        \\
        \midrule


        \(1\)
        &
        \(\{0\}\)
        &
        \(\displaystyle \mathbf{u}_0=\mathbf{a}_0\Phi_0\)
        \\

        \(1\)
        &
        \(\left\{\frac{\pi}{2}\right\}\)
        &
        \(\displaystyle \mathbf{u}_0=\mathbf{a}_{\pi/2}\Phi_{\pi/2}\)
        \\

        \(1\)
        &
        \(\{\pi\}\)
        &
        \(\displaystyle \mathbf{u}_0=\mathbf{a}_\pi\Phi_\pi\)
        \\

        \(1\)
        &
        \(\left\{\frac{3\pi}{2}\right\}\)
        &
        \(\displaystyle \mathbf{u}_0=\mathbf{a}_{3\pi/2}\Phi_{3\pi/2}\)
        \\

        \midrule


        \(2\)
        &
        \(\left\{0,\frac{\pi}{2}\right\}\)
        &
        \(\displaystyle
            \mathbf{u}_0
            =
           \mathbf{a}_0\Phi_0
            +
            \mathbf{a}_{\pi/2}\Phi_{\pi/2}
        \)
        \\

        \(2\)
        &
        \(\left\{\frac{\pi}{2},\pi\right\}\)
        &
        \(\displaystyle
            \mathbf{u}_0
            =
            \mathbf{a}_{\pi/2}\Phi_{\pi/2}
            +
            \mathbf{a}_\pi\Phi_\pi
        \)
        \\

        \(2\)
        &
        \(\left\{\pi,\frac{3\pi}{2}\right\}\)
        &
        \(\displaystyle
            \mathbf{u}_0
            =
            \mathbf{a}_\pi\Phi_\pi
            +
            \mathbf{a}_{3\pi/2}\Phi_{3\pi/2}
        \)
        \\

        \(2\)
        &
        \(\left\{0,\frac{3\pi}{2}\right\}\)
        &
        \(\displaystyle
            \mathbf{u}_0
            =
            \mathbf{a}_0\Phi_0
            +
            \mathbf{a}_{3\pi/2}\Phi_{3\pi/2}
        \)
        \\

        \(2\)
        &
        \(\{0,\pi\}\)
        &
        \(\displaystyle
            \mathbf{u}_0
            =
            \mathbf{a}_0\Phi_0
            +
            \mathbf{a}_\pi\Phi_\pi
        \)
        \\

        \(2\)
        &
        \(\left\{\frac{\pi}{2},\frac{3\pi}{2}\right\}\)
        &
        \(\displaystyle
            \mathbf{u}_0
            =
            \mathbf{a}_{\pi/2}\Phi_{\pi/2}
            +
            \mathbf{a}_{3\pi/2}\Phi_{3\pi/2}
        \)
        \\

        \midrule


        \(3\)
        &
        \(\left\{0,\frac{\pi}{2},\pi\right\}\)
        &
        \(\displaystyle
            \begin{aligned}
                \mathbf{u}_0={}&
                \mathbf{a}_0\Phi_0
                +
                \mathbf{a}_{\pi/2}\Phi_{\pi/2}
                +
                \mathbf{a}_\pi\Phi_\pi
            \end{aligned}
        \)
        \\

        \(3\)
        &
        \(\left\{\frac{\pi}{2},\pi,\frac{3\pi}{2}\right\}\)
        &
        \(\displaystyle
            \begin{aligned}
                \mathbf{u}_0={}&
                \mathbf{a}_{\pi/2}\Phi_{\pi/2}
                +
                \mathbf{a}_\pi\Phi_\pi
                +
                \mathbf{a}_{3\pi/2}\Phi_{3\pi/2}
            \end{aligned}
        \)
        \\

        \(3\)
        &
        \(\left\{0,\pi,\frac{3\pi}{2}\right\}\)
        &
        \(\displaystyle
            \begin{aligned}
                \mathbf{u}_0={}&
                \mathbf{a}_0\Phi_0
                +
                \mathbf{a}_\pi\Phi_\pi
                +
                \mathbf{a}_{3\pi/2}\Phi_{3\pi/2}
            \end{aligned}
        \)
        \\

        \(3\)
        &
        \(\left\{0,\frac{\pi}{2},\frac{3\pi}{2}\right\}\)
        &
        \(\displaystyle
            \begin{aligned}
                \mathbf{u}_0={}&
                \mathbf{a}_0\Phi_0
                +
                \mathbf{a}_{\pi/2}\Phi_{\pi/2}
                +
                \mathbf{a}_{3\pi/2}\Phi_{3\pi/2}
            \end{aligned}
        \)
        \\

        \midrule


        \(4\)
        &
        \(\left\{
            0,\frac{\pi}{2},\pi,\frac{3\pi}{2}
        \right\}\)
        &
        \(\displaystyle
            \begin{aligned}
                \mathbf{u}_0={}&
                \mathbf{a}_0\Phi_0
                +
                \mathbf{a}_{\pi/2}\Phi_{\pi/2}
                +
                \mathbf{a}_\pi\Phi_\pi
                +
                \mathbf{a}_{3\pi/2}\Phi_{3\pi/2}
            \end{aligned}
        \)
        \\

        \bottomrule
    \end{tabular}
\end{table}

\section{Quantitative-stratification inputs}\label{app:quantitative-stratification-inputs}

We present the standard rigidity, symmetry, and covering arguments of Naber--Valtorta and Edelen--Engelstein \cite{NV17,EE19}, in the form needed for the present distance weight. We follow McCurdy's scalar degenerate adaptation \cite[Section~7]{McC24}; for positive weights, the vector-valued framework is already included in \cite[Section~1.4 and Theorem~1.15]{EE19}. The proofs below verify the analytic substitutions for the \(C^{1,\rm Dini}\) setting and refer to these sources for the geometric construction.

Throughout, \((\mathbf{u},\Sigma)\in\mathscr A\), \(p=1+\gamma\), and \(\mathcal F:=\Gamma_{\Sigma,\rm ND}(\mathbf{u})\cap B_2(0)\). We use Lemma~\ref{lem:sequential-compactness-moving-centers} and Remark~\ref{rem:QS-normalized-limit} for compactness and boundary normalization. Their convergence statements also give fixed-radius Weiss convergence at convergent rescaled contact centers; the height bounds in Lemma~\ref{lem:T-and-homogeneous-rescaling-comparable} pass to those centers. All radii below are chosen so that the enlarged balls and corrected energies lie in the ranges fixed in Subsection~\ref{subsec:QS-normalization-compactness}.

\subsection{Rigidity and the accumulation of symmetries}

A small Weiss drop forces approximate homogeneity. This is the compactness principle of \cite[Lemma~3.1]{EE19} and \cite[Lemma~7.3]{McC24}.

\begin{lemma}[Quantitative rigidity with Dini error]\label{lem:quantitative-rigidity-Dini}
    For every \(\delta>0\), there are \(\eta_{\rm rig}\in(0,1)\) and \(0<r_{\rm rig}\le r_T\) such that
    \[
        \widetilde W_{x,\mathbf{u}}(8r)
        -\widetilde W_{x,\mathbf{u}}(\eta_{\rm rig}r)\le\eta_{\rm rig}
    \]
    implies that \(\mathbf{u}\) is \((0,\delta)\)-symmetric in \(B_r(x)\), whenever \(x\in\mathcal F\) and \(0<8r\le r_{\rm rig}\). Here \(\eta_{\rm rig}\) depends only on \(n,m,\gamma,\delta\); the radius \(r_{\rm rig}\) may also depend on \(r_\Sigma,\omega_\Sigma\).
\end{lemma}
\begin{proof}
    Apply the compactness argument of the cited lemmas first to the flat limiting class. Height non-collapse gives a nonzero limit, and vanishing Weiss drop makes it \(p\)-homogeneous by \eqref{eq.flat-Weiss-U0}. This gives a flat rigidity threshold with error \(\delta/2\), depending only on \(n,m,\gamma,\delta\). Reduce that threshold by a fixed factor. Then \eqref{eq:QS-moving-Weiss-convergence}, together with \(C_W\mathscr D_\Sigma(Rr)\to0\) for each fixed \(R\), transfers the criterion to \(\mathscr A\) at sufficiently small physical scales. Otherwise, a counterexample sequence would converge to a flat profile satisfying the \(\delta/2\)-criterion, contradicting failure of \((0,\delta)\)-symmetry. The geometric data enter only through \(r_{\rm rig}\).
\end{proof}

Homogeneity about independent centers produces translation invariance in the directions joining them. We use the following form of the symmetry argument in \cite[proof of Lemma~3.3]{EE19} and \cite[proof of Lemma~7.6]{McC24}.

\begin{lemma}[Quantitative cone splitting]\label{lem:cone-splitting-Dini}
    Let \(0\le j\le n-1\) and \(\varepsilon,\tau>0\). There are \(\delta_{\rm cs}>0\) and \(0<r_{\rm cs}\le r_T\) such that the following holds. Suppose \(0<8r\le r_{\rm cs}\) and \(x_0,\ldots,x_j\in\mathcal F\), with \(x_i\in B_r(x_0)\) for \(i\ge1\), and
    \[
        \dist\left(x_i,\operatorname{aff}
        \{x_0,\ldots,x_{i-1}\}\right)\ge\tau r,
        \qquad i=1,\ldots,j.
    \]
    If \(\mathbf{u}\) is \((0,\delta_{\rm cs})\)-symmetric in every \(B_{4r}(x_i)\), then it is \((j,\varepsilon)\)-symmetric in \(B_r(x_0)\). The same conclusion with \(j+1\) holds after adjoining one further center in \(\mathcal F\cap B_r(x_0)\) at distance at least \(\tau r\) from the preceding affine span, provided the same almost-homogeneity assumption holds there.
\end{lemma}
\begin{proof}
    Use the compactness argument of the cited proofs. A counterexample sequence with vanishing symmetry errors and scales would yield a nonzero limit \(\mathbf{v}\), homogeneous about independent centers \(y_0=0,y_1,\ldots,y_j\). Subtracting their Euler identities gives
    \[
        (y_i-y_0)\cdot\nabla\mathbf{v}=0,\qquad i=1,\ldots,j.
    \]
    Homogeneity extends these identities to the homogeneous profile, which is therefore \(j\)-symmetric. Boundary-normalized convergence contradicts the assumed failure of \((j,\varepsilon)\)-symmetry. An additional independent center gives one more invariant direction.
\end{proof}

\subsection{Density propagation and covering}

The covering argument uses the stronger density-propagation statement of \cite[Lemma~3.3 and Theorem~3.2]{EE19}, as adapted in \cite[Lemmas~7.5--7.6]{McC24}: independent high-density centers determine an affine plane near which the density remains high and the effective stratum must concentrate.
\begin{lemma}[Density propagation and the covering dichotomy]\label{lem:EE-density-propagation}
    Fix \(0\le k\le n-1\), \(\varepsilon>0\), \(0<\varrho<1/20\), \(0<\lambda<1\), and \(\eta'>0\). There are \(\beta_{\rm dp}\in(0,\varrho)\), \(\eta\in(0,\lambda\varrho)\), and a sufficiently small radius \(0<r_{\rm dp}\le r_T/160\) with the following property. Let \(a\in\mathcal F\cap B_{3/2}(0)\), \(0<s\le r_{\rm dp}\), \(B_{20s}(a)\subset B_2(0)\), and suppose
    \[
        \sup_{z\in\mathcal F\cap B_{2s}(a)}
        \widetilde W_{z,\mathbf{u}}(8s)\le E\le E_0.
    \]
    Set
    \[
        \mathcal S:=
        \Gamma^k_{\varepsilon,\eta s,20s}(\mathbf{u})\cap B_s(a),
    \]
    and 
    \[
        \mathcal H:=
        \{y\in\mathcal F\cap B_s(a):
        \widetilde W_{y,\mathbf{u}}(2\eta s)\ge E-\eta\}.
    \]
    If \(y_0,\ldots,y_k\in\mathcal H\) satisfy
    \[
        \dist\left(y_i,\operatorname{aff}\{y_0,\ldots,y_{i-1}\}\right)
        \ge\varrho s,
        \qquad i=1,\ldots,k,
    \]
    and \(L:=\operatorname{aff}\{y_0,\ldots,y_k\}\), then
    \begin{equation}\label{eq:EE-density-propagation}
        \widetilde W_{z,\mathbf{u}}(\lambda\varrho s)\ge E-\eta',
    \end{equation}
    for \(z\in\mathcal F\cap B_s(a)\cap N_{\beta_{\rm dp}s}(L)\), and \(\mathcal S\subset N_{\beta_{\rm dp}s}(L)\). Consequently, either
    \[
        \widetilde W_{z,\mathbf{u}}(\lambda\varrho s)\ge E-\eta'
        \quad\text{for every }z\in\mathcal S,
    \]
    or \(\mathcal H\) is contained in the \(\varrho s\)-neighborhood of an affine \((k-1)\)-plane. For \(k=0\), the latter alternative means \(\mathcal H=\varnothing\).
\end{lemma}
\begin{proof}
    Follow \cite[Lemma~3.3]{EE19} and \cite[Lemma~7.6]{McC24}, using the compactness and height bounds recalled above. In a contradiction limit, the near-ceiling hypotheses make the flat minimizer \(\mathbf{v}\) \(p\)-homogeneous about the independent limiting centers \(y_0,\ldots,y_k\), with common Weiss energy \(E\). Their differences generate translation invariance along \(L-y_0\). All centers lie in the limiting tangent plane \(T\), so these translations also preserve the flat weight. The Weiss density is therefore \(E\) at the relevant centers of \(L\), and fixed-radius convergence gives the first assertion of \eqref{eq:EE-density-propagation}.

    If an effective-stratum point remained away from \(L\), the usual blow-up argument in those proofs would give a homogeneous tangent at a limiting contact point outside \(L\), with one additional invariant direction. The inherited height lower bound makes this tangent nonzero. Boundary-normalized convergence would then contradict the \((k+1,\varepsilon)\)-symmetry exclusion, proving the second assertion. The final dichotomy follows by the successive-point selection in \cite[proof of Theorem~3.2]{EE19}; for \(k=0\), selection fails exactly when \(\mathcal H=\varnothing\).
\end{proof}

We now use the centered packing and covering construction of \cite[Lemma~6.1 and Sections~7--8]{EE19}, following the same division of work as \cite[Lemma~7.8 and Section~7.2]{McC24}. The only error summation needed for the Dini weight is recorded below.

\begin{proposition}[Local covering with the corrected density]\label{prop:EE-covering-package-Dini}
    Let \(0\le k\le n-1\) and \(\varepsilon>0\). There are \(0<r_{\rm qs}\le\min\{r_T/160,r_\Sigma/200\}\) and \(C_{\rm pack}<\infty\), with the dependencies specified for the normalized class and also depending on \(k,\varepsilon\), such that the following holds. If \(a\in B_{3/2}(0)\), \(0<s\le r_{\rm qs}\), and \(B_{20s}(a)\subset B_2(0)\), then for every \(0<\rho\le\tau\le s\), the set \(\Gamma^k_{\varepsilon,\rho,20s}(\mathbf{u})\cap B_s(a)\) admits a cover by balls \(\{B_\tau(p)\}_{p\in\mathcal P}\) satisfying
    \begin{equation}\label{eq:local-EE-packing}
        \#\mathcal P\,\tau^k\le C_{\rm pack}s^k.
    \end{equation}
    In particular,
    \begin{equation}\label{eq:local-Hk-effective-stratum}
        \cH^k\bigl(
        \Gamma^k_{\varepsilon,20s}(\mathbf{u})\cap B_s(a)\bigr)
        \le C_{\rm pack}s^k.
    \end{equation}
\end{proposition}
\begin{proof}
    Choose \(r_{\rm qs}\) so that the preceding lemmas apply on the required enlarged balls and \(C_W\mathscr D_\Sigma(100r_{\rm qs})\) is below their thresholds. Finite subdivision reduce nonempty root balls to contact-centered balls. For \(r_j:=2^{-j}r\) and \(0<8r\le r_T\), corrected monotonicity and bounded overlap give, for \(y\in\mathcal F\),
    \begin{equation}\label{eq:QS-Dini-telescoping}
        \begin{split}
            \sum_{j\ge0}\bigl[
                \widetilde W_{y,\mathbf{u}}(8r_j)
                -\widetilde W_{y,\mathbf{u}}(r_j)\bigr]
            &\le3\bigl[
                \widetilde W_{y,\mathbf{u}}(8r)-\vartheta_{\mathbf{u}}(y)\bigr],\\
            \sum_{j\ge0}\mathscr D_\Sigma(r_j,8r_j)
            &\le3\mathscr D_\Sigma(8r).
        \end{split}
    \end{equation}
    Only increments of the Dini correction are summed; the intervals \([r_j,8r_j]\) have overlap at most three, up to endpoints.

    For \(k\ge1\), apply the truncated-measure induction of \cite[Lemma~6.1]{EE19}, with Lemmas~\ref{lem:quantitative-rigidity-Dini} and~\ref{lem:beta-estimate-vectorial} supplying its analytic inputs. At each packing ball's stopping scale, the centered near-ceiling hypothesis bounds the corresponding truncated corrected-density sum by \(C\eta\). Thus \eqref{eq:QS-Dini-telescoping} supplies precisely the summability used in that induction. Lemma~\ref{lem:EE-density-propagation} then supplies the dichotomy for the good- and bad-tree constructions \cite[Theorems~7.4, 7.6, and 7.7]{EE19}. Their resulting cover \cite[Theorem~4.1]{EE19} consists of terminal balls and balls with a fixed decrease in the density ceiling. Iterate on the latter as in the proof of \cite[Theorem~1.11]{EE19}. The uniform range \eqref{eq:density-range-effective-strata-proof} bounds the number of density levels and gives \eqref{eq:local-EE-packing}. Apply the construction to \(\Gamma^k_{\varepsilon,\tau,20s}(\mathbf{u})\), which contains the stated stratum; fixed changes of terminal scale are absorbed by finite subdivision into radius-\(\tau\) balls.

    For \(k=0\), the centered packing step has an elementary replacement. In a disjoint family \(\{B_{\theta r_p}(p)\}\) satisfying the centered near-ceiling hypothesis, with fixed \(\theta>0\) and stopping radii above the effective cutoff, the separation of two centers controls both stopping radii. If their separation were sufficiently small relative to the root radius, Lemma~\ref{lem:quantitative-rigidity-Dini} would give almost homogeneity at both centers at a common scale comparable to their separation and above their stopping radii. Lemma~\ref{lem:cone-splitting-Dini} would then contradict the effective \(0\)-stratum condition. Thus the centers are separated by a fixed fraction of the root radius, and their number is bounded. In the bad alternative, Lemma~\ref{lem:EE-density-propagation} gives \(\mathcal H=\varnothing\). The same bounded density-level iteration proves \eqref{eq:local-EE-packing} also for \(k=0\).

    Finally, apply \eqref{eq:local-EE-packing} to the limiting effective stratum and let \(\tau\to0^+\) to obtain \eqref{eq:local-Hk-effective-stratum}.
\end{proof}


\begin{thebibliography}{99}
    \bibitem[AFT82]{AFT82}C.~J. Amick, L.~E. Fraenkel and J.~F. Toland, On the Stokes conjecture for the wave of extreme form, {\it Acta Math.} {\bf 148} (1982), 193--214.

    \bibitem[AC81]{AC81}H. Alt and L. Caffarelli, Existence and regularity for a minimum problem with free boundary, {\it J. Reine Angew. Math.} {\bf 325} (1981), 105--144.

    \bibitem[AL12]{AL12}D. Arama and G. Leoni, On a variational approach for water waves, {\it Comm. Partial Differential Equations} {\bf 37} (2012), no.~5, 833--874.

    \bibitem[BFS24]{BFS24}M. Bayrami, M. Fotouhi and H. Shahgholian, Lipschitz regularity of a weakly coupled vectorial almost-minimizers for the \(p\)-Laplacian, {\it J. Differential Equations} {\bf 412} (2024), 447--473. 

    \bibitem[BM93]{BM93} G. Buttazzo and G. Dal~Maso, An existence result for a class of shape optimization problems, {\it Arch. Rational Mech. Anal.} {\bf 122} (1993), no.~2, 183--195.
        
    \bibitem[CL08]{CL08}L.~Caffarelli and F. Lin, Singularly perturbed elliptic systems and multi-valued harmonic functions with free boundaries, {\it J. Amer. Math. Soc.} {\bf 21} (2008), no.~3, 847--862. 
        
    \bibitem[CSY18]{CSY18}L.~Caffarelli, H. Shahgholian and K. Yeressian, A minimization problem with free boundary related to a cooperative system, {\it Duke Math. J.} {\bf 167} (2018), no.~10, 1825--1882. 
        
    \bibitem[CTV05]{CTV05}M.~Conti, S.~Terracini and G.~Verzini, Asymptotic estimates for the spatial segregation of competitive systems, {\it Adv. Math.} {\bf 195} (2005), no.~2, 524--560. 

    \bibitem[DS24]{DS24} D. De~Silva and O.~V. Savin, An energy model for harmonic functions with junctions, {\it Adv. Math.} {\bf 447} (2024), Paper No. 109682, 54 pp. 

    \bibitem[DESV21]{DESV21} G.~De Philippis, M.~Engelstein, L.~Spolaor, and B.~Velichkov, Rectifiability and almost everywhere uniqueness of the blow-up for the vectorial Bernoulli free boundaries\arxiv{2107.12485}, 2021.

    \bibitem[DPY25]{DPY25}L.~Du, Y.~Pu and J.~Yang, Singular profile of free boundary of incompressible inviscid fluid with external force, {\it SIAM J. Math. Anal.} {\bf 57} (2025), no.~4, 4275--4313. 

    \bibitem[DSV21]{DSV21} G.~De Philippis, L.~Spolaor, and B.~Velichkov, Regularity of the free boundary for the two-phase Bernoulli problem, {\it Invent. Math.} {\bf 225} (2021), no.~2, 347--394.
    
    \bibitem[DST20]{DST20} D.~De Silva and G.~Tortone, Improvement of flatness for vector valued free boundary problems, {\it Math. Eng.} {\bf 2} (2020), no.~4, 598--613.

	\bibitem[EE19]{EE19} N. Edelen and M. Engelstein, Quantitative stratification for some free-boundary problems, {\it Trans. Amer. Math. Soc.} {\bf 371} (2019), no.~3, 2043--2072.
        
    \bibitem[FV25]{FV25}L. Ferreri and B. Velichkov, A one-sided two phase Bernoulli free boundary problem, {\it J. Math. Pures Appl.} (9) {\bf 195} (2025), Paper No. 103659, 42 pp. 
        
    \bibitem[FS24]{FS24}M. Fotouhi and H. Shahgholian, A minimization problem with free boundary for $p$-Laplacian weakly coupled system, {\it Adv. Nonlinear Anal.} {\bf 13} (2024), no.~1, Paper No. 20230138, 24 pp. 

    \bibitem[KL18]{KL18} D.~Kriventsov and F.~Lin, Regularity for shape optimizers: the nondegenerate case, {\it Comm. Pure Appl. Math.} {\bf 71} (2018), no.~8, 1535--1596.

    \bibitem[KL19]{KL19} D. Kriventsov and F. Lin, Regularity for shape optimizers: the degenerate case, {\it Comm. Pure Appl. Math.} {\bf 72} (2019), no.~8, 1678--1721. 

    \bibitem[MTV17]{MTV17} D.~Mazzoleni, S.~Terracini and B.~Velichkov, Regularity of the optimal sets for some spectral functionals, {\it Geom. Funct. Anal.} {\bf 27} (2017), no.~2, 373--426.
        
    \bibitem[MTV20]{MTV20}D.~Mazzoleni, S.~Terracini and B.~Velichkov, Regularity of the free boundary for the vectorial Bernoulli problem, {\it Anal. PDE} {\bf 13} (2020), no.~3, 741--764.

	\bibitem[McC24]{McC24}S. McCurdy, One-phase free-boundary problems with degeneracy, {\it Calc. Var. Partial Differential Equations} {\bf 63} (2024), no.~1, Paper No. 10, 39 pp.

	\bibitem[NV17]{NV17}A. Naber and D. Valtorta, Rectifiable-Reifenberg and the regularity of stationary and minimizing harmonic maps, {\it Ann. of Math.} (2) {\bf 185} (2017), no.~1, 131--227. 

    \bibitem[SV19]{SV19} L.~Spolaor and B.~Velichkov, An epiperimetric inequality for the regularity of some free boundary problems: the \(2\)-dimensional case, {\it Comm. Pure Appl. Math.} {\bf 72} (2019), no.~2, 375--421.

    \bibitem[SV26a]{SV26a} G.~Siclari and B.~Velichkov, On the blow-up of the vectorial Bernoulli free boundary problem\arxiv{2602.00741}, 2026.

    \bibitem[SV26b]{SV26b} G.~Siclari and B.~Velichkov, On a multiphase vectorial Bernoulli free boundary problem\arxiv{2605.19757}, 2026.
        
    \bibitem[STV20]{STV20}L.~Spolaor, B.~Trey and B.~Velichkov, Free boundary regularity for a multiphase shape optimization problem, {\it Comm. Partial Differential Equations} {\bf 45} (2020), no.~2, 77--108. 
        
    \bibitem[Sto80]{Sto80}G.~Stokes, {\it Mathematical and physical papers. Volume 1}, Reprint of the 1880 original, Cambridge Library Collection, Cambridge Univ. Press, Cambridge, 2009. 

    \bibitem[TV26]{TV26} G. Tortone and B. Velichkov, Vectorial Bernoulli problems and free boundary systems, {\it Matematica} {\bf 5} (2026), no.~1, Paper No. 16, 51 pp. 

	\bibitem[Vel23]{Vel23} B. Velichkov, {\it Regularity of the one-phase free boundaries}, Lecture Notes of the Unione Matematica Italiana, 28, Springer, Cham, 2023. 
        
    \bibitem[VW11]{VW11} E.~V\u{a}rv\u{ a}ruc\u{a} and G.~Weiss, A geometric approach to generalized Stokes conjectures, {\it Acta Math.} {\bf 206} (2011), no.~2, 363--403.

    \bibitem[VW12]{VW12} E.~V\u{a}rv\u{ a}ruc\u{a} and G.~Weiss, The Stokes conjecture for waves with vorticity, Ann. Inst. H. Poincar\'{e} C Anal. Non Lin\'{e}aire {\bf 29} (2012), no.~6, 861--885. 

    \bibitem[VW14]{VW14}E.~V\u{a}rv\u{ a}ruc\u{a} and G.~Weiss, Singularities of steady axisymmetric free surface flows with gravity, {\it Comm. Pure Appl. Math.} {\bf 67} (2014), no.~8, 1263--1306. 
\end{thebibliography}
\end{document}